%% file: DX-QX.tex
\documentclass[a4paper]{amsart}
\usepackage[english]{babel}
\usepackage[T1]{fontenc}
\usepackage[latin1]{inputenc}
\usepackage{amssymb}
\usepackage{amsmath}
\usepackage{amsthm}
\usepackage{amscd}
\usepackage{newlfont}
\usepackage{xcolor}
\usepackage[all]{xy}
\usepackage{float}
\usepackage{mathrsfs}
\usepackage{verbatim}
\usepackage{bm}
\usepackage{tikz-cd}
\usepackage{pgf,tikz}
\usepackage{mathrsfs}
\usepackage{float}
\usepackage{mathtools}
\usetikzlibrary{arrows}
\usepackage{pdfpages}
\usepackage{adjustbox}

\usepackage{xcolor}
\usepackage[colorlinks=true,allbordercolors=white, citecolor=blue]{hyperref}

\DeclareMathOperator{\id}{id}
\DeclareMathOperator{\im}{im}

\DeclareMathOperator{\sgn}{sgn}

\DeclareMathOperator{\tr}{tr}

\DeclareMathOperator{\Exp}{Exp}

\newtheorem{theorem}{Theorem}[section]
\newtheorem{proposition}[theorem]{Proposition}
\newtheorem{corollary}[theorem]{Corollary}
\newtheorem{lemma}[theorem]{Lemma}
\theoremstyle{definition}
\newtheorem{definition}[theorem]{Definition}

\theoremstyle{remark}
\newtheorem{remark}[theorem]{Remark}

\newcommand{\Rmod}{\ensuremath{R\textit{-mod}}}

\newcommand{\Rcoalg}{\ensuremath{R\textit{-coalg}}}

\newcommand{\dip}[1]{^{ [{#1}] }}

\newcommand{\ul}[1]{\underline{#1}}

\newcommand{\R}{{\mathbb R}}

\newcommand{\N}{{\mathbb N}}

\newcommand{\si}{{\mathcal S}}

\newcommand{\dph}{\overline{\mathcal{H}_D}}

\newcommand{\F}{{\mathbb F}}
\newcommand{\dpha}{DP_{HA}}

\newcommand{\hdp}{{H_{EP}}^*}

\newcommand\Item[1][]{%
  \ifx\relax#1\relax  \item \else \item[#1] \fi
  \abovedisplayskip=0pt\abovedisplayshortskip=0pt~\vspace*{-\baselineskip}}

\input colordvi                       %  I like red and blue
\newcommand{\AComment}[1]{ \color{red} { #1} \color{black}}

\newcommand{\BComment}[1]{ \color{blue} { #1} \color{black}}
\newcommand{\CComment}[1]{ \color{green} { #1} \color{black}}

\begin{document}

\author[L. Guerra]{Lorenzo Guerra }
\address{Universit\`a di Roma Tor Vergata}
\email{guerra@mat.uniroma2.it}
\author[P. Salvatore]{Paolo Salvatore}
\address{Universit\`a di Roma Tor Vergata}
\email{salvator@mat.uniroma2.it}
\author[D. Sinha]{Dev Sinha}
\address{Mathematics Department, 
University of Oregon}
\email{dps@uoregon.edu}

\begin{abstract}
We  calculate mod-$p$ cohomology  of extended powers, and their group completions which are free infinite loop spaces.
%as algebras over the Steenrod algebra.  
We consider the cohomology of 
all extended powers of a space together and identify a Hopf ring structure with divided powers %added
 within which cup product structure is more readily
computable than on its own.
We build on our previous calculations of cohomology of symmetric groups, which are the cohomology of extended powers of a point, 
 the well-known calculation of homology, and new results on cohomology of symmetric groups with coefficients in the sign representation.
 We then use this framework to understand cohomology rings of related spaces such as infinite extended powers and free infinite loop spaces.
\end{abstract}
\thanks{The first two authors acknowledge the 
MIUR Excellence Department Project awarded to the Department of Mathematics, University of Rome Tor Vergata, CUP E83C18000100006. 
} 

\subjclass{20J06, 20B30}
\title{Cohomology rings of extended powers and of free infinite loop spaces}
\maketitle

\input{DX-QXSection1} % Introduction
\tableofcontents
\input{DX-QXSection2} % Algebra of divided powers
\input{DX-QXSection3} % Divided powers operations through topology
%\input{DX-QXSection4} 
\input{DX-QXSection5}
\input{DX-QXSection6}

\bibliographystyle{alpha}
\bibliography{bibliografia}

\end{document}

\bibliographystyle{amsplain}
\bibliography{altgroups}

\end{document}

%% file: DX-QXSection1.tex
% !TEX root = DX-QX.tex

\section{Introduction}

Homotopy orbit spaces  with respect to the symmetric group action on iterated products, known as extended powers, 
play many key roles in algebraic topology.  Steenrod 
introduced them as a central character in the of study cohomology operations \cite{Steenrod:53}.  
Extended powers of spectra play an essential role
in Nishida's proof of his Nilpotence Theorem \cite{Nishida:75}.   By the Barratt-Priddy-Quillen
theorem \cite{BP:72}, after group completion the  infinite extended powers functor provides a 
model for the free infinite loopspace functor $Q$.  May and his collaborators demonstrated that the stable
version of extended powers not only were essential in defining structured ring spectra but also provided access to a wide range
of phenomena which had  previously been ad-hoc \cite{BMMS:86}.  
In addition to necessarily playing a role in derived algebra, modern applications of extended
powers and free infinite loop spaces range widely
from, for example, the calculus of functors \cite{AM:99} to the stable cohomology of mapping class groups \cite{Galatius:04}.

We focus on the cohomology rings of these spaces, as this ring structure is the most basic unstable homotopical structure.  
These rings have long been implicitly understood through
homology and coproduct calculations \cite{May-Cohen}. 
But  calculations with such require application of Adem relations for Kudo-Araki-Dyer-Lashof operations, 
so are difficult to work with.  
Here, as in our work on symmetric groups and alternating groups \cite{Guerra:17, Sinha:12, Giusti-Sinha:17}, we find that
 cohomology ring structure 
is best understood when 
coupled with a transfer (or induction) product, first  defined by Strickland and Turner \cite{Strickland-Turner:97}.  Our techniques
in this setting for example have led to further development of the Curtis-Wellington spectral sequence \cite{Hunter23}.

Our work is thus parallel to as well as building on the well-established homology calculations.  In \cite{May-Cohen}, the homology
of extended powers all together forms an algebra over the Dyer-Lashof algebra, which behaves similarly to and in fact can
be related to the cohomology rings of spaces over the Steenrod Algebra  \cite{May:70}.  
Our presentations capture  cohomology ring structure as part of a larger structure as well, ultimately presenting cohomology as
``universal component super-Hopf rigs with additive divided powers'' which then have small, if any, sets of relations.  
Prior to the elaboration of such Hopf ring structures, 
 calculations of mod-two cohomology rings of classifying spaces of symmetric groups -- the extended powers of a point --
 had cumbersome descriptions \cite{Adem-Milgram, Feshbach}, 
and calculations were not made at all for odd primes.

We develop these needed algebraic notions in Section~\ref{sec:defs}, and we encourage the reader to take a look at the statement of
first main results, namely Theorems~\ref{m1} through \ref{m4} to understand the goal of the algebraic work before going
through the details of that section.  Such statements are remarkably simple, given the complexity for example of individual component
rings.  In the odd primary setting, such brevity is not possible without use of trivial and twisted coefficients together.
 We follow these statements with a development of additive bases in their algebraic frameworks.
After extending well-known results about homology to the twisted coefficient setting, as needed for our cleanest statements,
we prove the theorems about ground rings, namely Theorems~\ref{m1}~and~\ref{m2}, in Section~\ref{sec:symmetric group}.
We then use those to prove our first
main theorems, namely Theorems~\ref{m3}~and~\ref{m4}, in Section~\ref{sec:DXdphr}.  Our strategy is to use the two products and divided
powers structures to produce a set of classes in cohomology which pairs perfectly with the standard  basis in homology.
This perfect pairing both establishes the set of cohomology classes as a basis and validates the algebraic presentation which governs them.

%From the group cohomology perspective, our current work extends this trivial coefficients case to
%cohomology of symmetric groups  with coefficients in tensor powers.  A key coefficient module is that of the sign representation, and
%we treat it as well as recast previous work with trivial coefficients in Section~\ref{sec:symmetric group}.  We then make our main calculations,
%additively and then with all of the multiplicative structures developed, in Sections~\ref{sec:algebraic basis}~and~\ref{sec:DXdphr}. 
% Our main result is a relatively succinct description of the cohomology of extended powers of $X$, with 
%coefficients in the direct sum of the trivial and sign representations, as a Hopf ring.  
%As when $X$ is a point, the cohomology of individual extended powers remains complicated to describe on their own,
%with only the cup product.   
%To outline our work, we present statements of main theorems without giving most definitions, which will be given in Section~\ref{sec:defs}.

After treating the cohomology of extended powers with disjoint basepoints, we calculate
the cohomology of $D X$, $D_\infty X$, $CX$ (defined below) and $QX$ in Section~\ref{CXQX}.  For $CX$ and $QX$
these reproduce Dung's description of the cohomology ring for $QX$ when $p=2$ and givie new results for odd primes. 
%Finally, in Section~\ref{steenrod}, we address Steenrod operations, 
%yielding an explicit presentation of the cohomology of extended powers and free infinite loop spaces as unstable algebras over the Steenrod algebra.

While we capitalize on finding the right algebraic structures and rely on standard techniques from algebraic topology and group
cohomology as well as previous homology calculations, geometry has inspired and guided us.  
We focus on $C(X)$, the free $E_\infty$-space on $X$.  We choose  $ E\si_n $ to 
be the (ordered) configuration space of $ n $ points in $ \mathbb{R}^\infty $
\[
E\si_n = \{ (x_1, \dots, x_n) \in (\mathbb{R}^\infty)^n : x_i \not= x_j \forall i \not= j \},
\]
with an action of $ \si_n $ given by permutation of coordinates.

\begin{definition}
Let $ X $ be a pointed topological space. Let $ \phi_{i,n} \colon E\si_n \to E\si_{n-1} $ the map that forgets the $ i^{th} $ point of the configuration. 
Let $ s_{i,n} \colon X^{n-1} \to X^n $ be the map that adds the basepoint as the $ i^{th} $ coordinate.
Define $ CX$, the {\bf space of configurations in $\R^\infty$ with labels in $X$}, as 
${\bigsqcup_{n \geq 0} E\si_n \times_{\si_n} X^n} / {\sim} $, 
where $ \sim $ is the equivalence relation generated by $ (\phi_{i,n} a,b) \sim (a, s_{i,n}b) $ for all $ a \in E\si_n $ and $ b \in X^{n-1} $.
\end{definition}

We consider the finite-dimensional versions of this model, with a limited number of
points in some $\R^N \subset \R^\infty$, in which case if $X$ is a manifold then so is this quotient.  In that setting, we can use
geometric chains and cochains \cite{FMS22}.  The homology story is well-known through Kudo-Araki-Dyer-Lashof operations,
which are manifestly geometric.  If $f : P \to X$ represents a mod-two geometric homology class $y$ then we
can map $S^i \times_{C_2} (P \times P)$ to $CX$ by sending $v \times p_1 \times p_2$ to the configuration of two points at $v , -v$ with labels
$f(p_1)$ and $f(p_2)$.  The resulting geometric homology class represents the operation $q_i(y)$.  On the cohomology 
side, the analogous geometric cochain would be the submanifold of $CX$ where two points in the underlying configuration share their first coordinate
with labels both in some geometric cochain $W \to X$.   This geometry guided us in formulating our main structures, with cup product
corresponding to intersection as usual, transfer product a type of ``union'', and divided powers a ``repeating''.

We see many possible investigations building on this work.  
Following Remark~\ref{bad-example}, it would be interesting and
probably reflective of deeper structure to fully understand divided powers operations on cohomology, as we only employ them on a limited set of classes.  
The pairing between homology and cohomology is not Kronecker, and is likely to be
useful since the two settings differ in which structures are most readily expressed.  
We have made partial progress on computing Steenrod operations, a natural ``self-serving'' calculation to consider.
Calculating cup product more organically as we have done here could help in extending these calculations to generalized cohomology 
theories, with even understanding whether divided powers operations exist for the transfer product an interesting first question.  We suspect
not, as the transfer product is analogous to the Pontrjagin product on the homology of Eilenberg-MacLane spaces, which has a divided
powers structure \cite{Cartan:55, HHPPST23}, but the homology of other infinite loop spaces instead have Dyer-Lashof operations.
Related to this, there is possibility of further binding these structures through higher algebra enhancements of transfer product, 
in which case all of this intricate structure could
likely be tied back to the homology of $\R P^1$ and $\mathbb{C} P^1$.  
%\pagebreak

%%% Local Variables:
%%% TeX-master: "DX-QX.tex"
%%% End:

%% file: DX-QXSection2.tex
% !TEX root = DX-QX.tex

\section{Hopf ring with additive divided powers structure on the cohomology of $DX$} \label{sec:defs}

%\CComment{added this -PS}
In this section we introduce the algebraic structure of component Hopf ring with additive divided powers.
We show in  Corollary~\ref{divided-powers-representations} that this structure applies to the cohomology of extended powers,
and then give an explicit presentation of this cohomology  in Theorem \ref{m2}. 

\subsection{Preliminaries}

%With many algebraic structures to develop we gather here names of the categories of such structures, for convenience of reference.

%\begin{itemize}
%\item  $\Rmod $ is the category of $ R $-modules.
%\item  $ \Ralg $ is the category of associative and commutative $ R $-algebras.
%\item  $ \Rdpalg $ is the category of  $ R $-algebras with divided powers structure.
%\item  $ \Rcoalg $ is the category of coassociative and cocommutative graded connected $ R $-coalgebras.
%\item   $ \Rhalg $ is the category of biassociative bicommutative graded connected $ R $-Hopf algebras. 
%\item $ \Rdphalg $ is the category of biassociative bicommutative graded connected $ R $-Hopf algebras with divided powers.
%\item $ \Rbcalg $ and $ \Rbcbialg $ are the categories of bigraded component algebras and bigraded component bialgebras, 
%respectively. %that satisfy the  connectedness condition $A = \im ( \eta ) \bigoplus \overline{A}$,
%where $ \eta \colon R \to A $ is the image of the unit map.
%\item  $ \Rbchrng $ the category of connected bigraded component Hopf rings over $ R $.
%\item $ \Rdphr $ the category of Hopf rings with additive divided powers, as defined below.
%\end{itemize}

Extended powers are our first main objects of study, and we assume throughout that 
$ X $ is Hausdorff and compactly generated, and that any basepoint has  a neighborhood deformation retract.

\begin{definition}\label{def:DXdphr}
If spaces $E$ and $Y$ are both $ \si_n $-spaces, with action preserving the base point of $Y$, then
$ E \rtimes_{\si_n} Y $ is defined as the quotient $\left( (E)_+ \wedge Y \right)/{ \si_n }$, where $ (E)_+ $ is  $E$ 
with a disjoint basepoint and $ \si_n $ acts diagonally.

 Define the {\bf $ n^{th} $ extended power} of $ X $ by
$ D_n X = E  \si_n  \rtimes_{\si_n} X^{\wedge^n} $.   Let $DX = \bigvee_n (D_n X)$.
Recall  that if $X$ is a $G$-space we let $X_{hG}$ denote the homotopy quotient $EG \times_G X$, so that
${X^{n}}_{h \si_n} = D_n(X_+)$. We will also consider this unbased version of $ D_n $, for which we use the symbol $ \tilde{D} $ 
to differentiate the two. Hence $ \tilde{D}(X) = X_{h \si_n} $.

For $p=2$ let $\hdp(X): = \bigoplus_{n \geq 0} H^*(D_n X; \mathbb{F}_2)  $.  

For odd primes let 
$\hdp(X) := \mathbb{F}_p \oplus \bigoplus_{n \geq 1} H^*(D_n X; \mathbb{F}_p \oplus sgn) $, where $\mathbb{F}_p$ denotes the trivial representation and $sgn$ is
the sign representation.  These are bi- or tri-graded by component, degree, and in the odd prime setting a mod-two grading induced by an obvious
grading on $\mathbb{F}_p \oplus sgn$.
%We bi-grade, first by $n$ and then by cohomological degree.  
\end{definition}

Observe that $ \tilde{D}(X) \cong D(X_+) \cong C(X_+) $. %modified

We use both algebraic and geometric variants of the extended powers topological construction, using the algebra 
to provide the main framework for our arguments.
Almost by definition, cohomology of $D_n X$ is the equivariant cohomology of the $n$-fold tensor product
of the cochains of $X$.  But the following basic computation shows one can consider only the cohomology of $X$
(see for example  \cite{May:70}, Lemma 1.1).

\begin{proposition} \label{prop:may}
There are natural isomorphisms $ H^*(D_n X; \mathbb{F}_p) \cong H^* ( \si_{n}; {\widetilde{H}^*(X; \mathbb{F}_p))}^{\otimes{n}} ) $,
and similarly for homology, compatible with standard pairings.
\end{proposition}

Our main results present the cohomology rings, 
with $ \mathbb{F}_p $ coefficients, of these
three related spaces, namely 
$ DX $, the extended  power of a topological space $ X $, $CX$,  the free $ E_\infty $-space over 
$ X $,  and $ QX $,   the free infinite-loop space over $ X $.  Because the homology of these spaces, as co-algebras, is known and
only depends on the 
homology of $X$ \cite{May-Cohen}, our first step is to develop the algebraic functors which will take as their
input the cohomology ring of $X$ and produce the cohomology rings of these spaces.   

\subsection{Component Hopf rings with additive divided powers} \label{algdef}

We work primarily in a bigraded setting  
in a strong sense, namely products vanish on elements which differ in the first grading.

\begin{definition} \label{def:bigraded component bialgebra}
Let $ (H, \Delta, \cdot) $ be a bialgebra over $ R $. We say that $ H $ is a {\bf bigraded component bialgebra} 
if  $ H = \bigoplus_{n,d \in \mathbb{N} \times \N} H_{n,d} $, where 
\begin{itemize}
\item  the product sends  $ H_{m,i} \otimes H_{m,j}$ to $H_{m, i+j}$ 
\item the product of $x$ and $y$ is zero if  $ x \in H_{m,i}$, $y \in H_{n,j}$ with $ n \not= m $
\item   the coproduct is standardly bigraded, sending $H_{p,k}$ to $\bigoplus_{n+m=p, k=i+j}  H_{n,i} \otimes H_{m,j}$.
\end{itemize}

The first grading of an element is called its {\bf component}, while the second  is called its {\bf dimension}.
For such $ H $, we let $ \overline{H}$ denote the direct sum of $ H_{n,d} $  for $n>0$ (and all $ d \geq 0$).
We  require the connectedness condition $A = \im ( \eta ) \bigoplus \overline{A}$,
where $ \eta \colon R \to A $ is the image of the unit map.

Similarly, {\bf super-component bialgebra} is a bialgebra graded over $ \mathbb{N} \times \frac{\mathbb{Z}}{2\mathbb{Z}} $ satisfying these same
axioms. %I put N instead of Z
%We denote by $ \Rbcalg $ and $ \Rbcbialg $ the categories of bigraded component algebras and bigraded component bialgebras, respectively,
%that satisfy the  
%The usefulness of this requirement will be apparent later.
\end{definition}

Recall that a rig
(ring without negatives) satisfies all ring axioms except the existence of additive inverses.

\begin{definition}\label{def:bigraded component Hopf rig}
A {\bf Hopf rig} $H$ is a rig object in the category of coalgebras, so that there are two products, a ``multiplication product'' $\cdot$ and an ``addition
product'' $\odot$, 
both giving bialgebras with one coproduct $\Delta$, 
and a distributivity axiom that $a \cdot (b \odot c) = \sum_{\Delta a = \sum a' \otimes a''} (a' \cdot b) \odot (a'' \cdot c)$.

A {\bf bigraded component Hopf rig} (respectively {\bf super-component Hopf rig})  is one for which $(H, \cdot, \Delta)$ is a 
bigraded component bialgebra (respectively  super-component bialgebra), and $\odot$ preserves both gradings. In this setting, we modify the distributivity axiom above by introducing a coefficient $ (-1)^{|a''||b|} $ corresponding to the usual Koszul sign.

We require the structural morphisms to be graded (co)commutative with respect to the dimension. When this commutativity condition is not necessarily satisfied, we speak of a {\bf non-commutative bigraded component Hopf rig}.
\end{definition}

The Hopf rigs we study have antipodes, and are thus Hopf rings.   But we will not need the antipode structures, and only
mention them in passing.
Hopf rings first appeared in topology in the study of homology of infinite loop spaces which represent cohomology with product structure
\cite{Milgram:72, RaWi80, Wilson:00}.  But following Strickland and Turner \cite{Strickland-Turner:97}, we have found them essential in
describing the cohomology of symmetric groups.  An interesting related case of Hopf rings arises when one introduces an induction or 
symmetrization product on rings of symmetric polynomials.  Readers who would like to see concrete examples can see these treated
in the second section of \cite{Sinha:12}.

We will show that the cohomology of $ DX $ not only is a Hopf rig, but is also endowed with divided powers operations. 
\begin{definition} \label{def:divided powers}
A {\bf divided powers algebra} is a triple $ (A,I,\gamma)$, where $ A $ is an algebra,  %is field ok ??
$ I \trianglelefteq A $ is a proper ideal, 
and $ \gamma = \left\{ \gamma_n \right\}_{n \in \mathbb{N}} $ is a family of functions from $I$ to $I$, 
also denoted  $x\dip{n}:=\gamma_n(x)$,  which satisfies the following relations, whenever $x,y \in I$ and $\lambda \in A$:
\begin{enumerate}
\item $ x\dip{0}= 1 $, and $x\dip{1} = x$  \hspace{2.5cm}  ({\bf $0,1$ Cases})
\item $  \left( x+y \right)\dip{n} = \sum_{i=0}^n x\dip{i} y\dip{n-i} $ \hspace{1.6cm}  ({\bf Binomial})   \label{r2} 
\item $ \left( \lambda x \right)\dip{n} = \lambda^n x\dip{n} $   \hspace{3.3 cm}  ({\bf Distributivity over Multiplication})\label{r3} 
\item $ x\dip{n} x\dip{m} = \left( \begin{array}{c} n+m \\ n \end{array} \right) x\dip{n+m} $  \hspace{1.2cm}  ({\bf Law of Exponents})
\item $ (x\dip{m})\dip{n} = \frac{ (nm)! }{ (m!)^n n! } x\dip{nm}$.   \hspace{2cm}  ({\bf Composition}) \label{r5}

\end{enumerate}
\end{definition}

We also  refer to such triples as {\bf divided powers structures}, on the algebra $A$.  Introduced in  \cite{Cartan:55}, these 
have been extensively studied and applied to various contexts, for example in \cite{Roby:65}, from which we borrow  notation, 
and \cite{Berthelot:74, Hazewinkel:78}. In algebraic topology, the homology of Eilenberg-MacLane spaces are divided powers
algebras \cite{EM:54, Thomas:57}.

 These relations imply that the formal series $ f (t) = \sum_{i=0}^{\infty} x\dip{i} t^i \in A[[t]] $ satisfies 
$ f(0) = 1$ and $ f(t+s) = f(t) f(s) .$
The set $ \Exp(A) $ of  $ f \in A[[t]] $ satisfying these conditions  constitutes an 
$ R $-module. %, with  $(\lambda f)(t) =f(\lambda t)$ and  $(f+g)(t) =f(t)g(t).$
The left adjoint of $ \Exp$ viewed as a functor from $R$-algebras to $R$-modules is the free divided powers functor $DP$, defined explicitly as follows.

\begin{definition}
For an $ R $-module $ M $, we let $ DP ( M ) $, the {\bf free divided powers algebra} generated by $M$,
be  generated as an algebra by the set
$ \left\{ x\dip{n} \right\}$, where $x \in M$, with relations  (1)-(5) from Definition~\ref{def:divided powers} imposed.
The ideal $I$ is given by the collection of $x\dip{n}$ with $n \geq 1$ and divided powers maps are defined 
through the Binomial, Scalar Multiplication and Composition relations 
 of Definition~\ref{def:divided powers}. \end{definition}

%\begin{proposition}
%$DP$ is a functor which is left adjoint to the forgetful functor from divided powers algebras to modules.
%\end{proposition}
 
If  $ M $ is graded, then naturally so are  $ DP(M)$ and  the universal map.
Over a field, if $M$ is finite-dimensional, there is a standard additive basis of $ DP(M) $ which we can associate to a choice of basis 
$ \left\{ x_1, \dots, x_r \right\} $ for $ M $, namely 
$ \left\{ \prod_{i=1}^r x_{i [n_i]} \right\}_{n_i \in \mathbb{N}}$.

\begin{definition} \label{Defdpbialg}
For a divided powers algebra $ (A,I,\gamma) $, the {\bf tensor product} $ A \otimes A $ has a natural
divided powers structure $ ( A \otimes A, A \otimes I + I \otimes A, \gamma_{\otimes} ) $ defined for  $ x \in I $ and $ y \in A $  by
$ ( x \otimes y )\dip{n} = x\dip{n} \otimes y^n $ and $ (y \otimes x )\dip{n} = y^n \otimes x\dip{n} $. 
%  extending to sums of pure tensors by Relation (2).
%We use this structure for developing  bialgebras.  

A {\bf divided powers  bialgebra} is a divided powers structure $ (H, \overline{H}, \gamma ) $ on a bialgebra $ H $,
with $\overline{H}$ kernel of the counit,  
such that the coproduct $ \Delta \colon H \to H \otimes H $ is a morphism of divided powers structures.
\end{definition}

In the graded setting, we require divided powers operations to be compatible with degrees, in the sense that $ \deg(x\dip{k}) = k\deg(x) $.
In characteristic not equal to two, elements of odd degree have zero squares and thus higher powers,  so in this setting
$ x\dip{k} = 0 $ if $ \deg(x) $ is odd and $ k \geq 2 $.  This is also the case for the odd degree classes in the super-bigraded component setting.
By a classical theorem of Milnor and Moore, a graded Hopf algebra $ H $ that is bicommutative in the graded sense can always be split into an even and odd part %$ H = H_{even} \otimes H_{odd} $, 
where the odd factor %was "summand"
 is an exterior algebra generated by primitive elements in odd degree. %, and $ H_{even} $ is concentrated in even degrees. 
A graded divided powers structure on $ H $ is thus equivalent to a divided powers structure in the standard sense on the even part, and on the odd part the divided powers with exponent larger than one are zero. %was: in degrees greater than or equal to two on the odd part. 
%As the parity of the degree of $ x\dip{k} $ depend only on $ k $ and the parity of the degree of $ x $, a notion of divided powers structure on super-Hopf algebras is defined the same way.

Admitting a divided powers structure is a very restrictive condition on a Hopf algebra.
Such algebras arise as dual Hopf algebras to free, primitively generated Hopf algebras or more generally 
as duals of enveloping algebras of Lie algebras  \cite{Schoeller:67, Andre:71}. 

Given a counital $R$-coalgebra  $ C$, let $ \varepsilon \colon C \to R $ be 
the counit and  $ \overline{C} = \ker(\varepsilon) $.
The coproduct $ \Delta \colon C \rightarrow C \otimes C $ extends uniquely to
$ DP( \overline{C} ) \to DP( \overline{C} ) \otimes DP( \overline{C} )$, 
which  by abuse of notation we still denote  $ \Delta $.
This is a coassociative and cocommutative coproduct, and defines a bialgebra structure on $ DP( \overline{C} ) $.

\begin{definition}
Denote by $ \dpha$ the {\bf free divided powers Hopf algebra} functor from $R$-coalgebras to Hopf algebras with divided powers over $R$,
with $ \dpha ( C ) = DP ( \overline{C} ) $ as an algebra,  the coproduct induced by that of $ C $, and extended to morphisms
by the universal property. 

 %Let $ u \colon C \to \dpha(C) $, denote the canonical universal map.
 %, defined on $ \overline{C} $ as for $ DP $ and
%mapping $ 1 \in C $ to $ 1 \in \dpha(C) $
\end{definition}

Then  $ \dpha(C) $ is left adjoint to the forgetful functor from divided powers  bialgebras to coalgebras.

Our main structure we use comprises the two structures detailed above.

\begin{definition}\label{def:dphr}
A {\bf bigraded component Hopf ring with additive divided powers} is a septad $ (A,\odot,\cdot,\Delta,\eta,\varepsilon,\{\gamma_n\}) $ such that
\begin{itemize}
\item $ (A,\odot,\Delta,\{\gamma_n\}) $ is a divided powers bigraded component Hopf algebra
\item $ (A,\odot,\cdot,\Delta) $ is a bigraded component Hopf rig
\end{itemize}
We also require that each component algebra $ (A_n,\cdot) $ is unital, and that the unit of $ A_n $ is the $ n^{th} $ divided power of the unit of $ A_1 $.

A {\bf component super-Hopf ring with additive divided powers} is a septad as above where $ A $ has an additional $ \mathbb{Z}/2\mathbb{Z} $-grading $ A = A_{even} \oplus A_{odd} $, preserved by $ \Delta $ and $ \cdot $ and such that $ A_{even} \odot A_{odd} = 0 $, satisfying the following conditions:
\begin{itemize}
	\item $ (A,\odot,\Delta,\{\gamma_n\}) $ is a divided powers Hopf algebra
	\item $ (A, \cdot, \Delta) $ is a component super-algebra
	\item the Hopf ring distributivity axiom holds in $ A $
\end{itemize}
\end{definition}

 We will prove below that the Hopf rings %was: rigs
 of our interest have divided powers operations 
for the addition product satisfying Definition \ref{def:dphr}. However, they are not compatible with the multiplication product $ \cdot $ in a simple way, as we show in Remark~\ref{bad-example}.

\subsection{The algebraic extended powers functor}

Recall that our first main goal is to describe the algebraic functor which takes the cohomology ring of a space $X$ and produces the cohomology
of its extended powers.  In light of Proposition~\ref{prop:may}, this can be viewed in terms of group cohomology.
For optimal results, at odd primes we take cohomology with coefficients which incorporate  sign representation.

\begin{definition}
Let $V$ be a representation of $\si_n$ over a field.  Define the {\bf cohomology of extended powers of a space $X$ with coefficients in $V$}, 
denoted $H^*(D_n X; V)$, to be the group cohomology
$H^* ( \si_{n}; {\widetilde{H}^*(X)}^{\otimes{n}} \otimes V)$.
\end{definition}

In our application $V$ will be either trivial when $p = 2$ or for odd primes will incorporate the sign representation.
Clearly the notation is consistent in the case of trivial coefficients by Proposition \ref{prop:may}. %added

\begin{definition} \label{def:rho}
By abuse, for any odd prime $p$ and any $n$ let $\rho$ be the representation of $\si_n$ given by $\rho =  \mathbb{F}_p \oplus sgn$,
where $ \mathbb{F}_p$ denotes the trivial representation and $sgn$ denotes the sign representation.  

 We grade $\rho$ by having 
the trivial representation in degree zero (even) and the sign representation in degree one (odd).  We use canonical isomorphisms $\mathbb{F}_p 
\otimes V \cong V$ and $sgn \otimes sgn \cong \mathbb{F}_p$ to define a graded  algebra structure on $\rho$.
\end{definition}

In our setting, the coefficients we are using have multiplicative structure which is then reflected in cohomology.

\begin{definition}
A {\bf product series of algebras} over a field $k$ is a collection  $\{A_n\}$ with $A_n$ $k[\si_n]$-algebra, with isomorphisms of $ k[\si_i \times \si_j] $-algebras
$\chi_{i,j} \colon A_{i+j} \to A_i \otimes_k A_j $ that are coherent in the sense that the following two conditions are satisfied:
\begin{enumerate}
\item for all $ i,j,k > 0 $ the following diagram commutes in the category of $ k[\si_i \times \si_j \times \si_k] $-modules:
\begin{center}
	\begin{tikzcd}
		A_{i+j+k} \arrow{r}{\chi_{i,j+k}} \arrow{d}{\chi_{i+j,k}} & A_i \otimes A_{j+k} \arrow{d}{\id \otimes \chi_{j,k}} \\
		A_{i+j} \otimes A_k \arrow{r}{\chi_{i,j} \otimes \id} & A_i \otimes A_j \otimes A_k
	\end{tikzcd}
\end{center}
\item for all $ n,m > 0 $, the following diagram commutes
\begin{center}
\begin{tikzcd}
	A_{n+m} \arrow{r}{\sigma_{n,m}} \arrow{d}{\chi_{n,m}} & A_{n+m} \arrow{d}{\chi_{m,n}} \\
	A_n \otimes A_m \arrow{r}{\tau} & A_m \otimes A_n,
\end{tikzcd}
\end{center}
where $ \tau $ exchanges the two factors and $ \sigma_{n,m} \in \Sigma_{n+m} $ is the permutation given by
\[
\sigma_{n,m}(i) = \left\{ \begin{array}{ll}
m+i & \mbox{if } 1 \leq i \leq n \\
i-n & \mbox{if } n+1 \leq i \leq n+m
\end{array} \right. . 
\]
\end{enumerate}

A {\bf super-product series of algebras} is a collection $ \{ A_n \} $ with $ A_n $  $k[\si_n] $-module with the following additional structure:
\begin{itemize}
	\item a grading $ A_n = A_{n,0} \oplus A_{n,1} $ of each $ A_n $ over $ \mathbb{Z}/2\mathbb{Z} $,
	\item a product $ A_n \otimes A_n \to A_n $ that makes each $ A_n $ a $ k[\si_n] $-super-algebra,
	\item and $ k[\si_i \times \si_j] $-module isomorphism $\chi_{i,j,e} \colon A_{i+j,e} \to A_{i,e} \otimes_k A_{j,e} $ for all $ i,j \in \mathbb{N}$ and $ e \in \mathbb{Z}/2\mathbb{Z}$, that are coherent in the sense that the two conditions above are satisfied and such that $ \chi_{i,j,0} \oplus \chi_{i,j,1} \colon A_{i,j} \to A_i \otimes_k A_j $ are super-algebra morphisms.
\end{itemize}
\end{definition}
We note that a product series of algebras can be regarded as a super-product series of algebras concentrated in $ \mathbb{Z}/2\mathbb{Z}$-degree $ 0 $.

The product series we use are built from a single algebra, ultimately the cohomology ring of a space.

\begin{definition} \label{def:divpower}
Let $A$ be an algebra.  Define $TA$ to be the product series of algebras with ${TA}_n = A^{\otimes{n}}$ with canonical restrictions.  
For odd primes define $T_\rho A$ to be the sequence of modules with ${T_\rho A}_n = \rho \otimes A^{\otimes n}$ and structure maps
given by the identity map on $\rho$ tensored with canonical restrictions.
\end{definition}

%The following statement is obvious.
%\begin{proposition}
%let $ V $ be an algebra over the field $ k $. If $ \{A_n\}_n $ is a product series of algebras, then so is $ \{V^{\otimes n} \otimes A_n\}_n $.
%\end{proposition}
%
%Our main application will be to $A_n = H^*(X)^{\otimes n} \otimes ()
%\cong H^*(X^n; \mathbb{F}_p \oplus sgn)$.

\begin{definition}\label{main-dphr-structure}
Let $ \{A_n\}$ be a product series of algebras. Let $f \in H^*(\si_n; A_n)$ and $g \in H^*(\si_m; A_m)$,
so $f$ is represented by a homomorphism from a resolution $ W_*^{\si_n} $ of $k$ over $k[\si_n]$  to a suspension 
of $A_n$.
Define the preliminary multiplicative structures on cohomology of symmetric groups with coefficients  in $\{A_n\}$
as follows.
 \begin{itemize}
 \item A coproduct $\tilde{\Delta} \colon H^*(\si_i+j; A_{i+j}) \to H^*(\si_i; A_i) \otimes H^*(\si_j; A_j) $ is induced by the composition
 \begin{center}
 \begin{tikzcd}
 \hom_{\si_{i+j}}(W_*^{\si_{i+j}}; A_{i+j}) \arrow{r}{\rho} & \hom_{\si_i \times \si_j}( W_*^{\si_i} \otimes W_*^{\si_j}; A_{i+j}) \arrow{r}{\chi_{i,j}} & { }\\
 \hom_{\si_i \times \si_j} ( W_*^{\si_i} \otimes W_*^{\si_j}; A_i \otimes A_j) \arrow{r}{\cong} & \hom_{\si_i}(W_*^{\si_i};A_i) \otimes \hom_{\si_j}(W_*^{\si_j};A_j),
 \end{tikzcd}
 \end{center}
where $ \rho $ is the restriction to the resolution of a subgroup;
 \item A transfer product $\tilde{\odot} \colon H^*(\si_n;A_n) \otimes H^*(\si_m; A_m) \to H^*(\si_{n+m}; A_{n+m}) $  is induced by
 \begin{center}
 \begin{tikzcd}
 \hom_{\si_n}(W_*^{\si_n};A_n) \otimes \hom_{\si_m}(W_*^{\si_m};A_m) \arrow{r}{\cong} & \hom_{\si_{n}\times \si_m}(W_*^{\si_n} \otimes W_*^{\si_m}; A_n \otimes A_m) \arrow{r}{\chi_{n,m}^{-1}} & { }\\
 \hom_{\si_n \times \si_m}(W_*^{\si_n} \otimes W_*^{\si_m};A_{n+m}) \arrow{r}{\tr} & \hom_{\si_{n+m}}(W_*^{\si_{n+m}};A_{n+m}),
 \end{tikzcd}
 \end{center}
where $ \tr $ denotes the usual transfer map;
\item a cup product $\cdot$ is defined 
in the standard way on each component, using restriction along the diagonal on a resolution of $k$ over $k[\si_n]$ and the product on $A_n$
\begin{center}
\begin{tikzcd}
\hom_{\si_n}(W_*^{\si_n};A_n) \otimes \hom_{\si_n}(W_*^{\si_n};A_n) \arrow{r}{\cong} & \hom_{\si_n \times \si_n} (W_*^{\si_n} \otimes W_*^{\si_n};A_n \otimes A_n) \arrow{r} & { }\\
\hom_{\si_n}(W_*^{\si_n};A_n \otimes A_n) \arrow{r} & \hom_{\si_n} (W_*^{\si_n}; A_n),
\end{tikzcd}
\end{center}
and
is zero between distinct components;
 %\item Divided powers operations are defined by taking external (tensor) product $f^{\otimes m}$ and taking the transfer from 
 %$H^*(\si_{m \cdot \ul{k}} ; (A_k)^{\otimes m})$
 %to $H^*(\si_{mk}; A_{mk})$.
 \end{itemize}

More generally, if we start with a super-product series of algebras $ \{A_n\} $ with $ A_n = \bigoplus_{e \in \mathbb{Z}/2\mathbb{Z}} A_{n,e} $, we similarly define:
\begin{itemize}
	\item a coproduct $ \tilde{\Delta} $ as the direct sum of the coproducts associated with the addend product series $ A_{*,e} $;
	\item a transfer product that restrict to that defined above on each addend $ A_{*,e} $, and such that $ f \tilde{\odot} g $ is zero if $ f \in H^*(\si_n;A_{n,e}) $ and $ g \in H^*(\si_m;A_{m,e'}) $ with $ e \not= e' $;
	\item if each $ A_n $ is an algebra, a cup product $ \cdot $ as above.
\end{itemize} 
\end{definition}

These structures are preliminary because of some additional signs which we will include below.

\begin{definition}
If $\{A_n\}$ is a product series of algebras or super-product series of algebras, define
$H^*(\si_\bullet; A_\bullet)$ to be the direct sum
 $\mathbb{F}_p \oplus \bigoplus_{n \geq 1} H^*(\si_n; A_n)$, obtained 
  from $ \bigoplus_{n \geq 0} H^*(\si_n; A_n) $  by replacing in component $ 0 $ the representation 
  $ \bigoplus_{e \in \mathbb{Z}/2\mathbb{Z}} A_{0,e} \cong \bigoplus_e \mathbb{F}_p $ via the epimorphism $ \sum_e \colon \bigoplus_e \mathbb{F}_p \to \mathbb{F}_p $, with coproduct $ \tilde{\Delta}$ and product $\tilde{\odot} $ 
  \end{definition}

\begin{theorem}\label{main-section-3-pre}
$H^*(\si_\bullet; A_\bullet)$ forms a non-commutative Hopf algebra,
  with structures defined in Definition~\ref{main-dphr-structure}. 
  Moreover, $\left(H^*(\si_\bullet; A_\bullet), \tilde{\Delta}, \cdot \right) $ is a bigraded component bialgebra.
\end{theorem}
\begin{proof}
This theorem has been essentially proved  by Strickland and Turner \cite{Strickland-Turner:97} for generalized cohomology theories, and in particular for cohomology with coefficients in the trivial representations. Their proof can be reinterpreted diagrammatically in a group-theoretic setting as explained by Giusti--Salvatore--Sinha \cite{Sinha:12}.% \AComment{We should include more proof of some key points (and then make this a Proof. again), maybe using some of what Chad and I did for alternating groups. -DS}
In our case, the modification in component zero is introduced only to ensure that the unit behaves correctly with both the coproduct and the counit. With this exception, those diagrams also prove the statement for coefficients in a general super-product series of algebras, with the exception of those yielding the commutativity of $ \tilde{\odot} $ and the cocommutativity of $ \tilde{\Delta} $, that involve conjugation by elements of the symmetric groups.

The associativity of $ \tilde{\Delta} $ and $ \tilde{\odot} $ is shown using the following commutative diagram of finite coverings:
\begin{center}
	\begin{tikzcd}
		B(\si_n \times \si_m \times \si_l) \arrow{r} \arrow{d} & B(\si_{n+m} \times \si_l) \arrow{d} \\
		B(\si_{n} \times \si_{m+l}) \arrow{r} & B(\si_{n+m+l})
	\end{tikzcd}
\end{center}
When taking cohomology with coefficient in the product series we obtain the following commutative diagram:
\begin{center}
	\adjustbox{scale=0.6}{%
	\begin{tikzcd}
		H^*(\si_{n+m+l};A_{n+m+l}) \arrow{r}{\rho} \arrow{d}{\rho} & H^*(\si_{n+m} \times \si_l; A_{n+m+l}) \arrow{r}{\chi_{n+m,l}} \arrow{d}{\rho} & H^*(\si_{n+m} \times \si_l; A_{n+m} \otimes A_l) \arrow{r}{\cong} \arrow{d}{\rho} & H^*(\si_{n+m};A_{n+m}) \otimes H^*(\si_l; A_l) \arrow{d}{\rho \otimes \id} \\ 
		H^*(\si_n \times \si_{m+l}; A_{n+m+l}) \arrow{r}{\rho} \arrow{d}{\chi_{n,m+l}} & H^*(\si_n \times \si_m \times \si_l; A_{n+m+l}) \arrow{r}{\chi_{n+m,l}} \arrow{d}{\chi_{n,m+l}} & H^*(\si_n \times \si_m \times \si_l; A_{n+m} \otimes A_l) \arrow{r}{\cong} \arrow{d}{\chi_{n,m} \otimes \id} & H^*(\si_n \times \si_m; A_{n+m}) \otimes H^*(\si_l; A_l) \arrow{d}{\chi_{n,m} \otimes \id} \\
		H^*(\si_n \times \si_{m+l}; A_n \otimes A_{m+l}) \arrow{r}{\rho} \arrow{d}{\cong} & H^*(\si_n \times \si_m \times \si_l; A_n \otimes A_{m+l}) \arrow{r}{\id \otimes \chi_{m,l}} \arrow{d}{\cong} & H^*(\si_n \times \si_m \times \si_l; A_n \otimes A_m \otimes A_l) \arrow{r}{\cong} \arrow{d}{\cong} & H^*(\si_n \otimes \si_m; A_n \otimes A_m) \otimes H^*(\si_l; A_l) \arrow{d}{\cong} \\
		H^*(\si_n; A_n) \otimes H^*(\si_{m+l}; A_{m+l}) \arrow{r}{\id \otimes \rho} & H^*(\si_n; A_n) \otimes H^*(\si_m \times \si_l; A_m \otimes A_l) \arrow{r}{\id \otimes \chi_{m,l}} & H^*(\si_n; A_n) \otimes H^*(\si_m \times \si_l; A_m \otimes A_l) \arrow{r}{\cong} & H^*(\si_n; A_n) \otimes H^*(\si_m; A_m) \otimes H^*(\si_l; A_l).
	\end{tikzcd}}
%}
\end{center}
In the diagram above $ \rho $ denotes restriction maps, and the central square commutes by the coherence condition $ 1 $ of the definition of product series of symmetric representation. The outer square looks as follows and encodes the co-commutativity of $ \tilde{\Delta} $.
\begin{center}
\begin{tikzcd}
	H^*(\si_n;A_n) \otimes H^*(\si_m;A_m) \otimes H^*(\si_l;A_l) %\cong H^*(\si_n \times \si_m \times \si_l;A_n \otimes A_m \otimes A_l) 
	& \arrow{l}{\tilde{\Delta} \otimes 1} H^*(\si_{n+m} \times \si_l; A_{n+m}) \cong H^*(\si_{n+m}; A_{n+m}) \otimes H^*(\si_l; A_l) \\
	H^*(\si_n;A_n) \otimes H^*(\si_{m+l}; A_{m+l}) %\cong H^*(\si_n \times \si_{m+l}; A_n \otimes A_{m+l})
	\arrow{u}{1 \otimes \tilde{\Delta}} & H^*(\si_{n+m+l}; A_{n+m+l}) \arrow{l}{\tilde{\Delta}} \arrow{u}{\tilde{\Delta}}.
\end{tikzcd}
\end{center}
Moreover, the representation considered in the diagram at the bottom are all canonically identified with the restriction of the $ \si_{n+m+l}$-representation $ A_{n+m+l} $.
Therefore, there is also an induced diagram of transfer maps, obtained by replacing restrictions with transfer maps in the commutative diagrams above:
\begin{center}
	\begin{tikzcd}
		H^*(\si_n;A_n) \otimes H^*(\si_m;A_m) \otimes H^*(\si_l;A_l)) %\cong H^*(\si_n \times \si_m \times \si_l;A_n \otimes A_m \otimes A_l) 
		\arrow{r}{\tilde{\odot} \otimes 1}\arrow{d}{1 \otimes \tilde{\odot}} & H^*(\si_{n+m} \times \si_l; A_{n+m}) \cong H^*(\si_{n+m}; A_{n+m}) \otimes H^*(\si_l; A_l) \arrow{d}{\tilde{\odot}}\\
		H^*(\si_n;A_n) \otimes H^*(\si_{m+l}; A_{m+l}) %\cong H^*(\si_n \times \si_{m+l}; A_n \otimes A_{m+l})
		 \arrow{r}{\tilde{\odot}} & H^*(\si_{n+m+l}; A_{n+m+l}).
	\end{tikzcd}
\end{center}

The fact that both $ \tilde{\odot} $ and $ \cdot $ are morphisms of coalgebras with the preliminary coproduct is proved in a similar way using the following commutative diagrams for all $ n,m \in \mathbb{N} $, where $ i $ denotes the standard inclusions, $ \tau \colon \si_b \times \si_c \to \si_c \times \si_b $ the switching map, and $ d $ the diagonal maps:
\begin{center}
	\begin{tikzcd}[column sep=10em]
		\bigsqcup B(\si_a \times \si_b \times \si_c \times \si_d) \arrow{r}{\bigsqcup B(i_{a,b} \times i_{c,d})} \arrow{d}{\bigsqcup B(i_{a,c} \times i_{b,d}) \circ B(1 \times \tau \times 1)} & B(\si_n \times \si_m) \arrow{d}{B(i_{n,m})} \\
		\bigsqcup B(\si_p \times \si_q) \arrow{r}{\bigsqcup B(i_{p,q})} & B(\si_{n+m}) \\
		B(\si_n \times \si_m) \arrow{r}{B(1 \times \tau \times 1) \circ B(d_{\si_n} \times d_{\si_m})} \arrow{d}{i_{n,m}} & B((\si_n \times \si_m)^2) \arrow{d}{i_{n,m}\times i_{n,m}}\\
		B(\si_{n+m}) \arrow{r}{d_{\si_{n+m}}} & B(\si_{n+m}^2) \\
	\end{tikzcd}
\end{center}
where the first unions are over all $p+q=n+m$ and the second union is over all $a+b=n, c+d=m, a+c=p, b+d=q$.

 It is well-known in general that $ \cdot $ is associative. Since, under our hypotheses, all the morphisms involved in the definition of the coproduct are algebra maps with respect to the cup product, so is $ \tilde{\Delta} $.
\end{proof}

In the cases of our interest, the statement of Theorem \ref{main-section-3-pre} can be improved to show the commutativity of a twisted version of the coproduct and the products and Hopf rig distributivity.

%Let $ p $ be a prime.
%For all $ n \in \mathbb{N} $ we let $ A_n $ be the $ \mathbb{F}_p[\si_n] $-module given by $ A_n = \mathbb{F}_2 $, the trivial representation, if $ p = 2 $, and by $ A_n = \mathbb{F}_p \bigoplus \sgn $, the direct sum of the trivial and sign representation, if $ p > 2 $. Let $ \mu_n \colon A_n \otimes A_n \to A_n $ be the $ \si_n $-equivariant product obtained by combining the isomorphisms of $ \si_n $-representations $ \mathbb{F}_p \otimes \mathbb{F}_p \cong \mathbb{F}_p $, $ \mathbb{F}_p \otimes \sgn \cong \sgn $, $ \sgn \otimes \mathbb{F}_p \cong \sgn $ and $ \sgn \otimes \sgn \cong \mathbb{F}_p $.
%
\begin{definition} \label{def:basic objects}

Given a graded commutative algebra $ A $ over the field $ \mathbb{F}_p $, define the {\bf algebraic extended powers functor }
by $EP(A) =  \mathbb{F}_p \oplus \bigoplus_{n,d \geq 0} H^d(\si_n; {T_\rho A}_n) $ for odd primes.  Replace  $T_\rho A$ by $T A$ 
to obtain the definition when $p=2$.
Given a pointed topological space $ X $, we define $ \hdp(X) = EP(\tilde{H}(X)) $.
\end{definition}

We observe three different gradings on $ EP(A) $.  The first two are the cohomological grading $ d $, and component $ n $, defined before.
We also use the decomposition ${T_\rho A}_n \cong A^{\otimes n} \oplus (sgn \otimes A^{\otimes n})$ to induce a decomposition on $EP(A)$.

\begin{definition}
Over $\mathbb{F}_p$ with $p$ odd, consider the decomposition 
$$EP(A) \cong  \bigoplus_{n,d \geq 0} H^d(\si_n; A^{\otimes n} ) \oplus \bigoplus_{n,d \geq 0} H^d(\si_n; sgn \otimes A^{\otimes n} ).$$ 
We define
{\bf sign degree} $e$ of an element to be zero or even if it is in the first summand of this decomposition of 
 and to be one or odd if it is in the second summand.  Equivalently, we define $EP(A)^{d,n,e}$ to be the summand $H^d(\si_n; A^{\otimes n} )$
 when $e = 0$ or the summand $H^d(\si_n; sgn \otimes A^{\otimes n} )$ when $e= 1$.  
 We define the {\bf total degree} to be $t=ne+d$. As the sign degree, the total degree is only defined modulo $ 2 $.  %added
\end{definition}

%Thus, $EP(A)$ is a vector space which is tri-graded, over $ \mathbb{N} \times \mathbb{N} \times \frac{\mathbb{Z}}{2\mathbb{Z}} $.
 %We  write $ EP(A) = \bigoplus_{(d,n,e) \in \mathbb{N} \times \mathbb{N} \times \frac{\mathbb{Z}}{2\mathbb{Z}}} EP(A)^{d,n,e} $, 
% a decomposition into tri-homogeneous elements of tri-degree $ (d,n,e) $.
%Given a tri-homogeneous element $ x $, we denote by $ d(x) $, $ n(x) $ and $ e(x) $ its cohomological dimension, component, and sign degree respectively. 
%We also say that its total degree is $ t(x) = d(x)+e(x)n(x) $. As the sign degree, the total degree is only defined modulo $ 2 $.

By convention, if $ p = 2 $, then the sign degree is always even/ zero.  In order to ultimately account for signs we need the following.

%We craft the following definition that will make Theorem \ref{main-section-3} work.
\begin{definition}
A tri-graded component super-Hopf ring is a super-Hopf ring endowed with the component grading $ n $, the $ \mathbb{Z}/2\mathbb{Z}$-grading $ e $ arising from the ``super-'' structure, and an additional grading $ d $ over $ \mathbb{N} $ where $ \Delta $ preserves all three gradings, $ \odot $ preserves $ n $ and $ d $, $ \cdot $ preserves $ d $ and $ e $ in each component, and the axioms of super-Hopf ring are considered in the graded sense with respect to the total degree $ t = ne+d $.
\end{definition}

\begin{definition}\label{main-dphr-structure-modified}
Let $ EP(A) $ be the extended powers of an algebra, as above. 
Let $ \tilde{\Delta} $ and $ \tilde{\odot} $ be the maps constructed in Definition \ref{main-dphr-structure}.
We define:
 \begin{itemize}
 \item For all $ x $ such that, using Sweedler's notation, $ \tilde{\Delta}(x) = \sum x_{(1)} \otimes x_{(2)} $, with $ x_{(1)} $, $ x_{(2)} $ 
 all tri-homogeneous, the modified coproduct  is given by
 $$ \Delta(x) = \sum (-1)^{d(x_{(2)})n(x_{(1)})e(x_{(1)})} x_{(1)} \otimes x_{(2)}.$$
 \item For all $ x,y $ tri-homogeneous, the modified transfer product $ x \odot y = (-1)^{d(y)n(x)e(x)} x \tilde{\odot} y $.
 \end{itemize}
\end{definition}

%\BComment{In the following theorem and its proof we use the notation $ t(x) $ to refer to the total degree of $ x $: $ t(x) = n(x) e(x) + d(x) $. This is not defined above. We should either define it above or refer explicitly to $ n(x) e(x) + d(x) $ in the theorem. Which do you prefer? -LG}. \AComment{Let's define it above. -DS}\CComment{Definition inserted -PS}
We can finally state our main result of this section
\begin{theorem}\label{main:structure}
Let $ A $ be a graded commutative algebra of finite type over $ \mathbb{F}_p $. Then $ EP(A)$, 
with the coproduct $ \Delta $ and the products $ \odot $ and $ \cdot $,
 is a super-component Hopf rig bigraded by the component $ n $ and the total degree $ t $. If $ p = 2 $, it is a component Hopf ring bigraded by the component $ n $ and the cohomological dimension $ d $.
\end{theorem}
\begin{proof}
%\BComment{What from your work on alternating groups are you thinking of? I think that everything can be proved in a purely diagrammatic way. One needs just to be careful with twisted coefficients. -LG} \AComment{The basic structures developed in Section 2 of that paper.  Yes, I think we pretty much do things diagrammatically.  If some of the work is done there, we can site specific results from that paper. -DS} \BComment{Yes, I agree and I added the precise reference. However, it seems to me that in your paper on the alternating groups you only consider cohomology with trivial coefficients. For this reason, I think that we should still spend some words on the issue arising from conjugation maps not inducing the identity with twisted coefficients. When this issue is resolved, please delete these comments -LG}

Theorem \ref{main-section-3-pre} guarantees that $ EP(A) $ is a non-commutative Hopf algebra with the non-modified coproduct and transfer product $ \tilde{\Delta} $ and $ \tilde{\odot} $.  Clearly the additional signs do not disrupt the (co)associativity of the transfer product $ \odot $ and the coproduct $ \Delta $, that still form a Hopf algebra, and $ \Delta $ and $ \cdot $ still form a component bialgebra.
On the addend corresponding to the trivial representation, the remaining Hopf rig axioms are proved diagrammatically as in \cite[Theorem 2.4]{Giusti-Sinha:17}.
On the addend corresponding to the sign representation, one can adapt their proof by additionally keeping track of coefficients. However, the cocommutativity of $ \tilde{\Delta} $, the commutativity of $ \tilde{\odot} $ and the Hopf ring distributivity axiom fail because, contrary to constant coefficients, inner automorphisms do not induce the identity on cohomology with coefficients in non-trivial representations.
For instance, the diagram that would encode the commutativity of $ \tilde{\odot} $ and the cocommutativity of $ \tilde{\Delta} $ is the following:
\begin{center}
\begin{tikzcd}
\si_n \times \si_m \arrow{r}{\tau} \arrow[hook]{d} &\si_m \times \si_n \arrow[hook]{d} \\
\si_{n+m} \arrow{r}{\gamma} & \si_{n+m},
\end{tikzcd}
\end{center}
where $ \tau $ is the obvious switching isomorphism and $ \gamma $ is the conjugation by a permutation $ \sigma_{n,m} \in \si_{n+m} $ depending on $ n $ and $ m $.
In cohomology with trivial coefficients, such conjugation maps induce the identity, but in cohomology with coefficients in $ \sgn $ they induce the multiplication by $ (-1)^{nm} $. Therefore $ x \tilde{\odot} y = (-1)^{n(x)n(y)e(x)e(y)+d(x)d(y)} y \tilde{\odot} x $.
The additional signs introduced with the modified transfer product guarantee that $ x \odot y = (-1)^{t(x)t(y)} y \odot x $.
The same argument shows that $ \Delta $ is cocommutative. By analyzing the corresponding diagram, one proves similarly that $ (EP(A), \Delta,\odot,\cdot) $ satisfy the graded version of Hopf ring distributivity, with respect to the total degree.
\end{proof}

\begin{remark}
	Additively, $ H^*(\si_n; \sgn \otimes A^{\otimes n}) $ is isomorphic to $ H^*(\si_n, \Sigma(A)^{\otimes n}) $, where $ \Sigma $ is the suspension functor on graded vector spaces. Hence, for $ p $ odd, $ DP(A) \cong \bigoplus_{n \geq 0} H^*(\si_n; A^{\otimes n}) \oplus H^*(\si_n; \Sigma(A)^{\otimes n}) $ as vector spaces. The total degree on $ EP(A) $ is defined in such a way that this isomorphism becomes a degree-preserving linear map. Seen through this lense, the correction signs appearing in the modified coproduct and transfer product become less mysterious and arise from the Koszul sign conventions on $ Hom(W^{\si_n}_*;\Sigma(A)) $.
\end{remark}

%The connection with the topology of extended powers is the following.

For $p=2$, $\hdp(X)$ differs from the cohomology of $DX$ by an innocuous completion.  We see something similar at odd primes if we take only the
$H^*(D_n X; \mathbb{F}_p)$ summands.   Our work ahead is to understand the algebraic extended powers functor as a free functor  to the
category of Hopf rings with additive divided powers, which we next develop.

\begin{remark}\label{geometry2}
% For developing intuition, we recommend the model $\Conf{n}{\infty}$ for $E \si_n$, so that $D_n(X)$ is given by configurations with points labeled by $X$.  
As we remarked in the first section, if $X$ is a manfiold then cochains of finite dimensional approximations of $CX$
can be defined by submanifolds, which informally we consider through ``conditions'' on the underlying points of the
configuration or on their labels in $X$.  

Cup product as usual is given by intersecting, or in this case requiring that two sets of conditions hold.  Transfer product
defines a condition on $n+m$ points by asking that  a condition is satisfied on some subcollection of $n$ points and another is satisfied
 on $m$ points.  The transfer product of a class with itself is divisible by two because  the condition is satisfied
 whether some $n$ points or its complement is considered ``first.''  The divided powers 
operation repeats  a condition on $n$ points $k$ times to define a condition on $nk$ points.  

See Theorem 4.9 of \cite{Sinha:12} where we discuss the representatives for Hopf ring generators as subvarieties defined by $2^n$ points sharing a coordinate,
and \cite{Giusti-Sinha:14} which gives a  development through cellular models for one-point compactifications of configuration spaces.
A chain-level model for the transfer product is given which would also work for the divided powers operations.  
(Note that the cup product result claimed but not proven in \cite{Giusti-Sinha:14}   is not correct.)
\end{remark}

\subsection{Divided powers  on cohomology of symmetric groups with coefficients in a product series} \label{dpproductseries}

We now develop a divided powers structure on $ \hdp(X) $, as a special case of divided power structure on cohomology of symmetric
groups with coefficients in a product series of algebras.
%We will check the defining axioms for the divided powers operations on the transfer product, as described in Definition~\ref{def:divided powers}, 
%and the compatibility of divided powers with coproduct, as in Definition~\ref{Defdpbialg}, 
%which address in turn.
%removed: We utilize additional structure on the series of representations ${\widetilde{H}^*(X)}^{\otimes{n}} \otimes V$ of $\si_n$.  
In order to define it we need restriction and transfer maps for subgroups of symmetric groups which are defined by partitions. %changed

\begin{definition}\label{partitionsubgroup}
A {\bf labeled multipartition} $\pi$ of a set $S$ is the labeling of the leaves of a rooted tree by the elements
of $S$, possibly with additional labeling of internal vertices.  The subsets defined by considering
all of the leaf labels over a fixed internal vertex are called the {\bf blocks} of the labeled multipartition.
The {\bf depth} of a block  is the number of edges between its corresponding internal vertex
and the root vertex.

Given a labeled multipartition $\pi$ of $\{1, \dots, n\}$ define its {\bf automorphism group} 
$\si_{\pi}$ to be the automorphisms of the tree which preserve any additional labels of internal
vertices. We identify $\si_\pi$ canonically with subgroup of $\si_n$  through the action on leaves.
\end{definition}

A labeled multipartition is  determined up to additional labels by its blocks.  The additional labels of internal
vertices are sometimes used to make such structures more rigid, ruling out automorphisms between blocks.
Since we only use them in this way, we do not give explicit additional 
labels or name the labeling set; we only
indicate when labels are shared or they differ.

There are some basic constructions which give rise to the multipartitions and thus 
automorphism subgroups we consider.

\begin{definition}

If $\pi_1, \dots, \pi_k$ are labeled multipartitions of $S_1, \dots, S_k$ define the {\bf union} 
$\pi_1 \cup \dots \cup \pi_k$ to be the
multipartition of $\bigsqcup S_i$ 
in which the trees defining the $\pi_i$ are grafted to a (new) root with $k$ edges, 
and the internal vertices corresponding to the original roots are given distinct labels.

If ${\pi}$ is a labeled multipartition of $S$ define the {\bf multiple} $m \cdot \pi$ to be the labeled multipartition 
of $\bigsqcup_{i=1}^m S$ in which $m$ copies of the tree defining $\pi$ are grafted to a 
single root and given the same label.

Let $\ul{n}$ denote the trivial labeled partition of $\{1, \dots, n\}$, defined by a rooted
tree with no internal vertices.

\end{definition}

It is convenient to use the standard finite sets $\{1, \dots, n\}$ for all labels, in which case
 we identify the abstract disjoint union $\{ 1, \dots, n_1\} \sqcup \dots \sqcup \{1, \dots, n_k\}$
with the set $\{ 1, \dots, \sum n_i \}$ through the standard ``ordering on the page.''  We identify
 $\bigsqcup_m \{ 1, \dots, n \}  \cong \{ 1, \dots, mn \}$ similarly.

The fact that the union of labeled multipartitions has distinct labels while those of a multiple multipartition
are repeated is a key distinction, in a sense  giving rise to our divided powers operations.  
In particular, $\si_{\ul{n} \cup \ul{m}}$ is isomorphic to $\si_n \times \si_m$, including when $n =m$,
while $\si_{2 \cdot \ul{n}}$ is isomorphic to $\si_n \wr \si_2$.

We consider standard set partitions as labeled multipartitions through trees whose internal vertices
all have depth one, and identical labels.
Recall that the set of partitions of a set is a poset under refinement and that this poset has meet and join operators,
which we denote  $ \wedge $ and $ \vee $, that provide it with a lattice structure.
The following properties of $ \si_{\pi} $ are  straightforward. 

\begin{lemma}\label{lem:partitions}
Let $ \pi = \{ S_i \}$ and $ \pi' $ be partitions of $ \{ 1, \dots, n \} $. 
\begin{enumerate}
\item   $ \si_{m \cdot{\ul{ k}}} $ is isomorphic to the wreath product $ \si_k \wr \si_m $.  More generally,
  if $ m_\pi(i) $ is the number of parts of $ \pi $ whose cardinality is $ i $, 
  $ \si_{\pi} $ is conjugate to $ \bigtimes_{i=0}^n \si_i \wr \si_{m_\pi(i)} $ in $ \si_n $.
\item For any $ \sigma \in \si_n$, the conjugate $\sigma \si_{\pi} \sigma^{-1}$ coincides with $\si_{\sigma \pi} $.
\item The map $ [\sigma] \mapsto \sigma\pi $ is a bijection between the cosets 
$ \si_n / \si_{\pi} $ and the set of partitions that are permutations of $ \pi $.
\end{enumerate}
\end{lemma}

Partition subgroups play key roles in defining the divided powers. %changed
\begin{definition} \label{dp-series}
Let $\{A_n\}$ be a product series of algebras. For $x \in H^*(\si_n;A_n)$ we define $x\dip{k} \in H^*(\si_{kn};A_{kn})$ 
to be 
$$x\dip{k} := tr^{\si_{kn}}_{\si_n \wr \si_k} (x^{\otimes k}),$$
where we identify $ A_{kn} $ with $ A_n^{\otimes k} $ via the structural isomorphism.
\end{definition}

\begin{proposition}\label{firststep} %added hypothesis
If $\{A_n\}$ is a product series of algebras then
 $\bigoplus_n H^*(\si_n; A_n)$ with cup and transfer product from Definition \ref{main-dphr-structure}, coproduct from Definition \ref{main-dphr-structure-modified} and divided powers from Definition \ref{dp-series} satisfies the axioms 
 for divided powers structures.
 \end{proposition}
  
Before proving this, we set aside  a basic fact from group cohomology which we repeatedly use, 
which follows from the standard fact that restriction followed by transfer from a subgroup 
is multiplication by the index of a subgroup.  

\begin{lemma}\label{group cohomology} %labeled lemma
  If $K \subset H \subset G$ are finite-index inclusions
of groups then $$tr^{G}_{K} \circ res^{H}_{K} = [K:H] tr^{G}_{H} $$ %put a formula  
%was:  transfer of a cohomology class of $H$ restricted to $K$ is the index of $K$ in $H$ times its transfer directly from $H$.  

\end{lemma}

\begin{proof}[Proof of Proposition~\ref{firststep}]
%\AComment{Cite Strickland-Turner; I'm pretty
%sure they treated $DX$.  -DS}\BComment{First, I added a bibliography entry and a citation. I did not recall them treating $ DX $ explicitly, but of course I trust you, and I should have cited them anyways, because the argument is the same for $ DX $ and $ DS^0 $.
%Second, I rephrased some sentences in the proof below by using components, since it is probably clearer as we use $ \bigoplus H^*(D_nX) $ instead of $ H^*(DX) $. -LG} \AComment{Ah, I was mistaken.  I recalled correctly that the beginning of the Strickland-Turner paper 
%discusses $DX$, but their Hopf ring results are only stated for $DS^0$.  I have re-edited this.  -DS}

Let $x \in H^*(\si_n; A_n)$ be represented by a homomorphism $f$ from a resolution of $\si_n$ to $A_n$, and similarly
$y \in H^*(\si_m; A_m)$ be represented by $g$. 

The $0,1$-Cases  axioms of Definition~\ref{def:divided powers} are immediate.  For the Binomial axiom, 
we calculate $(x + y)\dip{r}$ by taking the transfer (induction)  of $(f + g)^{\otimes r}$, 
which is a sum over $i$ of shuffles of $f^{\otimes i}$ and $g^{\otimes (r-i)}$.  But the transfer on such
shuffles is exactly $x\dip{i} \odot y\dip{r-i}$.

Checking the Exponent Axiom, 
$x\dip{m} \odot x\dip{r}$ and $x\dip{m+r}$ are both images under transfer of $f^{\otimes m + r}$, but the former is induced up from the 
subgroup $(\si_{m \cdot \ul{n}}) \times (\si_{r \cdot \ul{n }})$, while the latter is induced up from 
$\si_{(m + r) \cdot \ul{n}}$.  
%As they are related by being the transfer of a restriction,
By Lemma~\ref{group cohomology},  the former is  obtained from the latter by 
multiplication by the quotient of the indices of these subgroups, namely $\binom{m+r}{m}$.

 We next verify the Distributivity Axiom, which in the Hopf rig setting is that $ (x \odot y)\dip{k} = x^{\odot^k} \odot  y\dip{k} $.
 %assume, as before, that $ x \in H^*(\si_n; A_n) $ and $ y \in H^*(\si_m; A_m) $ are represented by two homomorphisms from a %resolution of symmetric groups to suitable terms of the product-wreath series, that we denote by 
 By definition
$ (x \odot y)^{\otimes^k} $ is a composite of induction maps  of $ (f \otimes g)^{\otimes^k} $ from 
$ \si_{k \cdot ({\ul{n} \cup \ul{m}})}$  ultimately to $ \si_{k(n+m)} $.
Similarly,  $ x^{\odot^k} \odot y\dip{k} $ is obtained by inducing $ f ^{\otimes^k} \otimes g^{\otimes^k} $ from 
$\si_{\cup_k \ul{n} \cup k \cdot \ul{m}}$. %  $ \si_n^k \times (\si_m \wr \si_k) $ to $ \si_{k(n+m)} $.
After conjugating so that $ (f \otimes g)^{\otimes^k} $ becomes $ f ^{\otimes^k} \otimes g^{\otimes^k} $, 
both subgroups are then subgroups of $\si_{k \cdot \ul{n} \cup k \cdot \ul{m}}$, for which 
$ f ^{\otimes^k} \otimes g^{\otimes^k} $ is invariant.
 Lemma~\ref{group cohomology} applies in both cases, so that each result is the induction 
 from $\si_{k \cdot \ul{n} \cup k \cdot \ul{m}}$  multiplied by the index of the corresponding subgroup.
  The Distributivity axiom follows as  these two indices are the same.

We verify the Composition Axiom $ (x\dip{k})\dip{h} = \frac{(hk)!}{(k!)^h h!} x\dip{hk} $ similarly.
By definition  $ x\dip{hk} $ is
given by inducing $ f^{\otimes^{hk}} $ from $\si_{hk \cdot \ul{n}}$ to  $ \si_{nhk} $, 
while $ (x\dip{k})\dip{h} $ is the induction of the same cochain from $\si_{h \cdot (k \cdot \ul{n})}$.
Since $\si_{h \cdot (k \cdot \ul{n})} \subset \si_{hk \cdot \ul{n}}$,  Lemma~\ref{group cohomology} applies
to show that the differ by a multiplicative coefficient equal to the index of this inclusion, which  is exactly $ \frac{(hk)!}{(k!)^h h!} $.
\end{proof}

 We next turn to compatibility of divided powers with coproduct.
 As our divided powers structure are defined by transfers, and our coproduct is defined by restriction, the
 Cartan-Eilenberg Double Coset Formula will be of use.    
Recall that if $ H $ and $ K $ are subgroups of a finite group $ G $ and $ \mathcal{R} $ is a set of 
representatives in $ G $ for the set of double cosets $ H \backslash G / K $, then, on cohomology with 
coefficients in a given $ G $-representation,
\[
\rho^G_H \circ \tr^G_K = \sum_{r \in \mathcal{R}} \tr^H_{H \cap c_r K} \circ \rho^{c_r K}_{H \cap c_r K} \circ {c_r}^*
\]
where $c_r$ is the isomorphism of $K$ with $r K r^{-1}$ and 
$ \rho $ and $ \tr $ denote restriction and transfer maps respectively.
This is proved for example as Theorem I.6.2 of \cite{Adem-Milgram}.

\begin{proposition}\label{coproductcompat}
If $\{A_n\}$ is a product series of algebras, then %added hypothesis
on $\bigoplus_n H^*(\si_n; A_n)$ 
 %as in Definitions \ref{main-dphr-structure},
the divided powers of Definition \ref{dp-series}  
commute with the coproduct of Definition \ref{main-dphr-structure-modified}.
\end{proposition}

\begin{proof}
We start with a more explicit expression of divided powers of a coproduct.   If $x \in H^*(\si_n; A_n)$, set $\Delta x = \sum x_{(i,n-i)}$,
where $x_{(i,n-i)} \in H^*(\si_i; A_i) \otimes H^*(\si_{n-i}; A_{n-i})$.  By the Binomial Axiom, which was just established
in Proposition~\ref{firststep}, 
$$
\left(\Delta x \right)\dip{k} = \sum_{j_0 +\dots + j_n = k} \bigodot_{i=0}^{n} \left( x_{(i,n-i)} \right)\dip{j_i}.
$$
If $a + b = n$ the  $(a,b)$ component of this consists of terms with $\sum j_i \cdot i = a$.  We show that $\left(\Delta (x\dip{k}) \right)_{(a,b)}$
agrees with this.

Let $x \in H^*(\si_n; A_n)$ be represented by a homomorphism $f$ from a resolution of $\si_n$ to $A_n$.
By definition $ (x\dip{k})_{(a,b)}$ will be represented by
$ \rho^{\si_{nk}}_{\si_{\ul{a} \cup \ul{b}}} \circ \tr^{\si_{nk}}_{\si_{k \cdot \ul{n}}} (f^{\otimes k})$.  Applying the Cartan-Eilenberg Double Coset 
Formula, this coincides with 
\[
 \sum_{r \in \mathcal{R}} \tr_{ \si_{\ul{a} \cup \ul{b}} \cap c_r \si_{k \cdot \ul{n}} }^{\si_{\ul{a} \cup \ul{b}}} \circ
  \rho_{\si_{\ul{a} \cup \ul{b}}  \cap c_r \si_{k \cdot \ul{n}} }^{c_r \si_{k \cdot \ul{n}}} \circ c_r^{\#}  (f^{\otimes^k}),
\]
where $ \mathcal{R} $ is a set of representatives for the set of double cosets 
$ (\si_{\ul{a} \cup \ul{b}})\backslash\si_{nk}/(\si_{k \cdot \ul{n}}) $. 

To understand this double coset space, recall that the cosets of symmetric groups modulo the automorphism groups
$\si_\pi$ correspond to partitions with the same shape as $\pi$.  Thus $ \si_{nk}/  \si_{k \cdot \ul{n}} $ corresponds to partitions of 
$ \{1, \dots, nk\} $ with $ k $ parts of cardinality $ n $, while $ \si_{\ul{a} \cup \ul{b}} \backslash \si_{nk}$ corresponds to 
partitions into two sets with cardinalities $a$ and $b$, or equivalently bicolorings.  Thus 
$ \mathcal{R} $ corresponds to bicolored partitions of $ \{1, \dots, nk\} $ with $ k $ parts of cardinality $ n $.
Equivalently $ \mathcal{R} $  is given by 
multipartitions governed by a tree with $k$ edges attached to the root, and each of those with at most 
two edges labeled $a$ and/or $b$, with a total of $n$ leaf edges over them.   All labels in the multipartition are identical.
The intersections $\si_{\ul{a} \cup \ul{b}}  \cap c_r \si_{k \cdot \ul{n}}$ are  the automorphisms of such multipartitions.  These are conjugate to $ \si_{\cup_i m_i \cdot (\ul{i} \cup \ul{n-i})}$, where 
$m_i$ is the number of times there are $i$ leaves in the group labeled
 by $a$ within a $k$-block.

We then identify  the restriction 
$\rho_{\si_{\ul{a} \cup \ul{b}}  \cap c_r \si_{k \cdot \ul{n}} }^{c_r \si_{k \cdot \ul{n}}} \circ c_r^{\#}  (f^{\otimes^k})$ 
 with $ \bigotimes_{i=0}^n {f_{(i,n-i)}}^{\otimes^{m_i}} $, where ${f_{(i,n-i)}}$ represents $x_{(i,n-i)}$. 
 To calculate $\tr_{ \si_{\ul{a} \cup \ul{b}} \cap c_r \si_{k \cdot \ul{n}} }^{\si_{\ul{a} \cup \ul{b}}} $ applied
 to this class, we factor the inclusion of $\si_{\ul{a} \cup \ul{b}}  \cap c_r \si_{k \cdot \ul{n}}$ in 
 $\si_{\ul{a} \cup \ul{b}}$.  Since $\si_{\ul{a} \cup \ul{b}}  \cap c_r \si_{k \cdot \ul{n}}$ is conjugate
 to $ \si_{\cup_i m_i \cdot (\ul{i} \cup \ul{n-i})}$ it is contained in a conjugate of 
 $ \si_{\cup_i (\ul{i m_i } \cup \ul{ (n-i)m_i})}$.  The transfer map up to this subgroup is a product
 of transfer maps defining divided powers, so $ \bigotimes_{i=0}^n {f_{(i,n-i)}}^{\otimes^{m_i}} $
 is sent to $ \bigotimes_{i=0}^n (f_{(i,n-i)})\dip{m_i} $.  The transfer map from 
 $ \si_{\cup_i (\ul{i m_i } \cup \ul{ (n-i)m_i})}$ to $\si_{\ul{a} \cup \ul{b}}$ is  the transfer
 product on the tensor product of the cohomology with itself, 
 so $ \bigotimes_{i=0}^n (f_{(i,n-i)})\dip{m_i} $ is sent to $ \bigodot_{i=0}^n (f_{(i,n-i)})\dip{m_i}$, 
 which agrees with our determination of $\left((\Delta x)\dip{k} \right)_{(a,b)}$ above. 
\end{proof}

\begin{corollary} \label{divided-powers-representations}
$ \hdp(X) $ is a bigraded component Hopf ring with additive divided powers.
\end{corollary}

%We now prove Theorem \ref{main-section-3} in steps.  
%That  $\bigoplus_n H^*(\si_n; A_n)$  is a Hopf ring follows immediately from the arguments
%made for the generalized cohomology of $DS^0$ by Strickland and Turner \cite{Strickland-Turner:97} or the almost-Hopf ring structure
%for cohomology of more general series of groups in \cite{Giusti-Sinha:17}.   
%So we need check the defining axioms for the divided powers operations on the transfer product, as described in Definition~\ref{def:divided powers}, the compatibility of divided powers with coproduct as in Definition~\ref{Defdpbialg}, 
%and the Hopf rig compatibility axiom of Definition~\ref{dpHopfrig}, which address in turn.

\subsection{Statement of first main results}

We state here our first main structure theorem, that we will prove in the remaining sections.

\begin{comment}
\BComment{Should the following theorem be deleted too once 2.37 and 2.38 are properly formulated? -LG}
\CComment{Yes -PS}
\begin{theorem} \label{main:structure}
Let $ X $ be a topological space, let $ p $ be a prime, and let $ \rho $ be the representation of Definition \ref{def:rho}. Then $ \hdp(X) $ is a bigraded component super-Hopf ring with additive divided powers satisfying these properties:
\begin{enumerate}
\item there is a map $ \pi \colon \bigoplus_{n \geq 0} H^*(\si_n; A_n) \to \hdp(X) $ that is both a super-Hopf ring map and a morphism of divided powers structures preserving the unit of each component;
\item there is a homomorphism of graded super-algebras $ \iota_X $ from $ H^*(X) \otimes \rho $ to the graded super-algebra of elements of component $ 1 $ inside $\dph(X) $; 
\item for all $ n \in \mathbb{N} $, for all $ x,x' \in H^*(X) $ (of even total degree if $ p > 2 $) and for all $ y \in H^*(\si_m; A_m) $, $ \iota_X(x)\dip{nm} \cdot \iota_X(x')\dip{nm} \cdot \pi(y)\dip{n} = (\iota_X(x)\dip{m} \cdot \iota_X(x')\dip{m} \cdot \pi(y))\dip{n} $.
\end{enumerate}
Moreover, $ \hdp(X) $ is universal among the bigraded component super-Hopf rings with additive divided powers satisfying these properties.
\end{theorem}
\end{comment}

%\AComment{I propose the following definition /  presentation. -DS}
%\CComment{It is not clear to me how to extend this formulation for $p$ odd, and to make sense of a Hopf ring over a ground ring -PS}
%\BComment{See my comments below. -LG}

We require two more definitions before stating our first main theorems.

\begin{definition} \label{non-nilpotent}
An element  of a divided powers algebra of even degree when $p$ odd, or any degree when $p=2$,  is
{\bf standard non-nilpotent} if all $x\dip{m}$ are non-zero, with  $\Delta x\dip{m} = \sum_{i + j = m} x\dip{i} \otimes x\dip{j}$ and 
$x\dip{i} \odot x\dip{j} = \binom{i+j}{i} x\dip{i+j}$.
\end{definition}

\begin{definition}
Let $\mathcal{C}$ be a category, $S$ a subcategory, and $F : \mathcal{C} \to \mathcal{D}$ a functor.  We say some object $x \in S$ is universal
among $S$ with respect to $F$ if its image under $F$ is initial in the full subcategory generated by $F(\mathcal{C})$.
\end{definition}

We first give an immediate reformulation of the calculations of Giusti--Salvatore--Sinha  utilizing divided powers structure.

\begin{theorem}[from \cite{Sinha:12}]\label{m1}
The mod-two cohomology of extended powers of $S^0$, namely $H_{EP}^*(S^0) \cong \bigoplus H^*(B\si_n)$, is a 
component Hopf rig with additive divided powers which contains standard non-nilpotent classes $\gamma_i \in H^{2^i - 1}(B \si_{2^i})$.
It is universal among such objects, with respect to the  functor which forgets divided powers.
\end{theorem}
\begin{comment}
\BComment{I think that this formulation is not entirely correct, because it doesn't take into accounts units (but correct me if I am wrong). I suggest the following two (equivalent) formulations, which should be correct:
\begin{theorem*}[Formulation 1]
Let $ \mathcal{C} $ be the category of Hopf rigs with additive divided powers over $ \mathbb{F}_2 $ such that each component is a unital algebra under $ \cdot $ and the units form a standard non-nilponent sequence. Let $ \mathcal{D} $ be the category of Hopf rigs satisfying the same conditions on the $ \cdot $-units of components. Let $ F \colon \mathcal{C} \to \mathcal{D} $ be the obvious forgetful functor.
The mod-two cohomology of extended powers of $ S^0 $, namely $ H_{EP}^*(S^0) \cong \bigoplus H^*(B\si_n; \mathbb{F}_2) $ is the universal object in $ \mathcal{C}$ which contains standard non-nilpontent classes $ \gamma_i \in H^{2^i-1}(B\si_{2^i}) $ for $ i \geq 1 $, with respect to the functor which forgets divided powers.
\end{theorem*}

\begin{theorem*}[Formulation 2]
The mod-two cohomology of extended powers of $ S^0 $, namely $ H^*_{EP}(S^0) \cong 1bigoplus H^*(B\si_n; \mathbb{F}_2) $ is a  component super-Hopf rig with additive divided poewrs which contains standard non-nilponent classes $ \gamma_i \in H^{2^i-1}(B\si_{2^i}) $ for all $ i \geq 0 $, quotiented by the relation
\[
\gamma_0\dip{n} \mbox{ is the $ \cdot $-unit of the $ n^{th} $ component}.
\]
\end{theorem*}
}
\end{comment}

In other words, there are no relations required to understand  $H_{EP}^*(S^0)$ as a Hopf ring other than those given in the axioms of component 
Hopf rig with additive divided powers.  This thus  determines the cohomology ring structure.
In particular, the subrings generated by $\gamma_{i}\dip{2^j}$ for $i+j = n$ are polynomial subrings of the cohomology
of $B\si_n$.  But there are many relations for classes involving the transfer product.   At the moment, we do not fully understand the divided powers 
structure -- see Remark~\ref{bad-example} below -- so while we use divided powers to generate classes we forget them in order to have
 a unique characterization.

The following is the odd primes counterpart of Theorem \ref{m1}, that we will prove in Section \ref{sec:symmetric group}.
For $p>2$, let $ T_{\rho}(H^*(X;\mathbb{F}_p)) = \{ T_{\rho,n}(H^*(X; \mathbb{F}_p)) \}_{n \in \mathbb{N}} $ be the super-product series of algebras defined above. We consider the sign degree $ e $ that determines the super-structure and the degree $ d $ induced by the cohomological dimension of $ H^*(X; \mathbb{F}_p) $. 
%above taken out of statement
\begin{theorem}\label{m2}
The mod-$p$ cohomology of extended powers of  $S^0$, namely %I put S^0 instead of point for coherency with above theorem, and added word "universal"
 $H_{EP}^*(S^0) \cong \bigoplus H^*(B\si_n; \rho)$, is a  trigraded component super-Hopf rig with additive divided powers which contains 
\begin{itemize}
\item  standard non-nilpotent classes 
$ {\gamma_k}$ in grading $(2(p^k-1),p^k,0)$,
\item primitive classes $ \lambda_k $ in grading $(p^k-1,p^k,1) $,
\item  standard non-nilpotent classes   $ {\gamma'_{k}}$ in grading $(p^k-2,p^k,1)$,

\end{itemize}
 with  cup product relations $ {\lambda_k}^2 = {\gamma_k}$.  %I cut part of sentence
 It is universal among such objects with respect to the  functor which forgets divided powers.

\end{theorem}

\begin{comment}
\BComment{
Again, this is not entirely correct. There is an issue with cup-product units of components that can be solved by reformulating as done in the mod 2 case. However, one also need to introduce the following relation:
\begin{center}
the transfer product of elements of different sign degrees is zero
\end{center}
That seems to me more difficult to properly formulate in the category of super-Hopf rigs with components and additive divided powers, because that structure depends only on the component and the total degree, while the transfer product relation above makes use of the existence of a tri-grading. Could we maybe provide a definition of "tri-graded super-Hopf rig with components and additive divided powers" and work with that? I don't know. -LG
}
\end{comment}

\begin{remark}
The sub-Hopf rig of $ \hdp(S^0) $ corresponding to the addend of even sign degree has been calculated by Guerra \cite{Guerra:17}. 
The generating classes $ \gamma_k $, $ \alpha_{i,k} $ and $ \beta_{i,j,k} $ appearing there can be retrieved from our presentation as 
$ \gamma_{k} $, $ (-1)^{\frac{p-1}{2}(k-i)} \lambda_k {\gamma'}_{i}\dip{p^{k-i}} $ and $ (-1)^{\frac{p-1}{2}(j-i)} {\gamma'}_{i}\dip{p^{k-i}} {\gamma'}_{j}\dip{p^{k-i}} $, respectively. %\AComment{I corrected this statement, I think, since it said $ \gamma_{\varnothing, k} $, which isn't even
%a class we've defined.  It now looks OK to me, but I think we should each check it again.  -DS}
\end{remark}

For any nonempty 
$X$, the cohomology of $D(X_+)$ is an algebra over the cohomology of $D(S^0)$, through the projection map which sends $X$ to the non-base
point of $S^0$.  This is injective, with any choice point in $X$ giving a splitting.
%\BComment{Fixed the statements of the two theorems below. -LG}
We use this algebra structure for  the two following results, that will be proved in Section \ref{sec:DXdphr}. For odd primes, we compute a bigger Hopf ring with cohomology taken with coefficients in $ \rho $, but one can easily recover the ordinary mod $ p $ cohomology of $ \tilde{D}(X) $ as a Hopf ring by extracting only the homogeneous part of even sign degree.

\begin{theorem}\label{m3}
The mod-two cohomology of extended powers of  $X_+$  is the universal 
component super-Hopf rig with additive divided powers over the cohomology of extended powers of $S^0$
which contains the cohomology classes of $X$  as classes which are standard non-nilpotent.
 Relations are
$$ x\dip{n} \cdot y\dip{n} = (x \cdot y)\dip{n} \;\; {\text for} \;\; x,y \in H^*(X) \mbox{ and}$$  
$$ x\dip{nm} \cdot y\dip{n} = (x\dip{m} \cdot y)\dip{n} \;\; {\text for} \;\; x \in H^*(X),  y \in H^*(D_m(S^0)).$$
It is universal among such objects with respect to the  functor which forgets divided powers.

\end{theorem}

\begin{theorem}\label{m4}
For $p>2$, the mod-$p$ cohomology of extended powers of $ X_+ $  is a
component super-Hopf rig with additive divided powers over 
the cohomology of extended powers of a point which contains two copies of the cohomology of $X$, which 
we denote by $x_e$ degree $(d,1,e)$ for $e \in \{ 0, 1 \}$.  These classes are standard non-nilpotent if $d + e$ is even and 
primitive if $d + e$ is odd, with 
		$$ x_e\dip{n} \cdot y_{e'}\dip{n} = (x \cdot y)_{e+e'}\dip{n},  $$
		$$ x_e\dip{nm} \cdot z\dip{n} = (x_e \dip{m} \cdot z)\dip{n} \mbox{ and} $$
		$$ x_e \odot y_{e'} = 0 \mbox{ if } e \not= e', $$
		if $ x,y \in H^*(X; \mathbb{F}_p) $, $ z \in H^*(D_m(S^0); \mathbb{F}_p) $, and $ e,e' \in\{0,1\} $,
 with the sum $ e + e' $  understood modulo two.
 It is universal among such objects with respect to the  functor which forgets divided powers.

\end{theorem}

\begin{remark}\label{bad-example}
Theorems \ref{m3} and \ref{m4} explicitly embed the cohomology of the symmetric groups and the cohomology of $ X $  in $ \hdp(X) $. 
The relations provide a compatibility identity between divided powers and cup product. We can interpret this structure as a 
``bigraded component Hopf ring with divided powers over the cohomology of symmetric groups generated by $ H^*(X) $''.
More generally, one could define a ``bigraded component Hopf ring with divided powers'' as a bigraded component Hopf ring with additive 
divided powers $ A $ such that a similar compatibility condition between divided powers and cup product hold for all $ x,y \in A $ divided 
powers of primitive elements. However, this simple identity does not extend to classes $ x $ not arising from the cohomology of $ X $.
For example, with $ p = 2 $ and $ X = \{*\} $, %\BComment{Throughout the paper we (mostly I) use the notation $ DX $ inconsistently: sometimes we let $ DX = \bigsqcup_n E(\si_n) \times_{\si_n} X^n $, some other we denote that $ \tilde{D}X $ and we let $ DX = \bigvee_n E(\si_n) \ltimes_{\si_n} X^{\wedge n} $. We should make a consistent choice and stick to it. In the first case, here we would have $ DX $ with $ X= \{ * \} $, in the second case $ DX $ with $ X = S^0 $ or $ \tilde{D}X $ with $ X = \{*\} $. -LG} \AComment{I agree that we (I am also quite guilty) have been inconsistent and should make this consistent.  Let's discuss merits of different choices. -DS}  \CComment{I have a preference for the classic notation $DX$ that is the second -PS}   \BComment{This should be fixed now. If it is good for you, please delete these comments. -LG}
$ \hdp(X_+) $ becomes the ordinary mod $ 2 $ cohomology of the symmetric groups and the divided powers operations 
$ \_\dip{k} $ agree with the cohomological transfer maps associated to $ \si_k \wr \si_n \to \si_{nk} $. A procedure to compute these 
transfer maps is described by Kechagias \cite{Kechagias:09}. Combining Kechagias's algorithm with the computations of the restriction to 
elementary abelian subgroups achieved by Giusti--Salvatore--Sinha \cite{Sinha:12} one obtains the following equality
\[
(\gamma_{1})\dip{2} \cdot (\gamma_{1}^2)\dip{2} = \gamma_{2}^2 + (\gamma_{1}^3)\dip{2}.
\]
As the right hand side is not simply $(\gamma_{1}^3)\dip{2}$, the relationship between the divided powers operations and the 
cup product in $ \hdp(X) $ is not as expected and appears likely to be complicated in general.
\end{remark}

\begin{remark}\label{geometry3}
Recall from Remark~\ref{geometry2} that Hopf ring with additive divided powers structures all have simple geometric cochain models when
using configuration space models for extended powers of manifolds.  Many of the basic classes in Theorems~\ref{m1}-\ref{m4} 
have geometric cochain models as well.  The $\gamma_k$ for $\F_2$ are represented by ``$2^k$ points which share a coordinate''.  Formally
one considers the manifold consisting of triples $(x, X, Y)$ where $x \in \R$, $X$ is a configuration of $2^k$ distinct unordered points in
$x \times \R^{N-1} \subset \R^N$, and $Y$ is a configuration of unordered points in the complement of $X$ in $\R^N$.  This manifold maps
properly to the configuration space, generically an embedding but finite-to-one if some point in $Y$ is contained in 
 $x \times \R^{N-1}$ or some $2^k$ points in $Y$ contained in an $x' \times \R^{N-1}$.  It thus defines a geometric cocycle, which represents 
 $\gamma_k$ as shown in Theorem~4.9 of \cite{Sinha:12}.
 Similarly, the $\gamma_k$ for odd primes are represented by the geometric 
 cochain defined by ``$p^k$ points in $\mathbb{C}^N$ which share a complex coordinate.''  The cup products $\gamma_k \cdot x\dip{k}$
 on the $k$th extended power is represented by labeled points which share a coordinate and are labeled by the representative of $x$.
 One must name which extended power because the cup product of manifestations of these classes on higher extended powers, given by transfer products
 with unit classes, can differ from this (as some points in a configuration could share their coordinate while others share their label).
 
 We have not developed geometric representatives for the $\gamma_k'$ and $\lambda_k$, or their even sign degree products
 as considered earlier in \cite{Guerra:17}.  And as mentioned before, one must take care with geometry and cup products since one would
 expect  ${\gamma_1\dip{2}}^3$ to be represented by the geometric cochain defined by 
 four points which consist of two sets of two points each of which share three coordinates.
 But it is represented by the union of this along with a copy of the geometric cochain defined by four points which share two coordinates.
 \end{remark}

\subsection{Explicit construction of (relatively) universal component Hopf rings with divided powers}\label{explicit}

We next better understand the structures elaborated in  the main theorems just stated, 
abstracting as follows.

\begin{definition}\label{def: H-alg}
Let $A$ be a component super-Hopf ring with additive divided powers containing $ \rho $ as a subalgebra of $ A_1 $, 
its $ \cdot $-algebra in component $ 1 $, 
 and let $V$ be a connected graded algebra. 
 Let $ H_{alg} (A, V)$, which we generally shorten to $ H_{alg} $, be the super-Hopf ring with additive divided powers  satisfying 
 %the analog of the universal property of the statements of Theorems \ref{m3} and \ref{m4} (depending on $ p=2 $ or $ p \not= 2 $) where $ \hdp(S^0) $ is replaced with $ A $ and $ H^*(X; \mathbb{F}_p) $ is replaced with $ V $. Explicitly, this amounts to $ H_{alg} $ being universal with
  the following five properties.
\begin{enumerate}
	\item There is a map $ \pi \colon A \to H_{alg} $ that is both a super-Hopf ring map and a morphism of divided powers structures 
	preserving the $ \cdot $-unit 	 of each component;
	%perhaps change notation, the product is not so visible

	\item There is a $ \rho $-algebra %meaning?
	 homomorphism $ \iota $ from $ V \otimes \rho $ to the subspace of elements of component $ 1 $ inside $ H_{alg} $, 
	 (that is, an algebra morphism $ \iota \colon V \to H_{alg} $ in component $ 1 $ if $ p = 2 $);
	\item For all  $ x,x' \in V \otimes \rho $ both with even total degree $ \iota(x)\dip{n} \cdot \iota(x')\dip{n}= \iota(xx')\dip{n} $;
	\item For all  $ x \in V \otimes \rho $ with even total degree and for all $ y \in A $ such that $ n(y) = m $, %notation kind of clashes with the other n
	 $$ \iota(x)\dip{nm} \cdot \pi(y)\dip{n} = (\iota(x)\dip{m} \cdot \pi(y))\dip{n};$$
	\item If $ p > 2 $, $ x\dip{n} \odot {x'}\dip{m} = 0 $ if $ x \in V \otimes \mathbb{F}_p \subseteq V \otimes \rho $ and $ x' \in V \otimes \sgn \subseteq V \otimes \rho $.
\end{enumerate}
\end{definition}

We wish to describe, in terms of additive bases of $ A $ and $ V $, a basis for  $ H_{alg} $. 
In Section~\ref{sec:algebraic basis} will  specialize it to our case of interest, where 
$ A = \mathbb{F}_p \oplus \bigoplus_{n \geq 1} H^*(\si_n; \rho) $ and $ V = H^*(X; \mathbb{F}_p) $.

In the remainder of this section, we will assume that $ p > 2 $, because the treatment for $ p = 2 $ is similar and much  simpler. 
Let $ V_{even} $ (respectively $ V_{odd} $) be the subspaces of $ V \otimes \rho $ of even (respectively odd) total degree.

As $ A $, with the product $ \odot $ and the coproduct $ \Delta $ alone is a bicommutative divided powers Hopf algebra, 
by a classical result \cite{Andre:71} it must be the free divided powers Hopf algebra on its subspace of primitive elements.  That is,
$ A \cong DP(P(A)) $. Since $ A $ is a component super-Hopf ring, $ P(A) $ is a non-unital component super-algebra with $ \cdot $, %again not so visible
 that inherits a triple grading from $ A $. 
 
 We consider the tensor product of algebras $ V_{even} \otimes P(A) $, that we tri-grade with the following rule: if $ v \in V_{even} $ is homogeneous 
 of degree $ d(v) $ and $ y \in P(A) $ is tri-homogeneous of tri-degree $ (d(y),n(y),e(y)) $, then $ v \otimes y \in V \otimes P(A) $ is 
 tri-homogeneous of tri-degree $(d(v)n(y)+d(y),n(y),e(y)+d(v)) $.
Intuitively, we interpret a pure tensor $ v \otimes y \in V_{even} \otimes P(A) $ as the element $ \iota(v)\dip{n(y)} \cdot \pi(y) $. This justifies our choice of degrees. Recall that the sign degree $ e $ is defined only modulo $ 2 $.

From a graded module $ M $ we can construct an augmented non-unital coalgebra $ M^{pr} = \mathbb{F}_p \oplus M $ 
with the $ \mathbb{F}_p $ addend in tri-degree $ (0,0,0) $ and  a coproduct $ \Delta(x) = 1 \otimes x + x \otimes 1 $ for all $ x \in M $. 
This is the ``primitive coalgebra extension'' of the module $ M $.

We then decompose $ V_{even} \otimes P(A) $ as the direct sum $ (V_{even} \otimes P(A))_+ \oplus (V_{even} \otimes P(A))_- $, 
where $ (V_{even} \otimes P(A))_+ $ is the subspace generated by elements $ v \otimes x \in V_{even} \otimes P(A) $ in which $ e(v) = e(x) $ mod $ 2 $ 
and $ (V_{even} \otimes P(A))_- $ consists of such with  $ e(v) \not= e(x) $ mod $ 2 $).  Then
\[
(H_{alg},\odot,\Delta) \cong {\dpha((V_{even} \otimes P(A))_+^{pr}) \oplus \dpha((V_{even} \otimes P(A))_-^{pr})}/{\sim},
\]
where $ \sim $ is the equivalence relation that identifies the units of the two addends.
Additive bases for $ V $ and $ P(A) $ induce a basis for $ \dpha((V_{even} \otimes P(A))_\pm^{pr}) $, and consequently on 
$ H_{alg} $, consisting of elements of the form $ \bigodot_i (v_i \otimes y_i)\dip{n_i} $ with $ v_i, y_i $ basis elements with 
$ n_i = 1 $ if $ v_i \otimes y_i $ has odd total degree and $ v_i \otimes y_i $ belong all to $ (V \otimes P(A))_+ $ or all to $ (V \otimes P(A))_- $.

For now, $ H_{alg} $ is only a component divided powers super-Hopf algebra. Under this isomorphism we
 next identify maps $ \pi \colon A \to H_{alg} $ and $ \iota \colon V \otimes \rho \to H_{alg} $, as well as a second product 
 $ \cdot \colon H_{alg} \otimes H_{alg} \to H_{alg} $ which provide it with a super-Hopf ring structure.

First, we define $ \pi \colon A \to H_{alg} $ as the unique divided powers super-Hopf algebra morphism extending the linear map on primitives $ \pi \colon y \in P(A) \mapsto 1_V \otimes y \in V_{even} \otimes P(A) \subseteq P(H_{alg}) $.
Second, let $ \iota \colon V \otimes \rho \to H_{alg} $. If $ v \in V_{even} $, then we let $ \iota(v) = v \otimes 1_1 $, 
where $ 1_1 \in A $ is the $ \cdot $-product unit of the $ 1 $-component of $ A $. 
If $ v \in V_{odd} $, then let $ s \in \rho $ be a generator of the $ \sgn $ addend. Since $ v s \in V_{even} $,  we can define $ \iota(v) $ as 
$ \iota(vs) \otimes s $, where $ s \in \rho \subseteq A_1 $ is considered as a primitive element of $ A $.

We construct the second product $ \cdot $ on $ H_{alg} $ step-by-step. As always, $ \cdot $ between elements of different components will be zero. 
As an intermediate step, we define the product $ \iota(v)\dip{n} \cdot x $ for all $ v \in V_{even} $ and $ x \in H_{alg} $ with $ n(x) = n \geq 1 $. 
The case $ n = 1 $ reduces to the product on $ V \otimes \rho $, so we can assume that $ n \geq 2 $.
If $ x = (w \otimes y)\dip{m} $ is a divided power of a primitive element $ w \otimes y \in V_{even} \otimes P(A) $, then we let 
$ \iota(v)\dip{n} \cdot x = (\iota(v)(w \otimes y))\dip{m} $. We can then extend it to general $ x \in H_{alg} $ by using Hopf ring distributivity, 
because $ H_{alg} $ is generated under $ \odot $ by the divided powers of its primitive elements.
Explicitly, every element of $ H_{alg} $ is a linear combination of products $ \bigodot_{i=1}^r x_i\dip{m_i} $ for some $ x_i = (w_i \otimes y_i) \in V_{even} \otimes P(A) $ and $ m_i \geq 1 $. We let
\[
\iota(v)\dip{\sum_i n(x_i)m_i} \cdot \bigodot_i x_i\dip{m_i} = \bigodot_i (\iota(v) (w_i \otimes y_i))\dip{m_i}
\]
and we extend to all $ H_{alg} $ by linearity.
\begin{proposition} \label{prop:algebraic cup product}
	There is a unique component super-Hopf ring structure on $ H_{alg} $ such that the multiplication by $ \iota(v)\dip{n} $ for $ v \in V_{even} $ has the form defined above, and satisfying the five conditions of Definition~\ref{def: H-alg}.
\end{proposition}
\begin{proof}
	Let $ v,w \in V_{even} $ and $ y,z \in P(A) $ be tri-homogeneous elements.  Note that, by the definition of the multiplication by divided powers of elements of $ V_{even} $ above, we must have $ v \otimes y = \iota(v)\dip{n(y)} \cdot \pi(y) $ and $ w \otimes z = \iota(w)\dip{n(z)} \cdot \pi(z) $.
	Hence, if $ H_{alg} $ satisfies the third and fourth conditions we must have
	\begin{align*}
	(v \otimes y)\dip{a} \cdot (w \otimes z)\dip{b} &= (\iota(v)\dip{n(y)} \cdot \pi(y))\dip{a} \cdot (\iota(w)\dip{n(z)} \cdot \pi(z))\dip{b} \\
	&= (-1)^{t(y)d(w)n(z)ab} \iota(v)\dip{n(y)a} \cdot \iota(w)\dip{n(z)b} \pi(y)\dip{a} \pi(z)\dip{b} \\
	&= (-1)^{t(y)d(w)n(z)ab} \iota(vw)\dip{n(y)a} \cdot \pi(y)\dip{a} \cdot \pi(z)\dip{b}
	\end{align*}
	for all $ a,b $ such that $ n(y)a = n(z)b $.
	Moreover, if $ \pi $ is a morphisms of Hopf rings and of divided powers structure, $ \pi(y)\dip{a} \cdot \pi(z)\dip{b} = \pi(y\dip{a} \cdot z\dip{b}) $.
	In conclusion
	\[
	(v \otimes y)\dip{a} \cdot (w \otimes z)\dip{b} = (-1)^{t(y)d(w)n(z)ab} \iota(vw)\dip{n(y)a} \cdot \pi(y\dip{a} \cdot z\dip{b}).
	\]
	Consequently, the conditions above uniquely determine the values of $ \cdot $ on $ \odot $-indecomposables. Therefore, if an extension of $ \cdot $ exists, it is necessarily unique.
	
	To prove existence, we only need to check that the formula above provides a well-defined bilinear product satisfying Hopf ring distributivity. 
	This is straightforward.  For instance, one can fix bases of $ A $ and $ V $, define the $ \cdot $ product as above on the induced basis, 
	extend it to all $ H_{alg} $ by bilinearity and directly check the axioms of Hopf rings using the basis.
\end{proof}

We now prove that $ H_{alg} $ is our desired object.

\begin{proposition}
	$ H_{alg} $ is universal among the super-Hopf rigs with additive divided powers satisfying the five conditions of Definition~\ref{def: H-alg}.
\end{proposition}
\begin{proof}
	Let $ B $ be a super-Hopf rig with additive divided powers with maps $ \pi_B \colon A \to B $, $ \iota_B \colon V \otimes \rho \to B $ satisfying our desired hypotheses.
	$ \dpha $ and $ \_^{pr} $ are left adjoints of the forgetful functor from divided powers Hopf algebras to coalgbras and the primitives $ P \colon \Rcoalg \to \Rmod $, respectively.
	Therefore, a morphism of divided powers Hopf algebras $ f \colon H_{alg} \to B $ such that $ f \pi = \pi_B $ and $ f \iota = \iota_V $ is uniquely determined by its restriction to $ V_{even} \otimes P(A) $. If, in addition, $ f $ is a super-Hopf rig morphism, then
	\[
		f(v \otimes y) = f(\iota(v)\dip{n(y)} \cdot \pi(y)) = \iota_B(v)\dip{n(y)} \cdot \pi_B(y) 
	\]
	for all $ v \in V_{even} $ and $ y \in P(A) $. Therefore, such a morphism is unique (if it exists).
	
	To prove existence, we let $ f \colon H_{alg} \to B $ be the divided powers super-Hopf algebra morphism adjoint of the linear map $ v \otimes y \in V_{even} \otimes P(A) \mapsto \iota_B(v) \cdot \pi_B(y) \in P(B) $.
	We only need to check that $ f $ preserves the $ \cdot $ product, and by Hopf ring distributivity it is enough to 
prove this on $ \odot $-indecomposables, that is divided powers of elements of $ P(H_{alg}) = V_{even} \otimes P(A) $. 
By our construction of the $ \cdot $ product in $ H_{alg} $, we immediately see that $ f(\iota(v)\dip{a} \cdot x) = \iota_B(v)\dip{a} \cdot f(x) $ 
for all $ v \in V_{even} $ and $ x \in H_{alg} $ and $ a \geq 2 $. Moreover, since $ \pi_B $ is a super-Hopf rig morphism, we have that
	\begin{align*}
		f \left((v \otimes y)\dip{a} \cdot (w \otimes z)\dip{b}\right) &= (-1)^{t(y)d(w)n(z)ab} f \left(\iota(vw)\dip{n(y)a} \cdot \pi(y\dip{a} \cdot z\dip{b}) \right) \\
		&= (-1)^{t(y)d(w)n(z)ab} \iota_B(vw)\dip{n(y)a} \cdot \pi_B(y\dip{a} \cdot z\dip{b}) \\
		&= \left( \iota_B(v)\dip{n(y)} \cdot \pi_B(y) \right)\dip{a} \cdot \left( \iota_B(w) \dip{n(w)} \cdot \pi_B(z) \right)\dip{b} = f(v \otimes y) \cdot f(w \otimes z).
	\end{align*}
\end{proof}

Since we want to specialize to the case $ A = \hdp(S^0) $ and $ V = H^*(X; \mathbb{F}_p) $ and as shown in Remark~\ref{bad-example}
the relation between cup 
product and divided powers in $ \hdp(S^0) $ is complicated, it is preferable to rephrase the construction above in terms of indecomposables.
In a Hopf algebra with divided powers, the indecomposables are elements of the form $ x\dip{p^k} $ with $ x $ primitive and $ k \geq 0 $.
Consequently, as a graded bicommutative Hopf algebra alone, $ \dpha{(V \otimes P(A)^{pr})} $ is generated by
\[
	\{ (v \otimes y)\dip{p^k}\}_{v \in V, y \in P(A), k \geq 0} = V \otimes Q(A).
\]
We obtain an isomorphism of indecomposables $ V_{even} \otimes Q(A) \cong QH_{alg} $ given by 
$ v \otimes y \mapsto v\dip{n(y)} \cdot y $ for all $ v \in V $ and $ y \in A $, and this realizes $ H_{alg} $, under the product $ \odot $ alone, 
as the free graded commutative algebra generated by $ V_{even} \otimes Q(A) $, quotiented by the image of the 
Frobenius and the ideal generated by the relation $ (5) $.
As a result, any pair of additive bases for $ V $ and $ Q(A) $ induce an additive basis for $ H_{alg} $. 

In summary, we have the following.

\begin{lemma} \label{lem:algebraic basis}
	Let $ e_0 $ (respectively $ e_1 $) in $ \rho $ be a non-zero element in the constant representation (respectively sign representation) addend of $ \rho $.
	Let $ Q(A) $ be the space of $ \odot $-indecomposables of $ A $. Let $ \mathcal{B}_V $ and $ \mathcal{B}_A $ be totally ordered additive bases of $ V_{even} $ and $ Q(A) $ respectively. Define $ \mathcal{B}_V^{\rho} $ as the set of elements $ \{ v \otimes e_i \}_{v \in \mathcal{B}_V, i \in \{0,1\}} $. Order $ \mathcal{B}_V $ in any way (e.g. lexicographically).
	Let $ \mathcal{B}$ be the set of elements of the form $ \bigodot_{i=1}^r \left(\iota(v_i)\dip{n(y_i)} \cdot \pi(y_i) \right) \in H_{alg} $ with $ v_1 \leq \dots \leq v_r \in \mathcal{B}_V^{\rho} $, $ y_1 \leq  \dots \leq y_r \in \mathcal{B}_A $, such that:
	\begin{itemize}
		\item the sign degree of all the $ \odot $-factors is the same
		\item the multiplicity of every $ \odot $-factor of even total degree is at most $ p-1 $,
		\item and the multiplicity of every $ \odot $-factor of odd total degree is at most $ 1 $.
	\end{itemize}
Then $ \mathcal{B}$ is a basis for $ H_{alg} $ as a $ \mathbb{F}_p $-vector space.
\end{lemma}

%%% Local Variables:
%%% TeX-master: "DX-QX.tex"
%%% End:

%% file: DX-QXSection3.tex
% !TEX root = DX-QX.tex

\section{The homology and cohomology of  symmetric groups with twisted coefficients} \label{sec:symmetric group}

Homology of extended powers and free infinite loop spaces are algebras over the  Kudo-Araki-Dyer-Lashof  algebra, an algebra of homology operations that constitute the basic building blocks of the homology of the extended powers 
of a point -- that is, the homology of symmetric groups.  This structure is well-understood by the work of Cohen--Lada--May \cite{May-Cohen}.
Our strategy in this paper is show that our descriptions which build on cup product structure pair perfectly with their descriptions.
%We focus on this section of the cohomology of symmetric groups, as calculated
%in \cite{Sinha:12} and  \cite{Guerra:17}.  
But for odd primes the cleanest descriptions of cohomology require considering coefficients in $ \rho $, 
the $ \mathbb{F}_p[\si_n] $-representation which is the sum of the trivial and sign representations, introduced in the previous section.

Thus, this section is divided in two parts. In the first subsection we provide a description of the homology of $ \tilde{D}X $ with coefficients in $ \rho $,
 in terms of Kudo-Araki-Dyer-Lashof operations. We do not make any claim of originality here, as we only review classical results by 
 Cohen--Lada--May \cite{May-Cohen} for the trivial summand and their recent extension to sign-twisted coefficients by Bernard \cite{Bernard-thesis}.
In the second subsection we dualize these results to obtain a description of 
$ \bigoplus_{n \geq 0} H^*(\si_n; \rho) $ as a Hopf ring, and prove Theorem \ref{m2}.
While the summand of sign degree zero is known by our previous work, namely
\cite{Sinha:12} and  \cite{Guerra:17}, the description of the sign degree one summand
is new.
%and relies on the computation of the dual of the space of Bernard's twisted Kudo-Araki-Dyer-Lashof operations as a module over the dual of the classical Dyer-Lashof algebra.

\subsection{The homology of extended powers in terms of KADL operations}

Since we work with field coefficients in a setting which is finite dimensional in each grading, 
the homology of $ DX $ with local coefficients given by the 
$ \si_n $-representation $ A_n $ of Definition \ref{def:basic objects} is isomorphic to the bigraded linear dual of 
$ \bigoplus_{n \geq 0} H^*(D_n(X); A_n) = \hdp(X) $. Hence it will be a Hopf co-ring, endowed with two coproducts 
$ \Delta_{\odot} $, $ \Delta_{\cdot} $, and a product $ * $ dual to $ \odot $, $ \cdot $, and $ \Delta $ respectively, 
satisfying all the axioms of a Hopf ring with the directions of all morphisms reversed.

The  product $ * \colon H_i(D_nX) \otimes H_j(D_m X) \to H_{i+j} (D_{n+m} X) $. also has a group-homology interpretation. 
Let $R_{\si_n}$ be a resolution of $\mathbb{F}_p$ as a $\mathbb{F}_p[\si_n]$-module, so that by Proposition~\ref{prop:may}
the homology of $R_{\si_n} \otimes_{\si_n}  {\widetilde{H}_*(X)}^{\otimes{n}}$ is that of $D_n X$.
The product $ * $ is induced, up to sign, by the tensor product of the map of resolutions  $R_{\si_n} \otimes R_{\si_m} \to R_{\si_{n+m}}$ induced by the
inclusion $\si_n \times \si_m \hookrightarrow \si_{n+m}$ with the isomorphism   ${\widetilde{H}_*(X)}^{\otimes{ n}} \otimes 
{\widetilde{H}_*(X)}^{\otimes{ m}}  \cong {\widetilde{H}_*(X)}^{\otimes{ n+m}}$.
Geometrically, the product on $ D(X) $ and $C(X)$ can be defined through an embedding 
$\R^\infty \coprod \R^\infty \to \R^\infty$, through which one can take the image of the union 
of configurations with labels in $X$. 
 Since it arises from a homotopy commutative multiplication
on spaces  we also call this product the Pontrjagin product.

Since
the coproduct and the transfer product on cohomology
form a bialgebra,  then for example if the evaluations of cohomology classes
$ a_{i}$ on homology $x_{j} $ are Kronecker -- that is, $\langle a_i, x_j  \rangle  = \delta_{i,j}$ -- 
and the $a_i$ are primitive, 
then the evaluation $\langle \bigodot_i a_i, *_i x_i \rangle$ will be one.   But if some classes are repeated then coefficients
are introduced, as for example $ \langle \bigodot_n a, x^{*n} \rangle = n!$.
The divided powers operations in cohomology ``fill in'' to produce duals in these cases, as 
$\langle a\dip{n}, x^{*n} \rangle$ is one, as it is given by the tensor product at the chain and cochain level.

%To do so, we show that both approaches give rise to additive bases which
%respect a partition decomposition, and we show the pairing on the $\pi$-summand coincides with the 
%tensor product of the pairing between homology and cohomology of symmetric groups with trivial coefficients.

%In part I of Cohen--Lada--May \cite{May-Cohen} the  homology of $C(X)$
%is computed as a free algebra over Kudo-Araki-Dyer-Lashof (KADL) operations.

Just as the inclusions $\si_n \times \si_m \hookrightarrow \si_{n+m}$ give rise to the Pontrjagin product, the inclusions of wreath products 
$\si_n \wr \si_k \hookrightarrow \si_{nk}$ give rise to the algebraic Kudo-Araki-Dyer-Lashof operations.

\begin{definition}[compare \cite{Bernard-thesis}]
Use $R_{\si_n} \otimes (R_{\si_k})^{\otimes n}$ as a resolution for $\si_n \wr \si_k \hookrightarrow \si_{nk}$.
Let $ W_* $ be the standard minimal
 $ \mathbb{F}_p $ resolution of $C_p$, the cyclic group of order $ p $ and $ e_i$ a generator for $ W_i$. 
 To every chain $ e \in W_* $ we associate an operation $q(e)^{\#} : R_{\si_n}  \to R_{\si_{pn}}$
 as sending a chain $c$ the image of $e \otimes c^p$ under the map of resolutions induced
 by the inclusion of $C_p \wr \si_n \hookrightarrow \si_{np}$.
 
 We consider the linear morphisms
 $$ q(e) \colon H_*(\si_n; {\widetilde{H}_*(X)}^{\otimes{n}} \otimes \rho ) \to H_*(\si_{pn};  {\widetilde{H}_*(X)}^{\otimes{ pn}} \otimes \rho) $$
 induced by the map $$c \otimes (x_1 \otimes \cdots \otimes  x_n) \mapsto q(e)^{\#} (c) \otimes 
 (x_1 \otimes \cdots \otimes x_n)^{\otimes p}.$$
 
 The  {\bf Kudo-Araki-Dyer-Lashof (KADL)} operations are linear morphisms 
$$ q_i = q(e_i) \colon H_m(\si_n; {\widetilde{H}_*(X)}^{\otimes{n}} \otimes \rho ) \to H_{pm + i(p-1)}(\si_{pn};  {\widetilde{H}_*(X)}^{\otimes{ pn}} \otimes \rho). $$
%induced by the map $$c \otimes (x_1 \otimes \cdots \otimes  x_n) \mapsto q^{\#}_i (c) \otimes 
%(x_1 \otimes \cdots \otimes x_n)^{\otimes p}.$$
\end{definition}

Geometrically, $q_i$ on a homology class of $X$ (the $n = 1$ case) when $i$ is odd
is represented in $C X$ by a family of configurations with labels where the $p$ points in the configuration
are on the $i$-sphere related by the action of a $p$-th root of unity, with labels  all in the same cycle on $X$.
 %defined as
% the image of $ e_{i(p-1)} \otimes x^{\otimes^p} \in W_* \otimes H_*(C(X))^{\otimes p} $ under the multiplication
 %defined by the $E_\infty$ structure on $C(X)$.  
%For $ q_i $ to act nontrivially of $ \iota_d $, the parities of $ i $ and $ d $ must be the same. 
%Consequently, the space of operations 
%that we need to take into consideration differs according to $ d $ being even or odd.  We  treat  both cases explicitly.

For $p=2$ we consider $r$-tuples $(i_1, \cdots, i_r)$.  For odd primes, 
let $ \beta $ be the Bockstein homomorphism.  We associate to an $ 2r $-tuple 
$ I = (\varepsilon_1, i_1, \dots, \varepsilon_r,i_r) $, with $ \varepsilon_k \in \{0,1\} $ and all $ i_k \geq 0  $,
the homology operation $ q_I = \beta^{\varepsilon_1} q_{i_1} \circ \dots \circ \beta^{\varepsilon_r} q_{i_r} $. 
We say that $ I $ is {\bf admissible} if $ i_k \leq i_{k+1} - \varepsilon_{k+1} $ (respectively $ i_k \leq i_{k+1}$ when $p=2$)
for all $ 1 \leq k < r $ and, in case $ p >2 $, $ i_k \equiv i_{k+1}-\varepsilon_{k+1} \mod 2 \, \forall 1 \leq k < r $. 
We say that $ I $ is {\bf strongly admissible} if it is admissible and $ i_1 > 0 $. This is a reformulation of Bernard's admissibility conditions,
in Section 7 of  \cite{Bernard-thesis}, with the lower indices notation.

\begin{theorem}[Cohen--Lada--May \cite{May-Cohen}]\label{teo:homology basis constant}
The mod-$p$ homology of $CX$ is a free graded commutative algebra with respect to the Pontrjagin product, 
generated by $q_I(x)$ where $ x $ ranges over a graded basis for the reduced homology of $ X $ and 
$I$ ranges over strongly admissible sequences, with the additional requirement that for $p$ odd  $ i_r $ and the homological degree of $ x $ have the same parity, where $I = (\varepsilon_1, i_1, \dots, \varepsilon_r,i_r) $ %modified
\end{theorem}

Similarly, the homology of $ CX $ with twisted coefficients given by the mod $ p $ sign representation for $p$ odd can be computed as an algebra using the twisted versions of Dyer-Lashof operations.
\begin{theorem}[Bernard \cite{Bernard-thesis}]\label{teo:homology basis twisted}
For $p$ odd, the homology of $ \tilde{D}(X) = C(X_+) $ with coefficients in $ \sgn $, the mod $ p $ sign representation, 
is the free graded commutative algebra %removed, t was something else? (graded by total degree $ t $)
with respect to the homology product generated by $ q_I(x) $, where $ x $ ranges over a graded basis for the homology of $ X $ and 
$I $ ranges over strongly admissible sequences such that $ i_r $ and the homological degree of $ x $ have different parity, with
$I = (\varepsilon_1, i_1, \dots, \varepsilon_r,i_r) $ %modified
\end{theorem}

\begin{remark}
Admissible, but not strongly admissible, KADL operations on $ CX $ can be retrieved by the identity $ Q_0(x) = x^{*p} $. 
Iterated non-admissible sequences of KADL operations can be computed by means of Adem relations.
Both in \cite{May-Cohen} and in \cite{Bernard-thesis}, an ``upper indices'' notation is used, 
because homological degrees behave better and Adem relations have a better form. We use a ``lower indices'' notation, 
because it makes the argument for the computation of the dual module and the construction of Hopf ring generators more transparent. 
The two conventions differ only by coefficients and reindexing, and are essentially equivalent. 
We refer to the two papers cited above for the precise relations between the two notations.
\end{remark}

%For the next subsection we also need a description of the homological coproduct $ \Delta_\cdot $ dual to the cup product in cohomology.
\begin{definition}
Let $ p $ be an odd prime and let $ k \geq 0 $. 
Define $ \mathcal{R}^t_k $ be the $ \mathbb{F}_p $-vector space spanned by the admissible KADL operations 
$ q_I $ where $ I =(\varepsilon_1,i_1,\dots,\varepsilon_k,i_k) $ has length $ 2k $.   Let $ \mathcal{R}^t = \bigoplus_{k \geq 0} \mathcal{R}^t_k $.

For $ p = 2 $, let $ \mathcal{R}_k $ be the $ \mathbb{F}_2 $-vector space spanned by the admissible KADL operations $ q_I $, 
where $ I =(i_1,\dots,i_k) $ has length $ k $.

\end{definition}

In light of our main application, we consider an admissible operation $ q_I $ as acting on a $ 0 $-dimensional class. 
In this case,  the subspace $ \mathcal{R}_k \subseteq \mathcal{R}^t_k $ spanned by $ q_I $ such that $ i_k $ is even corresponds to the classical 
untwisted KADL operations, and the subspace $ \mathcal{R}'_k \subseteq \mathcal{R}^t_k $ spanned by $ q_I $ such that $ i_k $ is odd 
corresponds to the twisted KADL operations. Precisely, by sending $ q_I $ to $ q_I(\iota) $, for $ \iota \neq 0 \in \overline{H}_0(S^0) $, we identify 
$ \mathcal{R}_k $ as a subspace of $ H_*(\si_{p^k}; \mathbb{F}_p) \subseteq H_*(D(S^0); \mathbb{F}_p) $ and 
$ \mathcal{R}'_k $ as a subspace of $ H_*(\si_{p^k}; \sgn) \subseteq H_*(D(S^0); \sgn) $. 
The component of an operation $ q_I \in \mathcal{R}^t_k $ is $ n(q_I) = p^k $, and its homological degree is 
$ d(q_I) = \sum_{j=0}^{k-1} p^j \left( i_{j+1}(p-1) - \varepsilon_{j+1} \right) $.

As a consequence of Theorem \ref{teo:homology basis constant}, $ H_*(DS^0; \mathbb{F}_p) = H_*(CS^0; \mathbb{F}_p) $ has a basis given by 
Pontryagin monomials in strongly admissible KADL operations in $ \mathcal{R} = \bigsqcup_k \mathcal{R}_k $. 
We denote this basis with $ \mathcal{B} $, and we call it {\bf Nakaoka  basis}. Similarly, by Theorem \ref{teo:homology basis twisted}, 
$ H_*(DS^0; \sgn) $ has a basis given by $ * $-monomials in strongly admissible KADL operations in $ \mathcal{R}' = \bigsqcup_k \mathcal{R}'_k $, 
that we denote with $ \mathcal{B}_{\sgn} $. In both $ \mathcal{B} $ and $ \mathcal{B}_{\sgn} $, operations with odd  %removed: total
degree must not appear twice in the same monomial.

For $p=2$,  $ \mathcal{R} = \bigsqcup_k \mathcal{R}_k $ embeds as a subspace of 
$ H_*(CS^0; \mathbb{F}_2) $ by sending $ q_I $ to $ q_I(\iota) $, whose component is $ n(q_I) = 2^k $ and whose homological degree is 
$ d(q_I) = \sum_{j=0}^{k-1} 2^j q_{j+1} $. Theorem \ref{teo:homology basis constant} provides a Nakaoka monomial basis as above.

%A possible starting point for the study of $ A_{\sgn} $ can be studied by first analyzing the structure of the cohomology classes that are linear duals to a set of algebra generators in homology (the $ q_I $ defined above), and then work out the rest by exploiting some relation between $ A $ and $ A_{\sgn} $. For this reason, we give now a description of the dual of the vector space generated by the $ q_I $.

%In what follows, let $ p $ be a prime number bigger than $ 2 $. Let $ \mathcal{R} $ be the image of the classical Dyer-Lashof algebra in $ A = \bigoplus_{n \geq 0} H_*(\si_n; \mathbb{F}_p) $ via the homomorphism defined as the evaluation on the class in $ H_0(S^0; \mathbb{F}_p) $ corresponding to the point different from the basepoint.

\subsection{Cohomology of  symmetric groups with coefficients in the sign representation mod-$p$} \label{sec:sgn}

To complete the description of $ H^*(DX; \mathbb{F}_p) $ as a Hopf ring in the odd primary case, we must first analyze 
the cohomology of the symmetric groups with coefficients in the sign representation, mod $p$. 
When we consider the cohomology of $X^n_{h \si_n}$,  through say
cellular chain models,  a $d$-cell in $X$ gives rise to a chain of dimension $dn$, and
a permutation $ \sigma \in \si_n $ acts on this chain as multiplication by $ (-1)^{d \sgn(\sigma)} $.  So for example 
$ H^{* + nd}(D_n(S^d); \mathbb{F}_p) $ has a summand isomorphic to $ H^*(\si_n; \sgn^d) $.
%when $ \sgn^d $ is the mod $ p $ trivial $ \si_n $-representation if $ d $ is even, and the $\si_n $-representation 
%$ \mathbb{F}_p \otimes \sgn $ if $ d $ is odd.  
While additive structures are isomorphic in the $d$ even and odd cases,
the product structures here differ significantly.

A first step in understanding a Hopf ring is through its $\odot$-indecomposables, which by Hopf ring distributivity 
 form a ring under the $\cdot$-product. In the bigraded component setting, this ring
is a coproduct of the $\odot$-indecomposables on each component.  
For the cohomology of symmetric groups these rings of $\odot$-indecomposables
form the foundation of our previous calculations 
\cite{Sinha:12, Guerra:17}, but the calculation of 
these rings go back to the work of Cohen--Lada--May \cite{May-Cohen}, where they 
were called $\mathcal{R}_k^*$.  % We first calculate the $\odot$-indecomposables of $B$, which 
%decomposes as $\mathcal{R}_k^* \oplus {\mathcal{R}'}_k^*$.

Thus the focus of this section is Lemma \ref{lem:dual KADL}, which consists of the calculation of the dual algebra of 
$ \mathcal{R} \bigoplus \mathcal{R}' $. While this computation for twisted operations is similar to the untwisted case, it has not been computed previously.
Our calculation here is a straightforward generalization of arguments in Section I.3  of \cite{May-Cohen}, and  as suggested in that text we
use lower index notation.

%Recall from \cite{Cohen-May-Taylor} that $ C(S^d) $ is stably homotopy equivalent to 
%$ D(S^d) = \bigvee_{n \geq 0} D_n(S^d) = \bigvee_{n \geq 0} (S^d)^n_{h \si_n} $.% In particular, they have the same homology.
%Chose a $ \si_n $-equivariant CW model for $ E(\si_n) $ and consider the standard cellular structure on $ S^d $ with one $ 0 $-cell and one $ d $-cell. This induces naturally a CW structure on $ D_n(S^d) $ whose cellular chain complex $ C^{CW}_*(D_nS^d) $ is $ CW_*^{CW}(E(\si_n)) \otimes_{\si_n} \tilde{C}^{CW}_*(S^d)^{\otimes^n} $, where $ \tilde{C}^{CW}_*(S^d) $ is the cellular chain complex of $ S^d $ without the $ 0 $-cell.
%Note that $ \tilde{C}^{CW}_*(S^d)^{\otimes^n} $ is the free $ \mathbb{Z} $-module over a single class $ \iota_{dn} $ of dimension $ dn $, that corresponds to the fundamental class of $ (S^d)^{\wedge^n} \cong S^{dn} $. Note that a permutation $ \sigma \in \si_n $ acts on $ \iota_{dn} $ as the multiplication by $ (-1)^{d \sgn(\sigma)} $, thus $ H_*(D_n(S^d); \mathbb{F}_p) $ is isomorphic, up to lowering the degree by $ dn $, to $ H_*(\si_n; \sgn^d) $, when $ \sgn^d $ is the mod $ p $ trivial $ \si_n $-representation if $ d $ is even, and the $\si_n $-representation $ \mathbb{F}_P \otimes \sgn $ if $ d $ is odd. In particular, for $ d = 1 $ (actually for any odd $ d $), we obtain the homology group $ A_{\sgn} = \bigoplus_{n \geq 0} H_*(\si_n; \sgn \otimes \mathbb{F}_p) $.

 %Many of the statements in this section appear there in some form.
 
%Recall that the homology of $ C(S^d) $ is a rank-one algebra over the Dyer-Lashof algebra.  

Let $ \mathcal{R}_k^* $ and $ {\mathcal{R}'}_k^* $ be the spans of linear duals 
$ q_I^\vee $ of the operations $ q_I $ of length $ k $ respectively in 
$ H^*(\si_{p^k}; \mathbb{F}_p) $ and $ H^*(\si_{p^k}; \sgn \otimes \mathbb{F}_p) $  with respect to the Nakaoka bases.
Because of the naturality of twisted KADL operations, the proof of Theorem 4.13 in \cite{Sinha:12} also shows that the subspace of primitives 
with respect to the coproduct dual to the transfer product $ P(\bigoplus_{n\geq 0} H_*(\si_n; A_n)) $ is $ \mathcal{R}^t = \mathcal{R}\bigoplus \mathcal{R}'$.
Because the transfer product forms a bialgebra with the coproduct dual to $*$, the dual classes $ \{q_I^\vee\}$ for $I$ admissible are a basis for our 
Hopf ring indecomposables. Moreover, each component of this module of $ \odot $-indecomposables is a graded commutative ring under cup product.
Recall that %there is a coproduct $ \psi \colon \mathcal{R} \bigoplus \mathcal{R}' \to (\mathcal{R} \bigoplus \mathcal{R}')^{\otimes^2} $ defined as follows. By construction every (twisted or untwisted) KADL operation $ \beta^{\varepsilon}q_i $ arises from an element $ e \in H_{i(p-1)-\varepsilon}(\Pi_p; \mathbb{F}_p) $ in the mod $ p $ homology of the cyclic group $ \Pi_p $ of order $ p $ as the map, natural for $ E_\infty $-spaces $ X $, %put "spaces" instead of "algebra"
%\[
%H_*(X) \stackrel{\_^{\times p}}{\rightarrow} H_*(X^p) \stackrel{e \otimes \_}{\rightarrow} H_*(E(\Pi_p) \times_{\Pi_p} X^p) \rightarrow H_*(X).
%\]
%Moreover every element $ e \in H_*(\Pi_p; \mathbb{F}_p) $ provide a linear combination of KADL operations in the same way, that we temporarily denote $ q(e) $. 
the coproduct $ \Delta \colon H_*(C_p) \to H_*(C_p) \otimes H_*(C_p) $ dual to cup product in the homology of the cyclic group $ C_p $ induces a unique coproduct $ \psi $ on $ \mathcal{R}^t $ such that $ \psi(q(e)) = (q \otimes q)(\Delta(e)) $. The coproduct $ \psi $ has the form
$\psi(q_I) = \sum_i \eta_i q_{J_i} \otimes q_{K_i}$, which encodes a Cartan formula $  q_I(x \otimes y) = \sum_i \eta_i q_{J_i}(x) \otimes q_{K_i}(y)$.
We refer to \cite{May-Cohen} for the precise construction and related calculations.

While the isomorphism $ H^*(\si_n; \sgn) \cong H^*(D_nS^1; \mathbb{F}_p) $ guides us, we need to make a finer distinction for this calculation.
While the transfer product and the coproduct are preserved by this isomorphism, the same does not happen for the cup product.
By letting an operation $ q_I $ act on $ [S^1] $ in $ H_*(DS^1; \mathbb{F}_p) $ and composing with that isomorphism, 
we can nevertheless embed $ \mathcal{R}'_k $ into $ H_*(\si_{p^k};\sgn) $. 
Combining this with what we know about the homology of the symmetric group with trivial coefficients, 
we obtain a map $ \kappa \colon \mathcal{R} \oplus \mathcal{R}' \to H_*(\si_n; \mathbb{F}_p \oplus \sgn) $. 
By duality this provides a map of $ \mathbb{F}_p $-modules 
$ \kappa^* \colon H^*(\si_{p^k}; \mathbb{F}_p \oplus \sgn) \to \mathcal{R}_k^* \oplus {\mathcal{R}'_k}^* $. We compare the products in those rings.
%To simplify our notation, in what follows we let $ \mathcal{R}^t_k = \mathcal{R}_k \oplus \mathcal{R}'_k $ 

\begin{proposition} \label{prop:cup product sign}
%Let $ \sgn^\varepsilon $ be equal to the trivial representation if $ \varepsilon = 0 $, to the sign representation if $ \varepsilon = 1 $. 
If $ x \in H^d(\si_{p^k}; \sgn^\varepsilon) $ and $ y \in H^{d'}(\si_{p^k}; \sgn^{\varepsilon'}) $, then
	\[
	\psi^*(\kappa^*(x) \otimes \kappa^*(y)) = (-1)^{\varepsilon ( d' + \frac{p-1}{2}k\varepsilon')} \kappa^*( x \cdot y ).
	\]
\end{proposition}
\begin{proof}
	Let $ X $ be a space, and let $ n = p^k $. Fix a graded basis $ {B}$ of the homology of $ X $. 
	We observe that the following diagram commutes, where $ d $ denotes diagonal maps, $ \tau $ the obvious shuffle map, and $ q $ is the obvious quotient map.
	\begin{figure}[H]
		\begin{center}
			\begin{tikzcd}
				\tilde{D}_n(X) \arrow{r}{d_{D_n(X)}} \arrow{d}{d_{E(\si_n)} \times_{\si_n} d_{X^n}} & \tilde{D}_n(X)^{2} \arrow{d}{\tau} \\
				(E(\si_n) \times E(\si_n)) \times_{\si_n} (X^{n} \times X^{n}) \arrow{r}{q} & (E(\si_n) \times E(\si_n)) \times_{(\si_n \times \si_n)} X^{2n}
			\end{tikzcd}
		\end{center}
	\end{figure}
	Consider $ q_I(x) \in H_*(\tilde{D}_n(X)) $, where $ x \in H_*(X) $. 
	Assume that $ \Delta(x) = \sum_i \lambda_i x_i' \otimes x_i'' $ in $ H_*(X) $, for some $ x_i',x_i'' \in {B} $. 
	Write $ \psi(q_I) = \sum_l \eta_l q_{J_l'} \otimes q_{J_l''} $. 
	Let $ \varphi $ be the coproduct in homology dual to the cup product of $ H^*(\si_n; \mathbb{F}_p \oplus \sgn) $.
	
	Taking homology and keeping track of the image of $ q_I(x) $ in upper path in the diagram above act on, we obtain 
	\begin{equation*}
		\begin{split}
			\tau_* \circ {d_{\tilde{D}_n(X)}}_* (q_I(x)) &= \tau_* \left( \sum_{i,k} \lambda_i \eta_l q_{J_l'}(x_i') \otimes q_{J_l''}(x_i'') \right)\\
			&= \sum_{i,l} (-1)^{|q_{J_{l''}}| |x_i'|} \lambda_i \eta_l q_{J_l'} \otimes q_{J_l''} \otimes (x_i')^{\otimes^n} \otimes (x_i'')^{\otimes^n}.
		\end{split}
	\end{equation*}
	Similarly, for the lower path, we have
	\[
	(d_{E(\si_n)} \times_{\si_n} d_{X^{n}})_*(q_I(x)) = \sum_{i_1,\dots,i_n} \prod_{j=1}^n \lambda_{i_j} (-1)^{\sum_{1 \leq j < l \leq n} |x_{i_j}''| |x_{i_l}'|} {d_{E(\si_n)}}_* (q_I) \otimes_{\si_n \times \si_n} \left( \bigotimes_{j=1}^n x_{i_j}' \otimes \bigotimes_{j=1}^n x_{i_j}'' \right). \\
	\]
	We can split this last summation as the sum over the set of $ n $-tuples where all $i_k$ are equal to each other (say $i$) and the sum over 
the set on $n$-tuples which  are not all equal.
	The first part can be rewritten as
	\[
	\sum_i \lambda_{i} (-1)^{\frac{n(n-1)}{2} |x_{i}'| |x_{i}''|} {d_{E(\si_n)}}_* (q_I) \otimes ({x_{i}'}^{\otimes^n} \otimes {x_{i}''}^{\otimes^n}),
	\]
	which maps to $ \sum_i \lambda_i (-1)^{\frac{n(n-1)}{2} |x_{i}'| |x_{i}''|} \varphi(q_I) \otimes {x_i'}^{\otimes^n} \otimes {x_i''}^{\otimes^n} $ under $ q $.
We claim that the second part is zero. 
Recall that there is an isomorphism 
$ H_*((E(\si_n) \times E(\si_n)) \times_{\si_n} X^{2n}; \mathbb{F}_p) \cong H_*(\si_n; ({H}_*(X; \mathbb{F}_p)^{\otimes^n})^{\otimes^2}) $. 
The $ \si_n $-subrepresentation of $ ({H}_*(X; \mathbb{F}_p)^{\otimes^n})^{\otimes^2} $ generated by 
$ (x_{i_1}' \otimes \dots x_{i_n}') \otimes (x_{i_1}'' \otimes \dots \otimes x_{i_n}'') $ with 
$ i_1,\dots,i_n $ not all equal is isomorphic to the induced $ \si_n $-representation of a $ G $-representation for a Young subgroup $ G \lneq \si_n $. 
The corresponding terms in the sum above amount to $ \tr_G^{\si_n} \rho_G^{\si_n} (q_I) $, 
which is multiplication by the index $ [\si_n:G]$, which is zero modulo $ p $. 
	
	In conclusion, 
		\[
	\sum_i \lambda_i (-1)^{|x_i'| |q_{J_l''}|} \psi(q_I) \otimes (x_i')^{\otimes^n} \otimes (x_i'')^{\otimes^n} = \sum_i \lambda_i (-1)^{\frac{n(n-1)}{2} |x_i'| |x_i''|} \varphi(q_I) \otimes (x_i')^{\otimes^n} \otimes (x_i'')^{\otimes^n}.
	\]
	Since $ n = p^k $, $ \frac{n(n-1)}{2} = \frac{p-1}{2}k \mod 2 $. Hence, the equality above is true for all spaces $ X $ and for all classes $ x \in H_*(X; \mathbb{F}_p) $ if and only if
	\[
	\sum_l \eta_l \kappa(q_{J_l}') \otimes \kappa(q_{J_l}'') = \sum_l \eta_l (-1)^{\varepsilon (|q_{I_l''}|+ \frac{p-1}{2} \varepsilon')} \varphi \kappa(q_I)
	\]
	for $ \varepsilon, \varepsilon' \in \{0,1\} $. Dualizing this we obtain the desired formula.
\end{proof}

To summarize, the coproduct dual to the cup product, when restricted on $ \mathcal{R}^t $, has the form
	\[
	\varphi(q_I) = \sum_{J+K=I} (-1)^{\sum_{l=1}^n \left( \frac{p-1}{2} j_l k_l + \delta_l j_l \right)} q_J \otimes q_K,
	\]
where $ J=(\varepsilon_1,j_1,\dots,j_n)$ and $ K=(\delta_1,k_1,\dots,k_n)$.

The difficulty with this formula is that the sum is over all ways to decompose the sequence $I$, not just admissible sequences.  
Adem relations are needed to make calculations.  The dual algebra is in the end manageable, though to this day (to our
knowledge) even in the trivial coefficient setting
the pairing between the Nakaoka basis on homology and the standard basis on cohomology from analysis such as 
we give below is not known.\\

To calculate the product dual to $\varphi$ in the sign representation setting, we continue  to follow the trivial coefficient treatment
of  \cite[Section I.3]{May-Cohen}.  
We consider sequences of operations with minimal entries. 
\begin{comment}
So we let
\begin{itemize}
\item $ I_{j,k} = (\underbrace{0,\dots,0,}_{2(k-j) } \underbrace{0,2,\dots,0,2}_{j \mbox{ times}}) $, 
so $ q_{I_{j,k}} = q_0^{\cdot^{k-j}} \circ q_{2}^{\cdot^j} $
\item $ J_{j,k} = (\underbrace{0,1,\dots,0,1}_{k-j \mbox{ times}}, 1,2, \underbrace{0,2, \dots, 0,2}_{j-1 \mbox{ times}}) $, 
so $ q_{J_{j,k}} = q_{1}^{\circ^{k-j}} \circ \beta q_{2}^{\circ^j} $
\item $ K_{i,j,k} = ( \underbrace{0,\dots,0}_{2(k-j)}, 
              1,1, \underbrace{0,1, \dots, 0,1}_{j-i-1 \mbox{ times}}, 1,2, \underbrace{0,2, \dots, 0,2}_{i-1 \mbox{ times}}) $,
so $ q_{K_{i,j,k}} = q_0^{\circ^{k-j}} \circ \beta q_{1}^{\circ^{j-i}} \circ \beta q_{2}^{\circ^i} $
\item $ I'_k = (\underbrace{0,1, \dots, 0,1}_{k \mbox{ times}}) $, so $ q_{I'_k} = q_{1}^{\circ^k} $ 
\item $ J'_{j,k} = (\underbrace{0, \dots, 0}_{2(k-j) }, 1,1, \underbrace{0,1, \dots, 0,1}_{j-1 \mbox{ times}}) $, so $ q_{J'_{j,k}} = q_0^{\circ^{k-j}} \circ \beta q_{1}^{\circ^j} $.
\end{itemize}
\end{comment}
%These last two tuples are new, as needed for the sign representation.

Let $ S \subseteq \{1, \dots, k\} $ and define $ I_S[k] $ and $ I'_S[k] $ respectively
as the sets of the admissible $ 2k $-tuples 
$ I =(\varepsilon_1, i_1, \dots, \varepsilon_k, i_k) $ such that $ i_j $ is equal to $ | S \cap \{ 1, \dots, k-j \} | $  (respectively
$ (| S \cap \{ 1, \dots, k-j \} | +1 ) $) mod $ 2 $ and 
$ \varepsilon_j = 1 $ if and only if $ k+1-j \in S $.
Note that $ \{ q_I: I \in I_S[k] \} \subseteq \mathcal{R}_k $ and $ \{ q_I: I \in I'_S[k] \} \subseteq \mathcal{R}'_k $.

For all $ k \in \mathbb{N} $ and $ S \subseteq \{1, \dots, k\} $, we can define a partial order $ \leq $ on $ I_S[k] $ and $ I'_S[k] $ by
comparison of all entries.
%as follows:
%\[
%(\varepsilon_1,i_1,\varepsilon_2,i_2,\dots,\varepsilon_k,i_k) \leq (\varepsilon'_1,i'_1,\varepsilon'_2,i'_2, \dots, \varepsilon'_l,i_k) \Leftrightarrow \forall 1 \leq j \leq k: i_j \leq i'_j 
%\]
This partially ordered set possesses all meets, obtained by taking the minimum on each entry. 
In particular, $ I_S[k] $ and $ I'_S[k] $ have a minimal element, that we denote $ L_{S,k} $ and $ L'_{S,k} $ respectively.

Although we will not need this,  it is straightforward to write down the elements $ L_{S,k} $ and $ L'_{S,k} $ explicitly. 
For instance, $ L_{\{1,2\},3} = (0,0,1,1,1,2) $. 
Up to switching to  upper-index notation, the $ (2k) $-tuples $ L_{S,k} $ coincide with those defined in Section I.3 of \cite{May-Cohen}.

\begin{comment}
For all $ e = 1 \leq e_1 < \dots e_r \leq k $ let:
\begin{equation*}
\begin{split}
L_{e,k} &= \left\{ \begin{array}{ll} K_{e_1,e_2,k} + \dots + K_{e_{r-1},e_r,k} & \mbox{for } r \mbox{ even} \\
K_{e_1, e_2, k} + \dots + K_{e_{r-2},e_{r-1},k} + J_{e_r,k} & \mbox{for } r \mbox{ odd}
\end{array} \right. \\
L'_{e,k} &= \left\{ \begin{array}{ll} I'_k & \mbox{for } e = \varnothing \\
K_{e_1,e_2,k} + \dots + K_{e_{r-3},e_{r-2},k} + J_{e_{r-1},k} + J'_{e_r,k} & \mbox{for } r \mbox{ even} \\
K_{e_1,e_2,k} + \dots + K_{e_{r-2},e_{r-1},k} + J'_{e_r,k} & \mbox{for } r \mbox{ odd}
\end{array} \right.
\end{split}
\end{equation*}
\end{comment}

The following is straightforward combinatorics, following the same argument as Lemmas I.3.3 and I.3.4 in \cite{May-Cohen}.

\begin{lemma} \label{lem:cup coproduct Dyer-Lashof}
For all $ j,k \in \mathbb{N} $ such that $ 1 \leq j \leq k $, let
\[
I_{j,k} = (\underbrace{0,\dots,0,}_{2(k-j)} \underbrace{0,2,\dots,0,2}_{2j}) \in I_{\phi}[k].
\]
The weighted sum %$ (n_1, \dots, n_k) \mapsto L_{e,k} + \sum_{j=1}^k n_j I_{j,k} $ and 
$ (n_1, \dots, n_k) \mapsto L'_{S,k} + \sum_{j=1}^k n_j I_{j,k} $  is a bijections from
  $ \mathbb{N}^k \to I'_{S}[k] $ for all subset $ S $ of $ \{1,\dots,k\} $.
\end{lemma}
%\begin{proof}
%The assertion for $ I_{e,k} $ is the content of Lemmas I.3.3 and I.3.4 in \cite{May-Cohen}.
%Fix $ e $ and $ k $ and write $ L'_{e,k} = (l_1, \dots, l_k) $.
%The assertion for $ I'_{e,k} $ follows from the fact that $ I = (i_1, \dots, i_k) $ belongs to $ I'_{e,k} $ if and only if $ i_j \geq l_j $ for all $ 1 \leq j \leq k $ and the $ k $-tuple $ I - L $ belongs to $ I_\varnothing[k] $.
%This is proved by induction on $ k $:
%\begin{itemize}
%\item The base case $ k = 1 $ is obvious.
%\item Assume $ k > 1 $. If $ e = \varnothing $, then $ i_j $ is odd for all $ 1 \leq j \leq k $, and, in particular, $ i_1 \geq 1 $. Hence $ I \geq L $. Moreover, $ i_1 - l_1 $ is equal to $ 0 $ modulo $ 2 $ and, by application of the induction argument, we see that $ I-L'_{\varnothing,k} \in I_\varnothing[k] $.
%A completely similar argument shows that the statement holds true if $ e = \{1\} $.
%In the remaining cases, note that $ l_1 = 1 $ if $ 1 \notin e $ and $ r $ is even, $ l_1 = 2 $ if $ 1 \in e $ and $ r $ is even, $ l_1 = 0 $ if $ 1 \notin e $ and $ r $ is odd, and $ l_1 = 1 $ if $ 1 \in e $ and $ r $ is odd.
%Our conditions on the elements of $ I'_e[k] $ imply that $ i_1 \equiv l_1 $ mod $ 2 $, thus $ i_1 \geq l_1 $ and $ i_1 - l_1 $ is divisible by $ 2 $. Now our induction hypothesis implies that $ I - L'_{e,k}  \in I_\varnothing[k] $.
%The inverse inclusion $ L'_{e,k} + I_\varnothing[k] \subseteq I'_e[k] $ is obvious.
%\end{itemize}
%\end{proof}

We can now compute the algebra dual to $ (\mathcal{R}^t,\varphi) $.
\begin{lemma}\label{lem:dual KADL}
	Assume that $ p $ is odd. We consider the following dual classes in $ (\mathcal{R}^t)^\vee $ with respect to the admissible sequences basis:
	\begin{itemize}
		\item $ \xi_{j,k} = q_{I_{j,k}}^\vee $ for $ 1 \leq j \leq k $
		\item $ \xi'_k = q_{L'_{\phi,k}}^\vee $ for $ k \geq 1 $
		\item $ \tau'_{j,k} = q_{L'_{\{j\},k}}^\vee $ for $ 1 \leq j \leq k $
	\end{itemize}
As a $ \mathbb{F}_p[\xi_{1,k},\dots,\xi_{k,k}] $-module, $ (\mathcal{R}^t_k)^\vee $ is isomorphic to $ \mathbb{F}_p[\xi_{1,k},\dots, \xi_{k,k}] \otimes (M_k \bigoplus M'_k) $, where
\begin{itemize}
	\item $ M_k $ is the $ \mathbb{F}_p $-vector space with basis $ \{ q_{L_{S,k}}^\vee\}_{S \subseteq \{1,\dots,k\}} $, as in Cohen--Lada--May, and
	\item $ M'_k $ is the $ \mathbb{F}_p $-vector space with basis $ \{ q_{L'_{S,k}}^\vee\}_{S \subseteq \{1,\dots,k\}} $.
\end{itemize}

As a graded commutative algebra, $ (\mathcal{R}^t)^\vee $ is generated by $ \xi_{j,k} $ with $ 1 \leq j \leq k $, $ \xi'_k $ and $ \tau'_{j,k} $ with $ 1 \leq j \leq k $, under the relations
$$ {\xi'_k}^2 = \xi_{k,k}. $$
\end{lemma}
\begin{proof}
Theorem I.3.7 in \cite{May-Cohen} states that $ \mathcal{R}_k^\vee $ is additively isomorphic to $ \mathbb{F}_p[\xi_{1,k},\dots, \xi_{k,k}] \otimes M_k $. The exact argument used to prove that result, with Lemma \ref{lem:cup coproduct Dyer-Lashof} in place of Lemma I.3.4 of \cite{May-Cohen}, yields an additive isomorphism $ {\mathcal{R}'_k}^\vee \cong \mathbb{F}_p[q_{I_{j,k}}^\vee]_{1 \leq j \leq k} \otimes M'_k $. In particular, $ \mathbb{F}_p[\xi_{1,k},\dots,\xi_{k,k}] $ is the dual of the quotient coalgebra $ \mathcal{R}_{\varnothing}[k] $.
%To prove our product relations, we first observe that the general formulas for $ \sigma_S \cdot \sigma_T $, $ \sigma_S \cdot \sigma'_T $, and $ \sigma'_S \cdot \sigma'_T $ are a consequence of the minimal relations.
%To check that $ \gamma_{\{j\},k} \gamma'_{\phi,k} = \pm \gamma_{\phi,k} \gamma'_{\{j\},k} $ and $ \gamma_{\{i,j\},k} \gamma'_{\phi,k} = \pm \gamma_{\{i\},k} \gamma_{\{j\},k} $, we observe that $ L_{\{j\},k} + L'_{\phi,k} = I_{k,k} + L'_{\{j\},k} $ and that $ L_{\{i,j\},k} + L_{\phi,k} = L_{\{i\},k} + L'_{\{j\},k} $. By dimensional reasons, the only possible summand of $ \gamma_{\{j\},k} \gamma'_{\phi,k} $ (respectively $ \gamma_{\{i,j\},k} \gamma'_{\phi,k} $) is $ \gamma_{\phi,k} \gamma'_{\{j\},k} $ (respectively $ \gamma_{\{j\},k} \gamma'_{\{j\},k} $). Our formulas follow immediately. Checking that the signs involved there are given by the desired expressions is straightforward.
Let $ + $ denote component-wise addition of $ 2k $-tuples. One can check that
\begin{align*}
	L_{S,k} &= \left\{ \begin{array}{ll}
		\sum_{s \in S} L'_{\{s\},k} & \mbox{if } |S| \mbox{ is even}\\
		\sum_{s \in S} L'_{\{s\},k} + L'_{\phi,k} & \mbox{if } |S| \mbox{ is odd}
\end{array} \right. \\
L'_{S,k} &= \left\{ \begin{array}{ll}
	\sum_{s \in S} L'_{\{s\},k} + L'_{\phi,k} & \mbox{if } |S| \mbox{ is even} \\
	\sum_{s \in S} L'_{\{s\},k} & \mbox{if } |S| \mbox{ is odd}
	\end{array} \right. .
\end{align*}
This is true because the function $ \_ + L'_{\{\ell\},k} \colon I_{S}[k] \to I'_{S \cup \{\ell\}}[k] $ 
(respectively $ \_ + L'_{\{\ell\},k} \colon I'_{S}[k] \to I_{S \cup \{\ell\}}[k] $) is an order-preserving bijection for all 
$ S \subseteq \{1,\dots,\ell-1\} $ of even (respectively odd) cardinality, thus it must preserve minimal elements. 
Since application of Adem relations to a sequence $ q_I $ produce elements $ q_J $ with $ J > I $ and $ L_{S,k}, L'_{S,k} $ are minimal, 
the argument in the proof of \cite[I.3.7]{May-Cohen} can be used to prove that the pairing between products of the form 
$ \prod_{s \in S} \tau'_{s,k} $ (if $ |S| $ is even) or $ \xi'_k \prod_{s \in S} \tau'_{s,k} $ (if $ |S| $ is odd) and the $ q_{L'_{S,k}} $ is perfect.
This implies by Lemma \ref{lem:cup coproduct Dyer-Lashof} that $ \xi'_k $ and $ \tau'_{j,k} $ with 
$ 1 \leq j \leq k $ generate $ (\mathcal{R}^t_k)^\vee $ as an $ \mathbb{F}_p[\xi_{1,k},\dots,\xi_{k,k}] $-algebra.
We can check the relations $ {\xi'_k}^2 = \xi_{k,k} $ between the generators also following  \cite{May-Cohen}. 
Explicitly, since $ I_{k,k} = L'_{\phi,k} + L'_{\phi,k} $,  $ \langle {\xi_k'}^2, q_{I_{k,k}} \rangle = 1 $. Tracking degrees, 
the only possible summand of $ {\xi'_{k}}^2 $ is $ q_{I_{k,k}}^\vee $, hence the relation holds. %Checking that the sign involved is given by the desired expression is straightforward by pairing both these classes with $ I_{k,k} $.
%The relations regarding the product $ {\mathcal{R}'_k}^\vee \otimes {\mathcal{R}'}_k^\vee \to \mathcal{R}_k^\vee $ are proved in the same way.
By comparing dimensions, we see that these provide a presentation of $ (\mathcal{R}^t_k)^\vee $ as a graded commutative algebra.
\end{proof}

\begin{remark}
	We recover the elements $ \tau_{j,k} $ and $ \sigma_{i,j,k} $ defined at page 28 in Cohen, Lada and May's book by the identities $ \tau_{j,k} = \pm \xi'_k \tau'_{j,k} $ and $ \sigma_{i,j,k} = \pm \tau'_{i,k} \tau'_{j,k} $. These are indecomposables in $ \mathcal{R}_k^\vee $, but not in $ (\mathcal{R}_k^t)^\vee $.
\end{remark}

%Letting $\vee$ denote linear dual with respect to the Nakaoka bases,

We now provide an additive basis for the twisted cohomology of the symmetric groups. 
In the following definition, gathered blocks are elements of the universal Hopf ring of the statement of Theorem \ref{m2}. 
As mentioned above,  the product of a collection of $\gamma_{S_i}\dip{p^{n_i}} $ only depends on the union of the $S_i$ as well as the list, 
with multiplicity, of the $n_i$.

\begin{definition} \label{def:gathered blocks signum}
%Fix some $k$.  
Let $S = \{ s_1 < s_2 \dots < s_r \} \subset \{1, \ldots, k\}$ and let $D = \{n_i\}_{i=1}^k$ be a finite sequence of non-negative integers of length $ k $.
If $ |S| $ is odd let $ m = 1 $, but if it is even let $ m \geq 1 $. 
%For definiteness of sign, if $ S = \{ s_1 < s_2 \dots < s_r \} $ we set $\Gamma_{S,D} = \prod_{i=1}^k \left( \gamma_{\phi,i}^{[p^{k-i}]} \right)^{n_i} \prod_{j=1}^{r/2} \gamma_{\{s_{2j-1},s_{2j}\},k} $  if $ r $ is even, and $ \Gamma_{S,D} = \prod_{i=1}^k \left( \gamma_{\phi,i}^{[p^{k-i}]} \right)^{n_i} \prod_{j=1}^{(r-1)/2} \gamma_{\{s_{2j-1},s_{2j}\},k} \gamma_{\{s_r\},k} $.

Set $$\Gamma_{S,D,m} = \prod_{i=1}^k \left( \gamma_{\varnothing,i}^{[p^{k-i}m]} \right)^{n_i} \prod_{j=1}^{r} (\gamma'_{s_{j}})\dip{m p^{k-s_j}} {\lambda_k}^{\varepsilon(r)},$$  where $ \varepsilon(r) $ is $ 0 $ if $ r $ is even and $ 1 $ if $ r $ is odd

Similarly, we  define 
$$ \Gamma'_{S,D,m} = \prod_{i=1}^k \gamma_{\varnothing,i}^{[p^{k-i}m]} \prod_{j=1}^r (\gamma'_{s_j})\dip{m p^{j-s_j}} {\lambda_k}^{1-\varepsilon(r)},$$ for $ S \subseteq \{1,...,k\} $, with $S$ and $n_i$ as above but now  $ |S| = 1 $ when $ m = 1 $.

We call $\Gamma_{S,D,m}$ and $\Gamma'_{S,D,m}$  {\bf gathered blocks} with {\bf profile} $(S,D)$ in $ \bigoplus_n H^*(\si_n; \mathbb{F}_p) $
 or $ \bigoplus_n H^*(\si_n; \sgn) $ respectively.

We define our preferred 
{\bf Hopf monomials} in $ \hdp(S^0) $ as a $ \odot $ product of divided powers of primitive gathered blocks with pairwise different profiles, 
all belonging to $ \bigoplus_n H^*(\si_n; \mathbb{F}_p) $ or all belonging to $ \bigoplus_n H^*(\si_n; \sgn) $
\end{definition}

We observe that $ \Gamma_{S,D} $ is primitive if and only if $ k \in S $, $ n_k > 0 $ or $ r $ is odd, while $ \Gamma'_{S,D} $ is primitive if and only if $ k \in S $, $ n_k > 0 $ or $ r $ is even.

%\begin{itemize}
%\item $ \gamma_{k_1,n_1} \dots \gamma_{k_r, n_r} \beta_{i_1,i_2,m_1} \dots \beta_{i_{2s-1},i_{2s},m_s} \alpha'_{j,l} $, where $ r \geq 0 $, $ s \geq 0 $, $ 1 \leq k_1 \leq \dots \leq k_r \leq n $, $ 1 \leq i_1 < \dots < i_{2s} < j $, and $ n_1p^{k_1} = \dots n_rp^{k_r} = m_1p^{i_2} = \dots = m_s p^{i_{2s}} = l p^j $. In this case, we let $ \underline{k} = (k_1, \dots, k_r) $ and $ \underline{i} = (i_1, \dots, i_{2s}, j) $.
%\item $ \gamma_{k_1,p^{n-k_1}} \dots \gamma_{k_r p^{n-k_r}} \beta_{i_1,i_2,p^{n-i_2}} \dots \beta_{i_{2s-1},i_{2s},p^{n-i_{2s}}} \alpha_{i_{2s+1},n} \alpha'_{j,p^n} $, where $ r \geq 0 $, $ s \geq 0 $, $ 1 \leq k_1 \leq \dots \leq k_r \leq n$, $ 1 \leq i_1 < \dots < i_{2s} < i_{2s+1} < j \leq n $. In this case, we let $ \underline{k} = (k_1, \dots, k_r) $ and $ \underline{i} = (i_1, \dots, i_{2s+1}, j) $.
%\item $ \gamma_{k_1,p^{n-k_1}} \dots \gamma_{k_r,p^{n-k_r}} \gamma'_k $, where $ r \geq 0 $, $ 1 \leq k_1 \leq \dots \leq k_r \leq n $. In this case, we let $ \underline{k} = (k_1, \dots, k_r) $ and $ \underline{i} = \varnothing $ be the empty tuple.
%\end{itemize}

These gathered blocks will be the building blocks of  Hopf monomial bases for the cohomology we compute, and can be assembled into ``skyline diagrams.''
The reader is encouraged to compare this with the notion of gathered block in the classical mod $ p $ cohomology of the symmetric groups, 
as defined in Guerra's paper \cite{Guerra:17} at page 964.

\begin{lemma} \label{sign-Hopf-monomials-generators}
	The universal Hopf ring described by the statement of Theorem \ref{m2} is spanned as an $ \mathbb{F}_p $-vector space by Hopf monomials.
\end{lemma}
\begin{proof}
	By unraveling the definition of the universal object, we directly see that a complete set of relations for the considered component super-Hopf rig between the generators $ \gamma_k\dip{n} $, $ {\gamma'}_{i}\dip{n} $, $ \lambda_k $ is the following:
	\begin{itemize}
		\item $ (\lambda_k)^2 = \gamma_{k} $
		\item the $ \odot $ product of a class in $ H^*(\si_n;\mathbb{F}_p) $ and a class in $ H^*(\si_n; \sgn) $ is $ 0 $.
		\item $ \Delta( \lambda_k) = 1 \otimes \lambda_k + \lambda_k \otimes 1 $
		\item $ \Delta (\gamma_k\dip{n}) = \sum_i \gamma_k\dip{i} \otimes \gamma_k\dip{n-i} $
		\item $ \Delta({\gamma'}_k\dip{n}) = \sum_i {\gamma'}_k\dip{i} \otimes {\gamma'}_k\dip{n-i} $
		\item $ \gamma_{k}\dip{n} \odot \gamma_{k}\dip{m} = \left( \begin{array}{c}
			n+m \\
			n
		\end{array}\right) \gamma_{k}\dip{m+n} \\ $
		\item $ {\gamma'}_k\dip{n} \odot {\gamma'}_k\dip{m} = \left( \begin{array}{c}
			n+m \\
			n
		\end{array}\right) {\gamma'}_{k}\dip{m+n} \\ $
		\item the $ \odot $ product of elements with different sign degree is $ 0 $
	\end{itemize}
	Since our Hopf rig is generated by the classes above, monomials in both products ($ \cdot $ and $ \odot $) generate it as a vector space. By Hopf ring distributivity we can reduce to $ \odot $-products of $ \cdot $-monomials of generators. Moreover, we can further restrict to considering $ \cdot $-monomials of generators all belonging to the same component and we can discard monomial in which $ \lambda_k $ appears at least twice thanks to the relation $ \lambda_k^2 = \gamma_k $.
	All the relations above preserve profile of such $ \cdot $-monomials and can be used to equate any product which has the same profile as a gathered block. Hence, Hopf monomials generate our Hopf ring as a vector space.
\end{proof}

We can now prove the main result of this section.

\begin{proof}[Proof of Theorem \ref{m2}]

%Given $ k \in \mathbb{N} $ and $ S \subseteq \{ 1, \dots, k \} $ not empty, define $ \alpha(S,k) = \frac{(p-1)(\max(S)}{2} $ if $ |S| = 2 $, $ \alpha(S) = 0 $ otherwise. Define $ \alpha'(S,k) = \frac{(p-1)\max(S)}{2} $ if $ |S| = 1 $, $ \alpha'(S) = 0 $ otherwise. Let $ \gamma_{S,k} = (-1)^{\alpha(S,k)} q_{L_{S,k}}^\vee \in H^*(\si_n; \mathbb{F}_p) $ if $ |S| $ is equal to $ 1 $ or $ 2 $, $ \gamma_{\phi,k} = q_{I_{k,k}}^\vee $, and $ \gamma'_{S,k} = (-1)^{\alpha'(S,k)} q_{L'_{S,k}}^\vee \in H^*(\si; \sgn) $ if $ |S| $ is equal to $ 0 $ or $ 1 $, where linear duals are taken with respect to the Nakaoka monomial basis.
Given a subset $ S \subset \{1,\dots,k\} $, define $ \gamma'_{S,k} = q_{L'_{S,k}}^\vee $, where linear duals are taken with respect to the Nakaoka monomial basis. Similarly, define $ \gamma_{S,k} = q_{L_{S,k}}^\vee $. % if $ |S| \not= 2 $. To obtain the relations of Theorem \ref{m2}, we need to slightly modify this definition if $ |S| = 2 $ and let $ \gamma_{S,k} = - (-1)^{\frac{p-1}{2}(k- \max(S))} L_{S,k}^\vee $.
\begin{comment}
\begin{center}
\begin{tabular}{c c c}
$ {}\gamma_{\varnothing,j}}\dip{p^{k-j}}  = {q_{I_{j,k}}}^{\vee} $ & $ \gamma_{\{i,j\},k\}}  = {q_{K_{i,j,k}}}^{\vee} $ &  $ \gamma_{\{j\},k\}}  = {q_{J_{j,k}}}^{\vee} $  \\
  
$\gamma_{\varnothing,k} = {q_{I'_k}}^{\vee} $  &  &  $ \gamma'_{\{j\},k\}}  = {q_{J'_{j,k}}}^{\vee} $.
\end{tabular}
\end{center}
\end{comment}
%In particular, $ \gamma'_{\varnothing,k} $ is defined.
We also define:
\begin{itemize}
\item $ \lambda_k = \gamma'_{\varnothing,k} $
\item $ \gamma_{j,m} = (q_{I_{j,j}}^{*m})^\vee $
\item $ \gamma_k = \gamma_{k,1} $
\item $ \gamma'_{i,m} = (q_{L_{\{i\},i}}^{*m})^\vee $
\item $ \gamma'_k = \gamma'_{k,1} $
\item $ \gamma_k = \gamma_{\varnothing,k} $
%\item $ \gamma'_{\varnothing,k} = \gamma'_{\varnothing,k} $
\end{itemize}
Since $ q_0 $ is the $ p^{th} $ power in the mod $ p $ homology of $ D(S^0) $ with coefficients in $ \rho $, then $ \gamma_{j,p^{k-j}} = (q_{I_{j,k}})^{\vee} $ and $ \gamma'_{i,p^{k-i}} = \gamma'_{\{i\},k} $.

We use Proposition \ref{prop:cup product sign} and Lemma \ref{lem:dual KADL} to deduce our cup product relations.
The structure of $ \mathcal{R}_k^\vee $ is described by Theorem I.3.7 in \cite{May-Cohen}.  In this reference, 
the product relations involve dividing by $ \gamma_{\varnothing,k} $, but they are equivalent to the identities stated in \cite{Guerra:17}, of 
whose our relations (as far as $ \mathcal{R}_k^\vee $ is concerned) are straightforward reformulations that use lower indexes. 
For the cohomology with coefficients in the sign representation, a similar argument using Lemma \ref{lem:dual KADL} yields our description.

%As previously mentioned, Theorem \ref{thm:cup product R} give some cup product relations between those five sets of classes that are true after they are mapped in $ \mathcal{R}_k^\vee $ or $ {\mathcal{R}'}_k^\vee $ via the obvious restriction map.
We first observe that $ \gamma_{j,k} $, $ \gamma'_{i,k} $ and $ \lambda_k $ are liftings of the generators $ \xi_{j,k} $, $ \tau'_{i,k} $ and $ \xi'_k $ of $ (\mathcal{R}^t_k)^\vee $.% in such a way that the product relations given by Lemma \ref{lem:dual KADL} hold in $ \hdp(S^0) $, without restricting to indecomposables, by the same argument used in \cite[Lemma 2.5]{Guerra:17} with minor modifications.
%Those identities also hold in $ A $ and $ A_{\sgn} $, without applying restriction maps.
%For classes in $ \bigoplus_n H^*(\si_n; \mathbb{F}_p) $, this statement is the content of Lemma 2.5 in \cite[Lemma 2.5]{Guerra:17}. For classes in $ \bigoplus_n H^*(\si_n; \mathbb{F}_p)$, the same proof works with obvious minor modifications.

We first prove by induction on $ m \geq 1 $ that $ {\gamma'}_{k}\dip{m} = \gamma'_{k,m} $. The base of the induction ($ m = 1 $) is obvious, so we assume $ m > 1 $.
Since the product $ * $ in homology is linear dual to the coproduct $ \Delta $ in $ \hdp(S^0) $, a straightforward calculation by induction on $ i $ yields the coproduct formula:
\[
%\begin{equation*}
%\begin{split}
%\Delta ( \gamma_{S,k}^{[p^i]} ) &= \sum_{j=0}^{p^i} \gamma_{S,k}^{[j]} \otimes \gamma_{S,k}^{[p^i-k]} \\
\Delta ( \gamma'_{k,m} ) = \sum_{i+j=m} \gamma'_{k,i} \otimes \gamma'_{k,j}
%\end{split}
%\end{equation*}
\]
Similarly, using the compatibility between coproduct and divided powers, we deduce by induction that
\[
\Delta ( {\gamma'}_{k}^{[m]} ) = \sum_{j=0}^{m} {\gamma'}_{k}^{[j]} \otimes {\gamma'}_{k}^{[m-j]}.
\]
By combining these two coproduct identities with the induction hypothesis we deduce that the difference between $ {\gamma'_k}\dip{m} $ and $ \gamma'_{k,m} $ must be primitive. If $ m $ is not a power of $ p $, then there are no $ \odot $-indecomposables in the right component. If $ m = p^l $, the primitives are determined by Lemma \ref{lem:dual KADL} and there is none in the correct dimension. In all cases, this difference must be zero.
This shows that $ \gamma'_k $ is standard non-nilpotent. 
A similar argument shows that $ \gamma_k\dip{m} = \gamma_{k,m} $ and $ \gamma_k $ is standard non-nilpotent.

It is obvious from its definition that $ \lambda_k $ is primitive. The cup-product relation $ (\lambda_k)^2 = \gamma_{k} $ holds because both sides are primitive, hence $ \odot $-indecomposables, and because this identity holds in $ (\mathcal{R}^t)^\vee $. Then our cup-product relations are a reformulation of those in $ (\mathcal{R}^t_k)^\vee $.

Let $ H $ be the universal Hopf rig described in the statement of the Theorem  \ref{m2}. Since we have checked that all the required 
relations hold in $ \hdp(S^0) $, there must be a unique Hopf rig map $ \pi \colon H \to \hdp(S^0) $ compatible with our choice of generators. 
As the ring of $ \odot $-indecomposables in $ \bigoplus_{n \geq 0} H^*(\si_n; \rho) $ is the linear dual of $ \mathcal{R}^t $, 
which is generated by the images of $ \gamma_i\dip{p^k} $, $ {\gamma'}_i\dip{p^k} $ and $ \lambda_k $ by Lemma \ref{lem:dual KADL},
$ \hdp(S^0) $ is generated by the classes defined at the beginning of this proof and, consequently, $ \pi $ is surjective. 
Hence, by Lemma \ref{sign-Hopf-monomials-generators}, the images of Hopf monomials generate $ \hdp(S^0) $ as a $ \mathbb{F}_p $-vector space.
To prove that $ \pi $ is an isomorphism, it is enough to check that they form a basis, which can be done by comparing dimension degree-wise with the
 Nakaoka basis in homology, as done in \cite[Theorem 2.7]{Guerra:17}.
\end{proof}

As a byproduct of this proof we directly obtain the following.

\begin{corollary} \label{sign-gathered-block-basis}
The set of Hopf monomials in $ \hdp(S^0) $, defined in Definition~\ref{def:gathered blocks signum}, is a tri-graded additive basis for $ \hdp(S^0) $ as a $ \mathbb{F}_p $-vector space.
\end{corollary}

%%% Local Variables:
%%% TeX-master: "DX-QX.tex"
%%% End:

%% file: DX-QXSection5.tex
% !TEX root = DX-QX.tex

\section{Cohomology rings of extended powers with basepoints} \label{sec:DXdphr} %modified

We are now ready to establish our main calculations.  In Section~\ref{sec:defs} we developed the algebraic framework which governs
the cohomology rings of extended powers, when considered together.  Such a framework demands twisted coefficients, so
in Section~\ref{sec:symmetric group} we extended well-known results about homology to that setting.  We now show that the cohomology
generated within the framework of component Hopf rings with divided powers pairs perfectly with homology.  We  develop
an additive basis for the former before proving the pairing result.

\subsection{Additive basis for the cohomology of $ DX_+ $} \label{sec:algebraic basis}
%modified

We describe  explicit additive bases of the super-Hopf ring with additive divided powers
given  Theorems \ref{m3} and \ref{m4}. We use these to pair with homology.  
Our basis is useful to perform concrete calculations, because we describe an algorithm to compute the two products and the coproduct in terms of it. 

We primarily combine results from Section~\ref{explicit} with Lemma~\ref{sign-Hopf-monomials-generators}.
We apply Lemma \ref{lem:algebraic basis} when $ A = \hdp(S^0) $ and $ V = H^*(X; \mathbb{F}_p) $, so that 
a basis for $ Q(A) $ is given by gathered blocks of component equal to a power of $ p $ by Corollary~\ref{sign-gathered-block-basis}. 
The basis provided by Lemma \ref{lem:algebraic basis} consists of such gathered blocks, with the extra datum of a class in the cohomology of $ X $.
Informally, by the Hopf ring relations in $ A $, the transfer product of two such gathered blocks with the same profile is again another gathered block. 
Therefore, we can merge $ \odot $-factors with the same profile and the same extra cohomology class of $ X $ and remove the constraint on the component.
We formalize this as follows.

\begin{definition} \label{decorated-hopf-monomials}
	Let $ X $ be a topological space and let $ \mathcal{B}$ be a graded additive basis for $ H^*(X; \mathbb{F}_p) $. A  {\bf decorated gathered block} 
on $ \mathcal{B} $ is a pair $ (x,b) $, where $ b $ is a gathered block in $ \bigoplus_{n \geq 0} H^*(\si_n; A_n) $ and $ x \in \mathcal{B} $. 
The {\bf profile} of the decorated gathered block $ (x,b) $ is the profile of $ b $, and its {\bf decoration} is $ x $. 
Its \emph{sign degree} is $ 0 $ if the degree of $ x $ and the sign degree of $ b $ have the same parity, is $ 1 $ if they have different parity.
	A {\bf decorated Hopf monomial} on $ \mathcal{B}$ is a formal expression of the form $ b_1 \odot \dots \odot b_r $, where $ b_1,\dots, b_r $ are decorated gathered blocks on $ \mathcal{B}$, up to permutation of the $ b_i $, where no two $ b_i $s have the same profile and decoration,
 and where all the $ b_i $s have the same sign degree.
\end{definition}

Our previous results assemble as follows.

\begin{proposition} \label{prop:Hopf monomial basis}
	Let $ X $ be a topological space and let $ \mathcal{B}$ be a graded additive basis for $ H^*(X; \mathbb{F}_p) $. Let $ A = \bigoplus_{n \geq 0} H^*(\si_n; A_n) $ and let $ V = H^*(X; \mathbb{F}_p) $. Let $ H_{alg} $ be the universal object among Hopf ring with additive divided powers satisfying the conditions of the statement of Theorem \ref{m4}.
	Let $ \mathcal{M}$ be the set of decorated Hopf monomials on $ \mathcal{B} $.
	We map $ \mathcal{M} $ to $ H_{alg} $ by realizing a decorated gathered block $ (x,b) $ as $ \iota(x)\dip{n(b)} \cdot \pi(b) $, and a decorated Hopf monomial as the $ \odot $-product in $ H_{alg} $ of the constituent gathered blocks.
	Then $ \mathcal{M} $ is an additive basis for $ H_{alg} $ as a $ \mathbb{F}_p $-vector space.
\end{proposition}

We can give a graphical description of our additive basis using a mild generalization of Giusti, Salvatore, and Sinha's skyline diagrams. Moreover, all our structural morphisms can be understood graphically.
We first consider undecorated skyline diagrams, which correspond to classes in $ \hdp(S^0) $.

First, we recall some definitions from \cite{Sinha:12} and \cite{Guerra:17}.

\begin{definition} \label{skyline}
A {\bf skyline diagram} is a display of 
columns placed one next to the other horizontally. Each column is comprised of  rectangular {\bf boxes} with the same width stacked one on top of the other. 
The possible dimensions of these boxes depend on a prime $ p $:
\begin{itemize}
	\item If $ p = 2 $, a column can be made of any number of boxes of width $ n2^k $ and height $ 1-2^{-k} $, for some $ k,n \in \mathbb{N} $, $ n \geq 1 $.
	\item If $ p > 2 $, a column can be made of any number of hollow rectangles of width $ np^k $ and height $ 2(1-p^{-k}) $ and at most one solid box of width $ np^k $ and height dependent on a subset $ S \subseteq \{1, \dots, k\} $. To simplify notation, we let $ |S| = 2a+\varepsilon $, with $ a \in \mathbb{N}$ and $ \varepsilon \in \{0,1\} $. For any choice of $ k $ and $ S $, we have two possible heights: $ n[2( a + \varepsilon- \sum_{s \in S} p^{-s} ) - \varepsilon p^{-k}] $ and $ n[2( a - \sum_{s \in S} p^{-s} ) - (1-\varepsilon)p^{-k} + 1] $.
	
	 If a box has height given by the first formula, we say that it is of {\bf even type}, while we say that it is of {\bf odd type} otherwise. 
Boxes of even type are always defined if the subset has even cardinality $ |S| = 2a $, while they are allowed for $ |S| =2a+1 $ only if $ n = 1 $. Boxes of odd type are always defined if the subset has odd cardinality $ |S| = 2a+1 $, while they are allowed for $ |S| =2a $ only if $ n = 1 $.
	
	We say that a column is of even type if its solid box is of even type or is absent, and we say that it is of odd type if its solid box is of odd type.
\end{itemize}
\end{definition}

We require that a skyline diagram does not contain two columns made of fundamental rectangles with matching heights, 
and that its constituent columns are all of the even type or all of odd type.
We do not keep track of the order in which columns are placed or in which boxes are stacked inside each column, so
 two skyline diagrams that differ by a permutation of their columns are considered the same.

Each fundamental box should be interpreted as a class in $ \bigoplus_{n \geq 0} H^*(\si_n; \mathbb{F}_p) $ or $ \bigoplus_{n \geq 0} H^*(\si_n; \sgn) $, in the component corresponding to its width. Precisely:
\begin{itemize}
	\item if $ p = 2 $, the rectangle of dimensions $ n2^k $ and $ 1-2^{-k} $ corresponds to $ {\gamma_{k}}\dip{n} $
	\item if $ p > 2 $, we regard the hollow rectangle whose width is $ np^k $ and whose height is $ 2(1-p^{-k}) $ as $ {\gamma_k}\dip{n} $
	\item if $ p > 2 $, the solid box of even (respectively odd) type with width $ np^k $ associated to the subset $ S $ corresponds to the gathered block $ \Gamma_{S,\underline{0},n} $ (respectively $ \Gamma'_{S,\underline{0},n} $), as defined in Definition \ref{def:gathered blocks signum}.
\end{itemize}

A column is understood as the cup product of the constituent fundamental boxes, and a general skyline diagram is interpreted as the transfer product of its columns.
With these correspondences, the elements of the additive bases for $ \bigoplus_{n \geq 0} H^*(\mathcal{S}_n; \mathbb{F}_2) $ and 
$ \bigoplus_{n \geq 0} H^*(\mathcal{S}_n; \mathbb{F}_p) $ ($ p > 2 $) obtained in the papers mentioned above are described as those skyline diagrams 
(with columns of even type if $ p>2 $).
For example, in Figure \ref{fig:example skyline} we depict the skyline diagrams corresponding to $ \gamma_{1}^3 \odot \gamma_{2} \in H^*(\si_6; \mathbb{F}_2) $ and $ {\gamma_{1}}\dip{3} \gamma_{\{1\},2} \odot \gamma_{1}^2 \in H^*( \si_{12}; \mathbb{F}_3) $.

It is sometimes  useful  to consider vertical dashed lines dividing some fundamental rectangles into equal parts. 
Explicitly, boxes corresponding to $ {\gamma_{k}}\dip{n} $ are divided into $ n $ parts of width $ p^k $ and, 
if $ p > 2 $, even-dimensional solid boxes of even type and odd-dimensional solid boxes of odd type corresponding to a subset 
$ S $ are divided into parts of width $ p^{\max(S)} $.

\begin{figure}[htbp]
	\begin{center}
		\begin{tikzpicture}[line cap=round,line join=round,>=triangle 45,x=1cm,y=1cm]
			\clip(-3,-4) rectangle (7,2);
			\fill[line width=2pt,fill=black,fill opacity=0.5] (0,-2) -- (0,-0.7777777777777778) -- (4.5,-0.7777777777777778) -- (4.5,-2) -- cycle;
			\draw [line width=1pt] (0,0)-- (1,0);
			\draw [line width=1pt] (0,0.5)-- (0,0);
			\draw [line width=1pt] (0,0)-- (1,0);
			\draw [line width=1pt] (1,0)-- (1,0.5);
			\draw [line width=1pt] (1,0.5)-- (0,0.5);
			\draw [line width=1pt] (0,1)-- (1,1);
			\draw [line width=1pt] (1,1)-- (1,0.5);
			\draw [line width=1pt] (1,0.5)-- (0,0.5);
			\draw [line width=1pt] (0,1.5)-- (0,1);
			\draw [line width=1pt] (0,1)-- (1,1);
			\draw [line width=1pt] (1,1)-- (1,1.5);
			\draw [line width=1pt] (1,1.5)-- (0,1.5);
			\draw [line width=1pt] (1,0.75)-- (1,0);
			\draw [line width=1pt] (1,0)-- (3,0);
			\draw [line width=1pt] (3,0)-- (3,0.75);
			\draw [line width=1pt] (3,0.75)-- (1,0.75);
			\draw [line width=1pt] (0,-2)-- (4.5,-2);
			\draw [line width=1pt] (4.5,-2)-- (4.48,-3.3533333333333335);
			\draw [line width=1pt] (4.48,-3.3533333333333335)-- (-0.02,-3.3533333333333335);
			\draw [line width=1pt] (-0.02,-3.3533333333333335)-- (0,-2);
			\draw [line width=1pt] (0,1)-- (0,0.5);
			\draw [line width=1pt] (0,-2)-- (0,-0.7777777777777778);
			\draw [line width=1pt] (0,-0.7777777777777778)-- (4.5,-0.7777777777777778);
			\draw [line width=1pt] (4.5,-0.7777777777777778)-- (4.5,-2);
			\draw [line width=1pt] (4.5,-2)-- (0,-2);
			\draw [line width=1pt] (4.5,-3.3333333333333335)-- (6,-3.3333333333333335);
			\draw [line width=1pt] (6,-3.3333333333333335)-- (6,-2);
			\draw [line width=1pt] (6,-2)-- (4.5,-2);
			\draw [line width=1pt] (4.5,-2)-- (4.5,-3.3333333333333335);
			\draw [line width=1pt] (4.5,-2)-- (6,-2);
			\draw [line width=1pt] (6,-2)-- (6,-0.6666666666666666);
			\draw [line width=1pt] (6,-0.6666666666666666)-- (4.5,-0.6666666666666666);
			\draw [line width=1pt] (4.5,-0.6666666666666666)-- (4.5,-2);
			\draw (-2.26,1) node[anchor=north west] {$ \gamma_{1}^3 \odot \gamma_{2} $};
			\draw (-2.76,-2) node[anchor=north west] {$ {\gamma_{1}}\dip{3} {\gamma'_{1}}\dip{3}\lambda_2 \odot \gamma_{1}^2 $};
		\end{tikzpicture}
		\caption{Examples of skyline diagrams in $ H^*( \si_6; \mathbb{F}_2) $ (top) and $ H^*( \si_{12}; \mathbb{F}_3) $ (bottom)}
		\label{fig:example skyline}
	\end{center}
\end{figure}
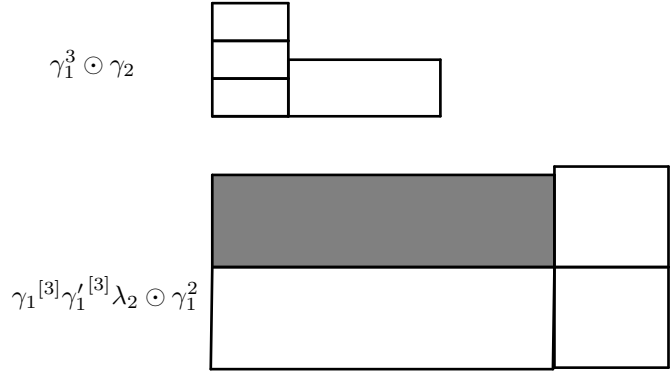

\begin{proposition}\label{prop:skyline-undecorated}
Hopf monomials in $ \hdp(S^0) $ are in bijective correspondence with skyline diagrams as defined above
\end{proposition}
\begin{proof}
We already noted that solid blocks correspond to gathered blocks $ \Gamma_{S,D,n} $ or $ \Gamma'_{S,D,n} $ and hollow 
rectangles correspond to $ \gamma_{k}\dip{n} $. Hence, the set of all possible columns is in bijective correspondence with the set of gathered blocks, 
with taking the cup product of the classes corresponding to the stacking of constituent boxes. 
Similarly, a skyline diagram give rise to a Hopf monomial in $ \hdp(S^0) $ by taking the transfer product of the gathered blocks 
corresponding to its constituent columns.

To define this  precisely, one needs to fix an ordering of the boxes and the columns for each skyline diagram, because cup product and transfer product are only commutative up to sign. With that subtlety, it is straightforward to check directly from the definitions that this is a bijection.
\end{proof}

We next describe our generalization of skyline diagrams for $ \hdp(X_+) $. 
%The difference between $ \bigoplus_{n \geq 0} H^*(\si_n; \mathbb{F}_p) $ and $ H^*(DX; \mathbb{F}_p) $ is the presence of the divided powers of 
%elements in $ H^*(X; \mathbb{F}_p) $. Therefore, we must add other fundamental boxes that reflect these classes. 

\begin{definition}
	Let $ p $ be a prime number and $ V $ be a graded algebra over $ \mathbb{F}_p $. A {\bf column decorated in $ V $} is a column diagram as 
in Definition~\ref{skyline}, together with an element $ x \in V $.
	Graphically, we depict the decoration of the column as an additional rectangle labeled $ x $ with width equal to its width and height equal to the degree of $ x $, placed beneath that column.
	%When $ p > 2 $, we require additionally that columns of even and odd type have decorations of even and odd dimension respectively.
	When $ p > 2 $, we say that a decorated column is of even type if the type of underlying un-decorated column and the dimension of the decoration are both even or both odd. We say that it is of odd type otherwise.
	
	A {\bf skyline diagram decorated in $ V $} is a diagram constituted by columns decorated in $ V $ placed one next to each other horizontally, in which there are not two columns made of blocks of corresponding heights and having the same decoration and in which the constituent columns are all of even type or all of odd type.
	
	If $ \mathcal{B} $ is an additive basis for $ V $ as a graded $ \mathbb{F}_p $-vector space, then we define a 
{\bf skyline diagram decorated in $ \mathcal{B} $} as a skyline diagram decorated in $ V $ whose decorations are elements of that basis.
\end{definition}

We now assume $ V = H^*(X; \mathbb{F}_p) $. %Via the map $ \tilde{D}_nX \to B(\mathcal{S}_n) $, we can interpret an unlabeled column diagram as a class $ c \in H^*(\tilde{D}X; \mathbb{F}_p) $. The same column decorated with an element $ x \in H^*(X; \mathbb{F}_p) $ describes the cup product $ c \cdot x^{[n]} $, where $ n $ is the width of the column. In the particular case in which the original column has height $ 0 $, the diagram reduces to a single $ n \times \deg(x) $ rectangle labeled $ x $, that represents $ x^{[n]} $.

\begin{proposition}
	Let $ p $ be a prime number and let $ X $ be a topological space and $ \mathcal{B} $ be a basis for $ H^*(X; \mathbb{F}_p) $. 
There is a bijective correspondence between the set of skyline diagrams decorated in $ \mathcal{B} $ and Hopf monomials 
with respect to the given basis in $ \hdp(X_+) $ as in Definition~\ref{decorated-hopf-monomials}. 
Stacking boxes on top of each other corresponds to taking the cup product of the associated cohomology classes (restricted
to the relevant component) and placing columns next to each other horizontally corresponds to taking their transfer product.
\end{proposition}
\begin{proof}
If $ c $ is a column with decoration $ x $, we associate with it the decorated gathered block $ \Gamma(c) \cdot x\dip{w} $, where $ w $ is the width of the column and $ \Gamma(c) $ is the gathered block in $ \hdp(S^0) $ corresponding to $ c $.
We then associate to a decorated skyline diagram the transfer product of the decorated gathered blocks corresponding to its constituent decorated columns. It is easy to deduce from Proposition \ref{prop:skyline-undecorated} that this is a bijection.
\end{proof}

By taking only decorated skyline diagrams whose columns are all of even type, we obtain
 a basis for $ \bigoplus_{n \geq 0} H^*(\tilde{D}_n(X); \mathbb{F}_p) $.

Our next aim is to describe a way to compute the structural morphisms via skyline diagrams.  The rough graphical interpretation of the cup product of
 gathered blocks is ``splitting and piling the corresponding columns one on top of the other.'' However, if $ p > 2 $, this procedure does not always 
 yield an ``allowable'' column. Hence, we must use our cup product relations in $ H^*(\mathcal{S}_n; \mathbb{F}_p) $ or $ H^*( \mathcal{S}_n; \mathbb{F}_p \otimes \sgn ) $ to write that stack of rectangles as a multiple of an allowable column as follows.

\begin{proposition} \label{rem:cup product columns}
The cup product of classes in $ \hdp(X_+) $ represented by single columns of width $ np^k, p \nmid n $,  proceeds as follows:
	\begin{itemize}
		\item We stack the hollow boxes of the two columns, and we substitute the two solid parts (if any) with a single solid rectangle whose height is the sum of the height of the two original solid boxes
		\item Assume that the two solid boxes correspond to two increasing subsets $ S = \{ s_1 < \dots < s_k \} $ and $ S' = \{ s'_1,\dots, s_l'\} $. We multiply the result by a scalar coefficient $ \lambda $. If the two sequences have a common entry, $ \lambda = 0 $. Otherwise, it is equal to $ \pm 1 $, the sign being that of the permutation that puts the sequence $ (s_1,\dots,s_k,s'_1,\dots,s'_l) $ in increasing order.\
		\item When $ n = 1 $ and the given columns are odd-dimensional and of even type or even-dimensional and of odd type, we add a hollow rectangle of height $ 2(1-p^k) $ and lower the solid part accordingly to preserve the total height of the column.
	\end{itemize}
	The cup product of columns decorated in $H^*(X)$ is obtained by cup-multiplying the underlying non-decorated columns and 
	choosing as decoration the cup product of the original decorations.
\end{proposition}

We are now ready to describe the two products $ \cdot $ and $ \odot $, the coproduct $ \Delta $, and the divided power operations $ \__{[n]} $ graphically:
\begin{itemize}
	\item We obtain the transfer product of two skyline diagrams by juxtaposing them horizontally. If there are two columns of width 
$ n $ and $ m $ having the same decoration and boxes with the same heights (i.e., gathered blocks with the same profile), we merge 
them into a single column whose width is $ n+m $, with a coefficient  $\binom{n+m}{n}$, or $ 0 $ if $ p>2 $ and the two decorated 
columns are odd-dimensional.
	\item The coproduct of a Hopf monomial is computed via the fact that the generators are standard non-nilpotent of primitive and the
 bialgebra property of $ \Delta $ and $ \odot $. Graphically, this corresponds to drawing vertical dashed lines inside the rectangles as explained above, 
 separating each column into two along vertical lines (dashed or not) of full height in all the possible ways, and arranging the new columns into two diagrams. Each new column is assigned the decoration of the old column from which it arises. 
 Equivalently, one can divide rectangles corresponding to $ x\dip{n} $, $ x \in \hdp(X_+) $ evenly in parts of width $ 1 $ using internal dashed lines.
	\item The cup product of two skyline diagram is obtained by first drawing all the dashed lines internal to the fundamental boxes. Then, divide the two diagrams into columns, either using the original boundaries or along the dashed lines that run entirely through the height of an existing column. 
Finally, if possible match the columns arising from the first diagram bijective with those coming from the second, respecting width, 
and cup-multiply the matched columns as in Proposition~\ref{rem:cup product columns}
and use them to create a new skyline diagram (with a suitable coefficient if $ p > 2 $).
	\item If $ p = 2 $, the divided power operation $ \dip{n} $ acts on single column diagrams by scaling by $ n $ the width of all its constituent rectangles, 
while retaining their heights and the decoration. If $ p > 2 $, this action is computed in the same way if the column is even-dimensional, 
while every non-trivial divided power operation acts trivially on odd-dimensional columns. The action of $ \dip{n} $ on a general skyline 
diagram with $ m $ columns is understood by applying it to each column and multiplying the resulting diagram by $ (n!)^{m-1} $.
\end{itemize}

\subsection{Analyzing the pairing}

To prove our main theorem we need a preliminary technical lemma.
\begin{lemma}\label{lemma-partial-compatibility}
Let $ \iota_X \colon H^*(X;\mathbb{F}_p) \otimes \rho \to \hdp(X_+) $ be the homomorphism that identifies the domain with the component-$ 1 $ algebra of the codomain. Let $ \pi \colon \hdp(S^0) \to \hdp(X_+) $ the morphism naturally induced by the terminal map $ X \to \{ * \} $. 
For all $ m \in \mathbb{N} $, for all $ x \in H^*(X; \mathbb{F}_p) $ the following statements are true.
\begin{enumerate}
\item For all $ \alpha_i \in H_*(X; \mathbb{F}_p) $, the Kronecker pairing between homology and cohomology satisfies
$$ \langle \iota_X(x)\dip{n}, \alpha_1 * \dots * \alpha_n \rangle = \pm \prod_{i=1}^n \langle \iota_X(x), \alpha_i \rangle,$$
while 
$$
 \langle \iota_X(x)\dip{n}, q_{I_1}(\alpha_1) * \dots * q_{I_r}(\alpha_r) \rangle = 0
 $$
if $ I_1, \dots, I_r $ are admissible sequences of KADL operations with at least one $ q_{I_i} $ different from $ q_0 \dots q_0 $.
\item For all $ x,x' \in H^*(X; \mathbb{F}_p) \otimes \rho $ (both of even total degree if $ p > 2$), $ \iota_X(x)\dip{n} \cdot \iota_X(x')\dip{n} = \iota_X(xx')\dip{n} $
\item For all $ y \in H^*(\si_n; \rho) $, $ (\iota_X(x)\dip{n} \cdot \pi(y))\dip{m} = \iota_X(x)\dip{nm} \cdot \pi(y)\dip{m} $.
\end{enumerate}
\end{lemma}

\begin{proof}
We provide the proof for $ p > 2 $. The case $ p = 2 $ can be proved in the same way, without the complication of the sign representation.
The statement for $ n = 1 $ is trivial, thus we assume that $ n \geq 2 $.

 The first statement follows from the formula for the coproduct of divided powers and the fact that in component $ n = 1 $ 
the Kronecker pairing coincides with the Kronecher pairing between $ H^*(X; \mathbb{F}_p) $ and $ H_*(X; \mathbb{F}_p) $.
For the last statement, exploiting again the coproduct formula of divided powers, we reduce to the  case  $ r = 1 $.

Thus we only need to prove that $ \langle \iota_X(x)\dip{p^k}, q_I(\alpha) \rangle = 0 $ if $ I $ is an admissible sequence of KADL operations of length $ k $ different from $ q_0 \dots q_0 $. To prove this, we pick a suitable cochain representative of $ \iota_X(x)\dip{m} $.
Assume that $ x \in H^d(X; \mathbb{F}_p) \otimes \rho $. Then for all $ m \geq 2 $, the class
$ \iota_X(x)\dip{m}$ is in $ \hdp(X)^{dm,m,e} $, where $ e = 0 $ if $ d $ is even, $ e = 1 $ if $ d $ is odd.
Almost by construction, $ \iota_X(x)\dip{m} $ is represented by the cochain
\[
W^{\si_m}_0 \stackrel{\varepsilon}{\to} \mathbb{F}_p \stackrel{x^{\otimes m}}{\to} H^*(X; \mathbb{F}_p)^{\otimes m} \otimes \sgn^e,
\]
where $ W^{\si_m}_* \stackrel{\varepsilon}{\to} \mathbb{F}_p \to 0 $ is a free resolution of $ \mathbb{F}_p $ as a $ \mathbb{F}_p[\si_m] $-module. This is a well-defined cocycle because $ x^{\otimes m} $ is invariant under the action of $ \si_m $ in $ H^*(X;\mathbb{F}_p)^{\otimes m} \otimes \sgn^e $.
If $ q_I $ is different from $ q_0 \dots q_0 $, then $ q_I(\alpha) \in H_*(\si_{p^k}; H_*(X; \mathbb{F}_p)^{\otimes p^k} \otimes_{\si_{p^k}} \rho) $ is represented by a cycle of the form $ y \otimes_{\si_{p^k}} \alpha^{\otimes p^k} $, with $ y \in W_l^{\si_{p^k}} $ and $ l > 0 $. Such a cycle must pair trivially with the cochain representative for $ \iota_X(x)\dip{p^k} $ above.

Similarly, $ \iota_X(x')\dip{n} $ can be represented at the cochain level by 
$ W_0^{\si_n} \stackrel{\varepsilon}{\to} \mathbb{F}_p \stackrel{{x'}^{\otimes n}}{\to} H^*(X; \mathbb{F}_p)^{\otimes n} \otimes \sgn^{e(x')} $.
Since the diagonal of the resolution $ W_* $ must preserve the augmentation $ \varepsilon $, composing the tensor product of the cochains 
representing $ \iota_X(x)\dip{n} $ and $ \iota_X(x')\dip{n} $ with the diagonal we obtain
$ W_0^{\si_n} \stackrel{\varepsilon}{\to} \mathbb{F}_p \stackrel{(xx')^{\otimes n}}{\to} H^*(X; \mathbb{F}_p)^{\otimes n} \otimes \sgn^{e(x) + e(x')} $, which represents both $ \iota_X(x)\dip{n} \cdot \iota_X(x')\dip{n} $ and $ \iota_X(xx')\dip{n} $.

 We can assume  that $ y $ is tri-homogeneous and that $ k \geq 2 $. 
 If $ t(y) $ is odd, then the statement is trivial because both $ (x\dip{n} \cdot y)\dip{k} $ and 
 $ x\dip{nk} \cdot y\dip{k} $ are zero, so we assume that $ t(y) $ is even. Represent $ y $ by a 
 $ \si_{n(y)} $-equivariant cocycle $ f \colon W^{\si_{n(y)}}_{d(y)} \to \sgn^{e(y)} $. 
 Let $ \tr \colon W^{\si_{n(y)k}}_* \to W^{\si_k} \otimes (W^{\si_{n(y)}}_*)^{\otimes k} $ be a chain-level representative of the transfer map. 
 Using our  representative of $ \iota_X(x)\dip{m} $  and 
 combining with the diagonal map, we immediately have that both $ \iota_X(x)\dip{n(y)k} \cdot \pi(y)\dip{k} $ and $ (\iota_X(x)\dip{n(y)} \cdot \pi(y))\dip{k} $ are represented by the same cocycle
\begin{center}
\begin{tikzcd}
W^{\si_{nk}}_* \arrow{r}{\tr} & W^{\si_k}_* \otimes (W^{\si_n(y)}_*)^{\otimes k} \arrow{r}{\varepsilon \otimes f^{\otimes k}} & \sgn^{e(y)} \\
\arrow{r}{x^{\otimes kn(y)}} & H^*(X; \mathbb{F}_p)^{\otimes kn(y)} \otimes_{\si_{kn(y)}} \sgn^e \otimes \sgn^{e(y)} \arrow{r}{\id \otimes \mu_{kn(y)}} & H^*(X; \mathbb{F}_p)^{\otimes kn(y)} \otimes_{\si_{kn(y)}} \sgn^{e + e(y)}.
\end{tikzcd}
\end{center}

\end{proof}

We can now prove the main result of this paper.

\begin{proof}[Proof of Theorems \ref{m3} and \ref{m4}]
We only prove the theorem for $ p > 2 $, because the assertion for $ p = 2 $ is similar  and significantly simpler.

Let $ H_{alg} $ be the component super-Hopf rig with additive divided powers having the algebraic presentation given. 
Since we know from Lemma \ref{lemma-partial-compatibility} (statements 2 and 3) that there are classes in $ \hdp(X) $ satisfying the 
relations stated by the theorem, there is a morphism of component super-Hopf rigs preserving the divided powers structure $ H_{alg} \to \hdp(X_+) $. 
We show that this map is an isomorphism.

First, as above we invoke  well-known results about structure of Hopf algebras \cite{Milnor-Moore,Andre:71} to see that the graded linear dual of a
 bicommutative divided powers super-Hopf algebra $ H $ is the free graded commutative algebra primitively generated by $ P(H)^{\vee} $,
  the dual space of the subspace of primitives $ P(H) \subseteq H $.
To prove that the map $ H_{alg} \to \hdp(X_+) $ is an isomorphism, it is enough to check that the induced pairing between $ P(H_{alg}) $ and the space $ Q $ of indecomposables of $ (\hdp(X_+))^{\vee} $ is perfect.  We achieve this by performing an ad-hoc analysis, which is is the most delicate part of the proof.

We need to have explicit descriptions of the space $ Q $ of indecomposables of $ \hdp(X_+)^\vee $ and the space $ P $ of primitives of $ H_{alg} $.
The graded dual of $ \hdp(X_+) $ is  described in terms of KADL operations, extended to the sign representation case as described in Section \ref{sec:symmetric group}. Precisely, $ Q $ has a basis $ \mathcal{B}_* $ given by classes of the form
\[
\beta^{\varepsilon_1} q_{i_1} \beta^{\varepsilon_2} q_{i_2} \dots \beta^{\varepsilon_r} q_{i_r}(\alpha),
\]
where $ (\varepsilon_1,i_1,\dots,\varepsilon_r,i_r) $ is a strongly admissible sequence of KADL operations, $ i_{j+1} - \varepsilon_{j+1} $ 
have the same parity for all $ j $, and where $ \alpha $ varies on a given additive basis of $ H_*(X; \mathbb{F}_p) $.
Since $ \hdp(X_+) $ is tri-graded and its structure maps preserve  degrees, $ (\hdp(X_+))^\vee $ and $ Q $ are also tri-graded. 
The homological dimension of a basis element $ \beta^{\varepsilon_1} q_{i_1} \beta^{\varepsilon_2} q_{i_2} \dots \beta^{\varepsilon_r} q_{i_r}(\alpha) $, 
when $ \alpha $ is tri-homogeneous, is computed from $ d(\alpha) $ via the usual dimension formulas for KADL operations. 
Its component is $ p^r n(\alpha) $. Its sign degree is $ 0 $ if $ d(\alpha) $ and $ i_r $ have the same parity, $ 1 $ otherwise.

$ P = P(H_{alg}) $ can be described directly in terms of the Hopf monomial basis of $ H_{alg} $. Indeed, by the construction of Section \ref{sec:algebraic basis}, a basis $ \mathcal{B}^* $ of $ P $ is given by the set of classes of the form $ b \cdot x\dip{p^n} $, where $ b \in H^*(\si_{p^n}; \rho) $ is a primitive gathered block in the twisted or untwisted cohomology of a symmetric group $ \si_{p^n} $ and $ x $ belongs to a given additive basis of $ H^*(X; \mathbb{F}_p) $.
The cohomological degree, the component and the sign degree of $ b \cdot x\dip{p^n} $ in $ H_{alg} $ are $ d(b) + d(x)p^n $, $ p^n $, $ e(b)+d(x) $ (modulo $ 2 $) respectively.

By the description above there are isomorphism $ P_{\{*\}} \otimes H^*(X; \mathbb{F}_p) \cong P $ and $ Q_{\{*\}} \otimes H_*(X; \mathbb{F}_p) \cong Q $ given by $ b \otimes x \mapsto b \cdot x\dip{n} $ and $ q_I \otimes \alpha \mapsto q_I(\alpha) $, respectively.
We claim that, under these isomorphisms, the pairing between $ P $ and $ Q $ is equivalent to the tensor product of the Kronecker pairings $ P_{\{*\}} \otimes Q_{\{*\}} \to \mathbb{F}_p $ and $ H^*(X; \mathbb{F}_p) \otimes H_*(X; \mathbb{F}_p) \to \mathbb{F}_p $. This claim
reduces the calculation to the special case $ X = \{*\} $ already discussed in Section \ref{sec:symmetric group} and would thus complete the proof.

To prove this claim, we first use May's formulas for the coproduct $ \Delta_\cdot $ dual to cup product of KADL operations and we obtain that, for all $ b \in P_{\{*\}} $, $ x \in H^*(X; \mathbb{F}_p) $, $ q_I \in Q_{\{*\}} $, $ \alpha \in H_*(X; \mathbb{F}_p) $,
\[
\langle b \cdot x\dip{n}, q_I(\alpha) \rangle = \sum_{J+K=I} \sum_{\Delta_{\cdot}(\alpha) = \alpha_{(1)} \otimes \alpha_{(2)}} \pm \langle b \otimes x\dip{n}, q_J(\alpha_{(1)}) \otimes q_K(\alpha_{(2)}) \rangle.
\]
By Statement 1 of Lemma \ref{lemma-partial-compatibility}, $ \langle x\dip{n}, q_K(\alpha_{(2)}) \rangle $ is zero unless $ q_K = q_0\dots q_0 $ and the homological degrees of $ \alpha_{(2)} $ is the same as the cohomological degree of $ x $. Moreover, being $ b $ the pullback of a class from $ \hdp(S^0) $, $ \langle b, q_J(\alpha_{(1)}) \rangle $ is $ 0 $ unless the homological degree of $ \alpha_{(1)}$ is  $0 $. These conditions can only be satisfied for a single addend: $ \alpha = \alpha_{(2)} $, $ \alpha_{(1)} $ is a zero-dimensional class, $ q_J = q_I $ and $ q_K = q_0 \dots q_0 $, for which we obtain $ \pm \langle b, q_I \rangle \langle x, \alpha \rangle $.
\end{proof}

%\BComment{Maybe the exposition would be clearer if we state the proof of (1) and the reduction to a tensor product of pairings in (3) as two different lemmas. What do you think? -LG}
%\CComment{Yes this might help -PS}

%% file: DX-QXSection6.tex
% !TEX root = DX-QX.tex

\section{Further cohomology ring calculations} \label{CXQX}
We use our description of $ H^*(\tilde{D}X; \mathbb{F}_p) $ to obtain a presentation of  the related algebras
$ H^*(DX; \mathbb{F}_p) $, $ H^*(CX; \mathbb{F}_p) $ and $ H^*(QX; \mathbb{F}_p) $. 
These results are all new except for $ H^*(QX; \mathbb{F}_2) $
for which an equivalent result has been obtained by Dung \cite{Dung}. 
The techniques used in that paper are different from ours, though, and we believe that our approach leads to a simpler proof.

\subsection{General extended powers}
First, we describe $ \hdp(X) $ when $ X $ is not necessarily obtained from a topological space by adjoining a disjoint basepoint.
The following result is a consequence of our main theorems.
\begin{corollary}
Let $ X $ be a pointed topological space. Let $ \mathcal{B}' $ be a graded basis for $ \tilde{H}^*(X; \mathbb{F}_p) $ and let $ \mathcal{B} = \mathcal{B}' \cup \{1_X\} $ be the basis of the unreduced cohomology of $ X $ obtained by adding the unit of $ H^*(X_{[*]}; \mathbb{F}_p) $, where $ X_{[*]} $ is the connected component of $ X $ containing the basepoint.
$ \hdp(X) $ is isomoprhic to the sub-Hopf rig with additive divided powers of $ \hdp(X_+) $ consisting of the unit and the linear span of decorated Hopf monomials (with respect to the basis $ \mathcal{B} $) with decorations different from $ 1_X $.
\end{corollary}
\begin{proof}
The map $ \pi \colon X_+ \to X $ that sends $ * $ to the basepoint of $ X $ induces a map $ \pi^* \colon \hdp(X) \to \hdp(X_+) $ that is both a 
morphism of Hopf rings and of divided powers structures. The dual map $ \pi_*$ in homology is an epimorphism of algebras whose kernel is the ideal
 generated by $ Q_I(*) $, with $ I $ admissible and non-empty. Consequently, $ \pi^* $ must be injective. From the proof of Theorem \ref{m4} 
 given in the previous section, we see that the Kronecker pairing between $ \hdp(X_+) $ and the homology of $ D(X_+) $ with coefficient 
 in $ \rho $ restricts to a perfect pairing between the primitive gathered blocks not decorated with $ 1_X $ and indecomposables of the homology 
 of $ D(X_+) $ not in $ \ker(\pi_*) $.  We deduce that the image of $ \pi^* $ must be the sub-Hopf ring identified in the statement of the corollary. 
\end{proof}

\subsection{Infinite extended powers}

Given a pointed topological space $ X $ with basepoint $ * $, we take into consideration the space $ D_\infty X $ defined as
\[
\tilde{D}_\infty X = \left\{ p \times_{\mathcal{S}_\infty} (x_1,\dots,x_n,\dots): | \{n: x_n \not= * \}| < \infty \right\} \subseteq E(\mathcal{S}_{\infty}) \times_{S_{\infty}} X^{\mathbb{N}},
\]
where $ \mathcal{S}_\infty = \varinjlim_n \mathcal{S}_n $ is the infinite symmetric group. 
(In Dung's paper, this space is denoted by $ \mathcal{S}_\infty \wr X $.) 
The calculation of the cohomology of $ \tilde{D}_\infty X $ is a straightforward consequence of our main theorems.

\begin{corollary}
Let $ X $ be a topological space and let $ \mathcal{B} $ be a basis of $ H^*(X; \mathbb{F}_p) $ as an $ \mathbb{F}_p $-algebra that contains 
$ 1_{H^*(X)} $ and extends an additive basis of the reduced cohomology of $ X $.
For all $ n \leq m$, let $ \rho_{n,m} \colon H^*(\tilde{D}_mX; \mathbb{F}_p) \to H^*(\tilde{D}_nX; \mathbb{F}_p) $ be the vector space homomorphism 
that maps a Hopf monomial $ x $ into $ 0 $ if $ x $ does not have a constituent block of the form $ 1_{H^*(X)}^{[r]} $ with $ r \geq m-n $, or into 
$ y \odot 1_{H^*X}^{[r-m+n]} $ if $ x = y \odot 1_{H^*X}^{[r]} $, where $ y $ is another Hopf monomial not containing any divided power of $ 1_{H^*X} $ 
as an $ \odot $-factor.

Then, all the maps $ \rho_{n,m} $ are ring homomorphisms and form an inverse system of $ \mathbb{F}_p $-algebras. Moreover, the mod $ p $ cohomology ring of $ \tilde{D}_\infty X $ is naturally isomorphic to its inverse limit
\[
\varprojlim_{n} H^*(\tilde{D}_nX; \mathbb{F}_p).
\]
\end{corollary}
\begin{proof}
By construction $ \tilde{D}_\infty X $ is the direct limit of the topological spaces $ \tilde{D}_n X $. 
The embedding of $ \tilde{D}_n X $ into $ \tilde{D}_m X $ for $ n \leq m $ is given by the functions
\[
f_{n,m} \colon \tilde{D}_n X = E(\mathcal{S}_n) \times_{\mathcal{S}_n} X^n \to E(\mathcal{S}_m) \times_{\mathcal{S}_m} X^m = \tilde{D}_m X, \quad (n \leq m)
\]
corresponding to the standard monomorphism $ \mathcal{S}_n \hookrightarrow \mathcal{S}_m $, and the inclusion $ X^n \hookrightarrow X^m $ that maps $ X^n $ identically onto the first $ n $ coordinates of $ X^m $ and makes the last $ m-n $ coordinates equal to the basepoint.

For all $ n \leq m $, the induced morphism $ f_{n,m}^* \colon H^*( \tilde{D}_mX; \mathbb{F}_p) \to H^*( \tilde{D}_nX; \mathbb{F}_p) $ is equal to 
$ \rho_{n,m} $. This fact follows directly from Theorem \ref{m3} and Theorem \ref{m4} because we can identify $ f_{n,m}^* $ with the composition 
of the iterated coproduct
\[
H^*(\tilde{D}_m X; \mathbb{F}_p) \to H^*( \tilde{D}_nX; \mathbb{F}_p) \otimes H^*( \tilde{D}_1X; \mathbb{F}_p)^{\otimes^{m-n}} \cong H^*( \tilde{D}_n X; \mathbb{F}_p) \otimes H^*(X; \mathbb{F}_p)^{\otimes^{m-n}}
\]
with $ \id_{H^*( \tilde{D}_nX; \mathbb{F}_p)} \otimes \varepsilon^{\otimes^{m-n}} $, where $ \varepsilon \colon H^*( \tilde{D}_1X; \mathbb{F}_p) \cong H^*(X; \mathbb{F}_p) \to \mathbb{F}_p $ is the augmentation given by evaluation against the class of the basepoint of $ X $.

Since the maps $ \rho_{n,m} $ are epimorphisms, taking the cohomology of the given limit of spaces yields the desired isomorphism between $ H^*( \tilde{D}_\infty X; \mathbb{F}_p ) $ and $ \varprojlim_{n} H^*(\tilde{D}_nX; \mathbb{F}_p) $.
% (see, for example, Section 19.4 of \cite{May:99}).
\end{proof}

The previous corollary provides a practical way to calculate cup products in the cohomology of $ \tilde{D}_\infty X $ by restriction
 to $ \tilde{D}_n X $ with $ n $ finite. Suppose $ X $ is connected, and $ \alpha_1, \alpha_2 \in H^*(\tilde{D}_\infty X; \mathbb{F}_p) $, 
 are cohomology classes of dimension $ d_1 $ and $ d_2 $ respectively. 
 The map $ \rho_{n,m} $ is an isomorphism in dimension less than or equal to $ d_1 + d_2 $, provided that $ n,m \geq d_1+d_2 $. 
 This fact is easily seen directly, but we can also view it as a consequence of homological stability. 
 Therefore the natural restriction morphism $ \rho_n \colon H^{d_1+d_2}( \tilde{D}_\infty X; \mathbb{F}_p) \to H^{d_1+d_2}( \tilde{D}_n X; \mathbb{F}_p) $ is an isomorphism for all $ n \geq d_1 + d_2 $.
  Hence, to calculate $ \alpha_1 \cdot \alpha_2 $ is sufficient to apply $ \rho_{d_1+d_2} $ and 
  perform the computation in $ H^*( \tilde{D}_{d_1+d_2}; \mathbb{F}_p) $.

 We use this method to extract a complete set of generators and relations for $ H^*(\tilde{D}_\infty X; \mathbb{F}_p) $.
In what follows, until the end of this section, we will always assume that $ \mathcal{B}' $ is an additive basis of $ \tilde{H}^*(X; \mathbb{F}_p) $ as an $ \mathbb{F}_p $-vector space, and that $ \mathcal{B} = \mathcal{B}' \cup \{ 1 \} $ is the basis of the unreduced cohomology of $ X $ obtained by adding the unit $ 1 \in H^*(X_{[*]}; \mathbb{F}_p) $, where $ X_{[*]} $ is the connected component of $ X $ containing the basepoint. In this section, we will often consider only connected spaces $ X $. In those cases, $ 1 $ will be the unit of the cohomology ring of $ X $.

\begin{definition}
Let $ \mathcal{M} $ be the skyline diagrams basis of $ H^*(\tilde{D}X; \mathbb{F}_p) $ constructed from $ \mathcal{B} $. The
 {\bf effective width} of an element $ x \in \mathcal{M} $ is the width of the skyline diagram obtained from $ x $ by removing any column with decoration 
 $ 1 $ and height $ 0 $. In terms of gathered monomials, the effective width of $ x $ is the component of the unique gathered monomial 
 $ y $ such that $ x = y \odot 1^{[r]} $ for some $ r \geq 0 $.
We say that $ x \in \mathcal{M} $ is {\bf pure} is its width is equal to its effective width or, equivalently, if 
$ x $ does not have any constituent block of the form $ 1^{[r]} $.
\end{definition}

\begin{lemma}
Let $ x \in H^*(\tilde{D}_nX; \mathbb{F}_p) $ be a pure Hopf monomial. There is a class $ x \odot 1^{[*]}$  in 
$ \varprojlim_{n} H^*(\tilde{D}_nX; \mathbb{F}_p) \cong H^*( \tilde{D}_\infty X; \mathbb{F}_p ) $ which maps to $x \odot 1^{[r]}$ in 
$H^*(D_m X)$ when $m = n+r$.
%whose representative 
 %(0,\dots,0, x, x \odot 1^{[1]}, x \odot 1^{[2]}, \dots ) \in \prod_{n=0}^{\infty} H^*( \tilde{D}_n X; \mathbb{F}_p) $. If $ x $ is pure, the morphisms 
 %$ \rho_{m,l} $ preserve the entries of this sequence, and thus $ x \odot 1^{[*]} $ defines an element in the inverse limit 
 % $ \varprojlim_{n} H^*(\tilde{D}_nX; \mathbb{F}_p) \cong H^*( \tilde{D}_\infty X; \mathbb{F}_p ) $, that is 
 % denoted by the same symbol with a slight abuse of notation.
\end{lemma}

\begin{proof}
It is sufficient to check that the following equalities hold to see there is such a class in the inverse limit
\begin{itemize}
\item $ \rho_{n-1,n}(x) = 0 $
\item $ \forall m > 0: \rho_{n+m-1,n+m}(x \odot 1^{[m]}) = x \odot 1^{[m-1]} $.
\end{itemize}

To prove the first identity, remember that $ \rho_{n-1,n} $ is equal to the iterated coproduct composed with augmentation $ \varepsilon$ on one tensor factor.
%$ H^*(\tilde{D}_nX; \mathbb{F}_p) \to H^*(\tilde{D}_{n-1}X; \mathbb{F}_p) \otimes H^*(X; \mathbb{F}_p) $ composed with 
% $ \id_{H^*(\tilde{D}_{n-1}X; \mathbb{F}_p)} \otimes \varepsilon $, where  \colon H^*(X; \mathbb{F}_p) \to \mathbb{F}_p $ is the augmentation.
Let $ V $ be the $ \mathbb{F}_p $-subspace of $ H^*(\tilde{D}X; \mathbb{F}_p) $ generated by pure Hopf monomials.
Since the coproduct preserves the decorations of columns, $ \Delta(V) \subseteq V \otimes V $. Thus $ \rho_{n-1,n} $ is equal to a sum of terms of the form $ \varepsilon (y) z $, where $ y \in H^*(X; \mathbb{F}_p) $ and $ z \in H^*(\tilde{D}_{n-1}X; \mathbb{F}_p) $ are pure Hopf monomials. By definition, pure Hopf monomials in the cohomology of $ X \simeq \tilde{D}_1X $ are the elements of $ \mathcal{B}' $, and, as a consequence, they belong to the kernel of $ \varepsilon $. Hence $ \rho_{n-1,n} (x) = 0 $.

In order to prove the latter identity, we observe that $ \rho_{n+m-1,n+m}( x \odot 1^{[m]} ) $ is the image of $ x $ under the morphism
\begin{center}
\begin{tikzpicture}[baseline= (a).base]
\node[scale=0.94] (a) at (0,0){
\begin{tikzcd}
H^*(\tilde{D}_nX; \mathbb{F}_p) \arrow{r}{\_ \otimes 1^{[m]}} & H^*(\tilde{D}_nX; \mathbb{F}_p) \otimes H^*(\tilde{D}_mX; \mathbb{F}_p) \arrow{r}{\odot} & H^*(\tilde{D}_{n+m}X; \mathbb{F}_p) \\ \arrow{r}{\Delta_{n+m-1,1}} & H^*(\tilde{D}_{n+m-1}X; \mathbb{F}_p) \otimes H^*(X; \mathbb{F}_p) \arrow{r}{\id \otimes \varepsilon} & H^*(\tilde{D}_{n+m-1}X; \mathbb{F}_p).
\end{tikzcd}
};
\end{tikzpicture}
\end{center}
Since $ \Delta $ and $ \odot $ form a bialgebra, this is equal to the sum of the following two composition of maps, where $ \tau $ exchanges two factors:
\begin{center}
\begin{tikzpicture}[baseline= (a).base]
\node[scale=0.71] (a) at (0,0){
\begin{tikzcd}
H^*(\tilde{D}_nX; \mathbb{F}_p) \arrow{r}{\_ \otimes 1^{[m]}} & H^*(\tilde{D}_nX; \mathbb{F}_p) \otimes H^*(\tilde{D}_mX; \mathbb{F}_p) \arrow{r}{\id \otimes \Delta_{m-1,1}} & H^*(\tilde{D}_nX; \mathbb{F}_p) \otimes H^*(\tilde{D}_{m-1}X; \mathbb{F}_p) \otimes H^*(X; \mathbb{F}_p) \\
\arrow{r}{\odot \otimes \id} & H^*(\tilde{D}_{n+m-1}X; \mathbb{F}_p) \otimes H^*(X; \mathbb{F}_p) \arrow{r}{\id \otimes \varepsilon} & H^*(\tilde{D}_{n+m}X; \mathbb{F}_p)\\
H^*(\tilde{D}_nX; \mathbb{F}_p) \arrow{r}{\_ \otimes 1^{[m]}} & H^*(\tilde{D}_nX; \mathbb{F}_p) \otimes H^*(\tilde{D}_mX; \mathbb{F}_p) \arrow{r}{\Delta_{n-1,1} \otimes \id} & H^*(\tilde{D}_{n-1}X; \mathbb{F}_p) \otimes H^*(X; \mathbb{F}_p) \otimes H^*(\tilde{D}_mX; \mathbb{F}_p) \\
\arrow{r}{(\odot \otimes \id) \circ \tau} & H^*(\tilde{D}_{n+m-1}X; \mathbb{F}_p) \otimes H^*(X; \mathbb{F}_p) \arrow{r}{\id \otimes \varepsilon} & H^*(\tilde{D}_{n+m}X; \mathbb{F}_p)
\end{tikzcd}
};
\end{tikzpicture}
\end{center}
The first composition is equal to $ \_ \otimes 1^{[m-1]} $, because $ \Delta_{m-1,1}(1^{[m]}) = 1^{[m-1]} \otimes 1 $ and $ \varepsilon(1) = 1 $. Pure Hopf monomials belong to the kernel of the second one, because, as we have already seen, the coproduct preserves the subspace generated by them and every pure monomial in $ H^*(X; \mathbb{F}_p) $ lies in the kernel of $ \varepsilon $.
Thus, $ \rho_{n+m-1,n+m} (x \odot 1^{[m]}) = x \odot 1^{[m-1]} $.
\end{proof}

\begin{lemma} \label{lem:basis infinite wreath}
The set $ \{ x \odot 1^{[*]}: x \in \mathcal{M} \mbox{ pure} \} $ is linearly independent in $ H^*(\tilde{D}_\infty X; \mathbb{F}_p) $. Moreover, if $ X $ is connected, it is a basis for that vector space.
\end{lemma}
\begin{proof}
Given a finite set of elements of the form $ x \odot 1^{[*]} $, for $ N \in \mathbb{N} $ large enough, their restriction to 
$ H^*(\tilde{D}_NX; \mathbb{F}_p) $ are linearly independent, which implies the linear independence of the original elements.
If $ X $ is connected, given $ k \in \mathbb{N} $, for $ N $ large enough dependent on $ k $ the map $ \rho_N \colon H^k(\tilde{D}_\infty X; \mathbb{F}_p) \to H^k(\tilde{D}_N X; \mathbb{F}_p) $ is an isomorphism, thus the given set is a basis for $ H^k( \tilde{D}_\infty X; \mathbb{F}_p) $, 
because it restricts to a basis for $ H^k( \tilde{D}_N X; \mathbb{F}_p) $.
\end{proof}

If $ X $ is not connected, we can reduce the computation of $ H^*(\tilde{D}_\infty X; \mathbb{F}_p) $ to the connected case as follows. 
Let $ \{ X_\alpha \}_{\alpha \in \pi_0(X)} $ be the set of the connected components of $ X $. Let $ [*] $ denote the element of $ \pi_0(X) $ corresponding to the component containing the basepoint. Then there is a natural homotopy equivalence
\[
\tilde{D}_\infty X \simeq \prod_{\alpha \in \pi_0(X) \setminus \{ [*] \}} \tilde{D} X_{\alpha} \times \tilde{D}_\infty X_{[*]},
\]
which induces an isomorphism of the cohomology algebras
\[
H^*(\tilde{D}_\infty X; \mathbb{F}_p) \cong \prod_{\{n_\alpha\}_{\alpha} \in \mathbb{N}^{\pi_0(X) \setminus \{[*]\}}} \bigotimes_{\alpha} H^*(\tilde{D}_{n_\alpha} X_\alpha; \mathbb{F}_p) \otimes H^*(\tilde{D}_\infty X_{[*]}; \mathbb{F}_p).
\]

We next analyze cup product computations for the classes $ x \odot 1^{[*]} $. 
As a technical tool, we will exploit the existence of an increasing filtration $ \{ F_n H^*(\tilde{D}_\infty X; \mathbb{F}_p) \}_{n \geq 0} $ of $ \mathbb{F}_p $-vector spaces, where $ F_n H^*(\tilde{D}_\infty X; \mathbb{F}_p) $ is the linear span of the elements $ x \odot 1^{[*]} $ arising from pure Hopf monomials $ x $ with a width lower than or equal to $ n $.
Lemma \ref{lem:basis infinite wreath} guarantees that $ \bigcup_{n \geq 0} F_n H^*(\tilde{D}_\infty X; \mathbb{F}_p) = H^*(\tilde{D}_\infty X; \mathbb{F}_p) $ when $ X $ is connected.
\begin{lemma} \label{lem:cup product blocks limit}
Let $ b_1, \dots, b_r $ be primitive gathered block in $ H^*(\tilde{D}X; \mathbb{F}_p) $ and let $ n_1, \dots, n_r $ be strictly positive integers. Then
\[
\prod_{i=1}^r (b_i^{[n_i]} \odot 1^{[*]}) - b_1^{[n_1]} \odot \dots \odot b_r^{[n_r]} \odot 1^{[*]} \in F_{w-1}H^*(\tilde{D}_\infty X; \mathbb{F}_p),
\]
where $ w $ is the sum of the widths of $ b_1^{[n_1]}, \dots, b_r^{[n_r]} $.
\end{lemma}
\begin{proof}
The argument boils down to restricting to $ H^*(\tilde{D}_N X; \mathbb{F}_p) $ with $ N \in \mathbb{N} $ large enough and perform the calculations there, where the claim is a straightforward consequence of our Hopf ring presentation.

Explicitly, we observe that $ \Delta $ preserves the effective width. Formally,  if the effective width of a Hopf monomial $ x $ is $ w $, 
then its coproduct can be written as $ \Delta(x) = \sum_i x_i' \otimes x_i'' $ where the sum of the effective widths of $ x_i' $ and $ x_i'' $ is $ w $.
Moreover, our cup product formulas imply that given a non-trivial gathered block $ b \in H^*(\tilde{D}_nX; \mathbb{F}_p) $ and a Hopf monomial 
(not necessarily pure) $ x \in H^*(\tilde{D}_nX; \mathbb{F}_p) $ belonging to the same component, the cup product 
$ x \cdot b $ is always the sum of pure Hopf monomials (when it is not zero).
These two facts together with the Hopf ring distributivity formula guarantee that, given $ x \in H^*(\tilde{D}_nX; \mathbb{F}_p) $ with effective width $ w $, 
and a non-trivial gathered block $ b \in H^*(\tilde{D}_mX; \mathbb{F}_p) $ whose width is less than or equal to $ n - w $, the following equality holds:
\[
x \cdot (b \odot 1^{[n-m]}) = \sum_i (x_i' \cdot b) \odot x_i'',
\]
where the effective width of $ x_i'' $ is less than or equal to $ w $, and we have equality in a single case in which $ x_i' = 1^{[m]} $ and $ x_i'' = \rho_{n-m,n}(x) $. Therefore the difference $ x \cdot ( b \odot 1^{[n-m]}) - b \odot \rho_{n-m,n}(x) $ is a linear combination of Hopf monomials whose effective widths are strictly less than $ n $.

We can now prove our claim by induction on $ r $.
For $ r = 1 $ the statement is trivial. For $ r > 1 $, the previous argument with $ x = b_1^{[n_1]} \odot \dots \odot b_{r-1}^{[n_{r-1}]} \odot 1\dip{N}$ and $ b = b_r^{[n_r]} \odot 1\dip{M} $, with $ N,M \in \mathbb{N}$ big enough, completes the induction.
\end{proof}

\begin{lemma} \label{lem:cup product relations limit}
Let $ b \in H^*(\tilde{D}X; \mathbb{F}_p) $ be a gathered block. If $ p > 2 $ we also assume that $ b $ is even-dimensional. Then $ (b \odot 1^{[*]})^p = b^p \odot 1^{[*]} $.
\end{lemma}
\begin{proof}
First, we prove an auxiliary formula. For any $ a \in \mathbb{N} $, define $ \nu_j(a) \in \{0,1\} $ as the coefficient of $ 2^j $ in the dyadic expansion of $ a $. Given $ n,k,N \in \mathbb{N} $, let $ \mathcal{A}_{n,k,N} $ be the set of $ 2^k $-tuples $ (a_1, \dots, a_{2^k}) \in \mathbb{N}^{2^k} $ that satisfy the following three conditions:
\begin{itemize}
\item $ \forall 0 \leq j < k: \sum_{i=1}^{2^k} a_i \nu_j(2^k-i) = n $
\item $ \sum_{i=1}^{2^k} a_i = n+N $.
\end{itemize}
For all $ n,k \in \mathbb{N} $, for all $ N \in \mathbb{N} $ large enough, and for all $ b \in H^*(\tilde{D}X; \mathbb{F}_p) $ primitive gathered block of width $ w $, the following equality holds:
\[
(b^{[n]} \odot 1_X^{[Nw(b)]})^k = \sum_{\underline{a} \in \mathcal{A}_{n,k,N}} \bigodot_{i=1}^{2^k} (b^{[a_i]})^{\sum_{j=0}^{k-1} \nu_j(2^k-i)}.
\]
We prove this formula by induction on $ k $. For $ k = 1 $ it is trivial, and the induction step is a straightforward application of Hopf ring distributivity.

We now construct a permutation $ \sigma \in \mathcal{S}_{2^k} $ that maps $ \mathcal{A}_{n,k,N} $ into itself, when acting on 
$ \mathbb{N}^{2^k} $ by permutation of the $ 2^k $ coordinates. %, and we discuss this map.
Given $ \tau \in \mathcal{S}_k $, we define $ \overline{\tau} \in \mathcal{S}_{2^k} $ by the formula
\[
\forall 1 \leq i \leq 2^k: \overline{\tau}(i) = 2^k - \sum_{j=0}^{k-1} 2^j \nu_{\tau(j+1)-1}(2^k-i)
\]
We let $ \sigma = \overline{\tau} $ where $ \tau \in \mathcal{S}_k $ is a $ k $-cycle.
Observe that:
\begin{itemize}
\item $ \sigma $ has order $ k $ and its fixed points in $ \{ 1, \dots, 2^k \} $ are $ 1 $ and $ 2^k$;
\item $ \sigma( \mathcal{A}_{n,k,N} ) = \mathcal{A}_{n,k,N}$;
\item Two $ 2^k $-tuples in $ \mathcal{A}_{n,k,N} $ that belong to the same $ \sigma $-orbit give rise to addends in the formula above that differ only by a permutation of the transfer product factors.
\end{itemize}

These remarks imply that the fixed point of $ \sigma $ in $ \mathcal{A}_{n,k,N} $ are the $ 2^k $-tuples that are constant on the $ \sigma $-orbits of $ \{ 1, \dots, 2^k \} $.

We specialize to the case $ k = p $ to complete the proof. Since, under this condition, $ \sigma $ has order $ p $, the sum of all the addends in the 
previous equality that do not correspond to fixed points of $ \sigma $ in $ \mathcal{A}_{n,p,N} $ is zero, because of the commutativity of $ \odot $.
Moreover, the function $ \sum_{j=0}^{k-1} \nu_j (2^k - \_ ) $ is constant on the orbits of $ \sigma $ in $ \{ 1, \dots, 2^k \} $. 
Therefore, in the addends corresponding to a $ 2^k $-tuple $ \underline{a} \in \mathcal{A}_{n,k,N} $ such that $ \sigma(\underline{a}) = \underline{a} $, 
the factor $ (b^{[a_i]})^{\sum_{j=0}^{k-1} \nu_j(2^k-i)} $ appears at least $ p $ times if $ i \notin \{1, 2^k \} $.
Due to the properties of divided power algebras, all these addends are zero, except when $ a_i = 0 $ for all $ i \notin \{1, 2^k \} $.
 This leaves only one addend,  easily seen to be equal to $ (b^{[n]})^p \odot 1^{[Nw]} $. Passing to the limit completes the proof.
\end{proof}

We are finally ready to produce our presentation for $ H^*(\tilde{D}_\infty X; \mathbb{F}_p) $. In order to do so, we need to impose some restrictions on the basis $ \mathcal{B} $ for $ H^*(X; \mathbb{F}_p) $.
Recall that the Frobenius homomorphism makes $ H^*(X; \mathbb{F}_p) $ an abelian restricted Lie algebra over $ \mathbb{F}_p $. If it is of finite type, then it can be decomposed as a direct sum of monogenerated abelian restricted Lie algebras, of the form $ \mathbb{F}_p \{x,x^p, \dots,x^{p^k}\} $ with $ \deg(x) = d \in \mathbb{N} $, $ k \in \mathbb{N} $, and $ x^{p^{k+1}} = 0 $. The bases $ \{ x, x^p, \dots, x^{p^k} \} $ on those subspaces and the existence of that decomposition determine a basis $ \mathcal{B}_F $ of $ H^*(X; \mathbb{F}_p) $.
\begin{theorem}[after Dung \cite{Dung} for $ p=2 $] \label{thm:Dung}
Let $ X $ be a connected space of finite type. Let $ \mathcal{B}_F$ be the basis constructed above. Then $ H^*(\tilde{D}_\infty X; \mathbb{F}_p) $ is the graded commutative $ \mathbb{F}_p $-algebra generated by the set of classes of the form $ b \odot 1^{[*]} $, where $ b $ is a gathered block (or decorated column, in the language of skyline diagrams) whose width is a power of $ p $ satisfying at least one of the following conditions:
\begin{itemize}
\item the decoration of the column is not a $ p^{th} $ power in $ H^*(X; \mathbb{F}_p) $
\item at least one of the constituent rectangles of the column appears a number of times that is not divisible by $ p $
\item the column has a non-trivial solid part (if $ p > 2 $)
\end{itemize}
Moreover, a complete set of relations for these generators is given by equalities of the form $ (b \odot 1^{[*]})^{p^{h(b)}} = 0 $, with $ b $ even-dimensional. If $ x $ denotes the decoration of $ b $, we define $ h(b) $ through the following formula:
\[
h(b) = \left\{ \begin{array}{ll}
\min \{ n \in \mathbb{N}: x^{p^n} = 0 \} & \mbox{if } p = 2 \mbox{ or } b \mbox{ has no solid part} \\
1 & \mbox{otherwise}.
\end{array} \right.
\]
\end{theorem}
\begin{proof}
Lemma \ref{lem:cup product relations limit} guarantees that the given relations hold in $ H^*(\tilde{D}_\infty X; \mathbb{F}_p) $. 
Thus we need to prove two facts, namely that the given classes generate that cohomology algebra, and that there are no other independent relations.

In order to prove the first claim, we use the filtration $ \{ F_n H^*(\tilde{D}_\infty X; \mathbb{F}_p) \}_{n \geq 0} $ of the cohomology of $ \tilde{D}_\infty X $ by effective width defined previously.
We prove by induction on $ n $ that $ F_n H^*(\tilde{D}_\infty X; \mathbb{F}_p) $ is contained in the subspace $ V $ generated by our selected elements $ b \odot 1^{[*]} $.
For $ n = 0 $, the claim is obvious. So assume that $ n > 0 $ and that $ F_{n-1} H^*(\tilde{D}_\infty X; \mathbb{F}_p) \subseteq V $.

First, for any gathered block $ b $ whose width is a power of $ p $, $ b \odot 1^{[*]} \in V $, even if $ b $ does not satisfy any of the three required conditions. Indeed, because of Lemma \ref{lem:cup product relations limit}, any such element is the $ (p^{k})^{th} $ power, for some $ k \in \mathbb{N} $, of some other class of the form $ b' \odot 1^{[*]} $  arising from a gathered block $ b' $ that meets one of those conditions.

Second, for any pure Hopf monomial $ x $ of width $ n $, $ x $ can be written as a multiple of the transfer product of some gathered blocks $ b_1, \dots, b_r $ whose widths are powers of $ p $. By Lemma \ref{lem:cup product blocks limit} we have that
\[
x \in \prod_{i=1}^r ( b_i \odot 1^{[*]} ) + F_n H^*(\tilde{D}_\infty X; \mathbb{F}_p) \subseteq V
\]
and this proves our first claim.

In order to prove that the given relations suffice to describe $ H^*(\tilde{D}_\infty X; \mathbb{F}_p) $ completely, it is sufficient to check that the Poincar\'e series of this cohomology algebra and the graded commutative $ \mathbb{F}_p $-algebra $ A_\infty(X) $ defined by the given presentation are equal.

On the one hand, we observe that
\[
A_\infty(X) = \bigotimes_{b, \deg(b) \text{ even}} \frac{ \mathbb{F}_p [ b \odot 1^{[*]} ] }{ (b \odot 1^{[*]})^{h(b)} } \otimes \Lambda \{b, \deg(b) \text{ odd} \},
\]
where $ \Lambda $ is the exterior algebra functor and the tensor product is over gathered blocks $ b $ satisfying the conditions in the statement of this theorem.

On the other hand, the subspace (isomorphic to $ H^*(\tilde{D}_\infty X; \mathbb{F}_p) $ by Lemma \ref{lem:basis infinite wreath}) generated by pure Hopf monomials in $ H^*(\tilde{D}X; \mathbb{F}_p) $ is a subalgebra under the transfer product, and is isomorphic to
\[
\bigotimes_{b, \deg(b) \text{ even}} \frac{ \mathbb{F}_p [ b ] }{ (b^{\odot^p}) } \otimes \Lambda \{ b, \deg(b) \text{ odd} \}.
\]
Here the tensor product is performed over all the gathered blocks whose width is a power of $ p $ that are different from $ 1^{[p^k]} $, 
regardless of the stated conditions.

As we already observed, for every generator $ b \odot 1^{[*]} $ and every $ 1 \leq k < h(b) $, the power $ (b \odot 1^{[*]})^{p^k} $ is again a class of the form $ b' \odot 1^{[*]} $, where $ b' $ is another non-trivial gathered block whose width is a power of $ p $. 
Conversely, any such element can be obtained by iteratively applying the Frobenius homomorphism to one of the generators of $ A_{\infty}(X) $. 
This implies that in any degree $ d $, $ (A_\infty(X))_d $ and $ H^d(\tilde{D}X; \mathbb{F}_p) $ have the same dimension.
% Hence, the Poincar\'e series of these two graded algebras are equal.
\end{proof}

\subsection{Free $E_\infty$ spaces}
We move to the cohomology of $ CX $, the free $ E_\infty $-space generated by $ X $, as defined in Cohen--May--Lada \cite{May-Cohen}.

Let $ X_+ $ be the pointed topological space obtained from $ X $ by adding a disjoint basepoint $ * $. Observe that $ C(X_+) \cong D(X_+) \cong \tilde{D}X $.
 The map $ p \colon X_+ \to X $ that sends $ * $ to the basepoint of $ X $ induces a surjective function $ Cp \colon \tilde{D}X \cong C(X_+) \to CX $. 
 This map factors through the projection $ \tilde{D}X \to \tilde{D}_\infty X $.
 
The description of the functor $ H_*(C\_; \mathbb{F}_p) $ given in \cite[part I, Theorem 4.1]{May-Cohen} and the surjectivity of the 
map $ p_* \colon H(X_+; \mathbb{F}_p) \to H_*(X; \mathbb{F}_p) $ imply that the induced morphism in homology 
$ (Cp)_* \colon H_*(\tilde{D}X; \mathbb{F}_p) \to H_*(CX; \mathbb{F}_p) $ and, as a direct consequence, 
$ H_*(\tilde{D}_\infty X; \mathbb{F}_p) \to H_*(CX; \mathbb{F}_p) $, are epimorphisms.
Dually, this means that $ H^*(CX; \mathbb{F}_p) \to H^*(\tilde{D}_\infty X; \mathbb{F}_p) $ is a monomorphism. 
Therefore, we identify $ H^*(CX; \mathbb{F}_p) $ with a subring of $ H^*(\tilde{D}_\infty X; \mathbb{F}_p) $.

At the level of homology, $ (Cp)_* \colon H_*(\tilde{D}X; \mathbb{F}_p) \to H_*(CX; \mathbb{F}_p) $ is the unique morphism of algebras over the
 Dyer--Lashof algebra that extends the map $ p_* \colon H_*(X_+; \mathbb{F}_p) \to H_*(X; \mathbb{F}_p) $. Therefore, using homology operations, its kernel can be described as the ideal generated by $ Q_I([x_0] ) $ with $ I \not= () $ and $ 1 - [x_0] $, $ [x_0] \in H_*(X; \mathbb{F}_p) $ being the class in $ H_0(X; \mathbb{F}_p) $ corresponding to the basepoint $ x_0 \in X $. Dualizing yields the following.
 
\begin{corollary}
Let $ X $ be a connected space of finite type. Then $ H^*(CX; \mathbb{F}_p) $ is naturally isomorphic to the subring of $ H^*(\tilde{D}_\infty X; \mathbb{F}_p) $ generated by classes of the form $ ( b \odot 1^{[*]} ) $ satisfying the criteria stated in Theorem \ref{thm:Dung} and the  additional condition
that 
the decoration of $ b $ is different from $ 1 $.

Moreover, the relations $ (b \odot 1^{[*]})^{p^{h(b)}} = 0 $ among those generators are sufficient to give a complete description of $ H^*(CX; \mathbb{F}_p) $ 
as an algebra.
\end{corollary}

Although the previous corollary holds only for connected spaces $ X $, we can reduce the general calculation of $ H^*(CX; \mathbb{F}_p) $ 
to this particular case. It is enough to note that, if $ \{ X_\alpha \}_{\alpha \in \pi_0(X)} $ is the set of the connected components of $ X $, 
then there is a natural homeomorphism $ CX \cong \prod_{\alpha \in \pi_0(X) \setminus \{[*]\}} \tilde{D}(X_\alpha) \times C(X_{[*]}) $, and 
to apply the K\"unneth isomorphism.

\subsection{Free infinite loop spaces}

We conclude this section with an analogous description of the cohomology of $ QX $. When $ X $ is connected, $ CX $ and $ QX $ are also connected and the natural map $ CX \to QX $ induces an isomorphism in cohomology. 
When $ X $ is not connected, $ CX $ is not connected, and its commutative H-space structure makes $ \pi_0(CX) $ a commutative monoid.
Using the description above, we immediately see that $ \pi_0(CX) \cong \mathbb{N}^{\pi_0(X) \setminus \{[*]\}} $ and that the component corresponding to 
$ \underline{k} = \{ k_{\alpha} \}_{\alpha \in \pi_0(X) \setminus \{[*]\}} $ is homeomorphic to 
$ \prod_{\alpha} \tilde{D}_{k_\alpha} (X_\alpha) \times C(X_{[*]}) $. We name this component $ (CX)_{\underline{k}} $.

We recall the existence of a natural isomorphisms of Hopf algebras 
$ H_*(QX; \mathbb{F}_p) \cong H_*(CX; \mathbb{F}_p) \otimes_{\pi_0(CX)} \mathbb{F}_p[G] $, where 
$ G = \mathbb{Z}^{\pi_0(X) \setminus \{[*]\}} $ is the group completion of $ \pi_0(CX) $, and $ \mathbb{F}_p[G] $ is its group algebra over $ \mathbb{F}_p $. This isomorphism is classically well known and is discussed in \cite[part I, Theorem 4.2]{May-Cohen}.

Now we can observe that $ H_*(CX; \mathbb{F}_p) \otimes_{\pi_0(CX)} \mathbb{F}_p[G] $ is a direct sum of coalgebras 
$ \bigoplus_{g \in G} H_g $ indexed by $ G $. 
Explicitly, for all $ g \in G $, $$ H_g = \sum_{h \in \pi_0(CX)} H_*((CX)_{h}; \mathbb{F}_p) \otimes_{\pi_0(CX)} (gh^{-1}).$$ 
Let $ \Pi $ be the poset $ \mathbb{N}^{\pi_0(X) \setminus \{[*]\}} = \pi_0(CX) $ with order
\[
\{k_{\alpha}\}_{\alpha} \leq \{n_{\alpha}\}_{\alpha} \Leftrightarrow \forall \alpha \in \pi_0(X) \setminus \{[*]\}: k_{\alpha} \leq n_{\alpha}.
\]
For all $ \underline{k}, \underline{n} \in \Pi $ with $ \underline{k} \leq \underline{n} $, let $ f_{\underline{k},\underline{n}} \colon H_*((CX)_{\underline{k}}; \mathbb{F}_p) \to H_*((CX)_{\underline{n}}; \mathbb{F}_p) $ be defined by multiplication with $ [x_{\underline{n} - \underline{k}}] $, the class of a point $ x_{\underline{n} - \underline{k}} \in (CX)_{\underline{n} - \underline{k}} $. These maps define a direct system, whose limit is isomorphic to $ H_g $ for all $ g \in G $.
The topological counterpart of this is the homotopy equivalence $ QX \simeq Q_0X \times G $,
 where $ Q_0X $ is the connected component of $ QX $ containing the basepoint. Indeed, the direct limit above is isomorphic to the homology of $ Q_0X $.

Therefore describing the cohomology of $ QX $ is equivalent to calculate the dual of that direct limit, i.e. the inverse limit of the inverse system 
made by the dual spaces. Note that the morphisms 
$ f_{\underline{k},\underline{n}}^* \colon H^*((CX)_{\underline{n}}; \mathbb{F}_p) \to H^*((CX)_{\underline{k}}; \mathbb{F}_p) $ 
are analogous to the maps $ \rho_{n,m} $ used to compute 
$ H^*( \tilde{D}_\infty X; \mathbb{F}_p) $ in Lemma \ref{lem:basis infinite wreath}, but with blocks of the form $ 1^{[r]} $ replaced with any gathered block of dimension zero.  Explicitly, they are of the form $ 1_{H^*(X_\alpha; \mathbb{F}_p)}^{[r]} $ for some $ \alpha \in \pi_0(X) $ if the basis $ \mathcal{B} $ 
contains all the units $ 1_{H^*(X_\alpha; \mathbb{F}_p)} $). Thus, we can replicate the argument used to compute the cohomology of $ \tilde{D}_\infty X $
 to determine $ H^*(QX; \mathbb{F}_p) $, by replacing pure Hopf monomials with Hopf monomials that do not contain columns of 
 height zero or with decoration $ 1 $. We give below the precise statements.

\begin{definition}
Let $ x $ be a Hopf monomial in $ H^*(\tilde{D}X; \mathbb{F}_p) $. We say that $ x $ is {\bf full-width} if it does not have any gathered block of 
dimension $ 0 $.
\end{definition}

\begin{lemma}
Assume that the chosen basis $ \mathcal{B} $ of $ H^*(X; \mathbb{F}_p) $ contains $ 1_\alpha = 1_{H^*(X_\alpha;\mathbb{F}_p)} $ for all $ \alpha \in \pi_0(X) $. Then, to every full-width Hopf monomial without constituent blocks with $ 1_{[*]} $ as decoration, we can associate an element $ 
x_\infty \in \varprojlim_{\underline{k} \in \Pi} H^*((CX)_{\underline{k}}; \mathbb{F}_p) \cong H^*(Q_0X; \mathbb{F}_p) $ defined by the formula
\[
(x_\infty)|_{H^*((CX)_{\underline{k}}; \mathbb{F}_p)} = \left\{ \begin{array}{ll}
(x \odot \bigodot_{\alpha \in \pi_0(X) \setminus \{[*]\}} 1_{\alpha}^{[k_{\alpha}-n_{\alpha}]}) \odot 1^{[*]} & \mbox{if } \underline{k} \geq \underline{n}_x \\
0 & \mbox{otherwise}
\end{array} \right.
\]
where $ \underline{n}_x $ is the unique element of $ \Pi = \pi_0(CX) $ such that $ x \odot 1^{[*]} \in (CX)_{\underline{n}_x} $.
These elements $ x_\infty $ form a basis of $ H^*(Q_0X; \mathbb{F}_p) $ as an $ \mathbb{F}_p $-vector space.
\end{lemma}

\begin{theorem}
Chose the basis $ \mathcal{B} $ of $ X $ as we did for Theorem \ref{thm:Dung}. Moreover, assume that $ 1_{H^*(X_\alpha; \mathbb{F}_p)} \in \mathcal{B} $ for all $ \alpha \in \pi_0(X) $.
Consider the set of elements $ b_\infty $, where $ b $ is a full-width gathered block of width equal to a power of $ p $, whose decoration is different from $ 1_{[*]} $, and that satisfies at least one of the following conditions:
\begin{itemize}
\item the decoration of the column is not a $ p^{th} $ power in $ H^*(X; \mathbb{F}_p) $
\item at least one of the constituent rectangles of the column appears a number of times that is not divisible by $ p $
\item the column has a non-trivial solid part (if $ p > 2 $)
\end{itemize}
This set generates $ H^*(Q_0X; \mathbb{F}_p) $ as a graded commutative algebra, with relations given by $ (b_\infty)^{p^{h(b)}} = 0 $, where we define $ h(b) $ as in Theorem \ref{thm:Dung}.
\end{theorem}
\begin{proof}
Since the obvious analogs of Lemma \ref{lem:cup product blocks limit} and Lemma \ref{lem:cup product relations limit} also hold in this case, the statement can be proved with the same argument used for Theorem \ref{thm:Dung}, by replacing $ CX $ with $ Q_0X $ and $ b \odot 1^{[*]} $ with $ b_\infty $.
\end{proof}